\documentclass[11pt]{amsart}
\usepackage[top=1in, bottom=1in, left=1in, right=1in]{geometry}
\usepackage{float}
\usepackage{fancyhdr}
\usepackage{caption}
\usepackage{subcaption}
\usepackage{amsfonts}
\usepackage{amsmath}
\usepackage{amsthm}
\usepackage{amssymb}
\usepackage{bbm}
\numberwithin{equation}{section}
\usepackage{times}

\usepackage[usenames,dvipsnames]{color}
\usepackage{comment}
\usepackage{mathtools}
\usepackage{bbm}
\usepackage{bm}
\usepackage{esvect}
\usepackage{mathtools}
\usepackage{amssymb, amsthm, thmtools, mathrsfs}
\usepackage[backend=biber, sorting=nyt, url=false, doi=false, maxnames=99]{biblatex}
\usepackage[stretch=10, shrink=10]{microtype}
\usepackage{graphicx, xcolor}
\usepackage{colortbl}
\usepackage{hyperref}
\hypersetup{
    colorlinks = true,
linkcolor={black},
urlcolor={blue},
citecolor={blue},    
urlcolor = {blue},
citebordercolor = {0.33 .58 0.33},
 linkbordercolor = {0.99 .28 0.23},
 breaklinks=true}
\usepackage{placeins}

\numberwithin{equation}{section}

\declaretheorem[numberwithin=section]{theorem}

\declaretheorem[name=Conjecture]{conjecture}
\declaretheorem[name=Definition]{definition}

\declaretheorem[numberwithin=subsection, numberlike=theorem]{lemma}
\declaretheorem[numberwithin=subsection, numberlike=theorem]{corollary}
\declaretheorem[numberwithin=subsection, numberlike=theorem]{proposition}

\newtheoremstyle{myremarkstyle} 
  {}                           
  {}                           
  {\normalfont}                
  {}                           
  {\bfseries}                  
  {.}                          
  { }                          
  {}                           

\theoremstyle{myremarkstyle}

\declaretheorem[
    numberwithin=subsection,
    name=Remark
]{remark}

\declaretheorem[
    numberwithin=section,
    name=Remark
]{sectionremark}

\NewDocumentCommand{\diff}{}{\,\mathrm{d}}

\AtBeginDocument{\RenewDocumentCommand{\Re}{}{\operatorname{Re}}}
\AtBeginDocument{\RenewDocumentCommand{\Im}{}{\operatorname{Im}}}
\AtBeginDocument{}
\AtBeginDocument{}
\AtBeginDocument{}
\graphicspath{ {./image/} }

\title{Surfaces Associated to Zeros of Automorphic $L$-functions}

\author{Debmalya Basak, Cruz Castillo and Alexandru Zaharescu}

\address{
Debmalya Basak: Department of Mathematics,
University of Illinois Urbana-Champaign,
Altgeld Hall, 1409 W. Green Street,
Urbana, IL, 61801, USA and Max Planck Institute for Mathematics, Vivatsgasse 7, Bonn, Germany
53111}
\email{basakd@mpim-bonn.mpg.de}

\address{
Cruz Castillo : Department of Mathematics,
University of Illinois Urbana-Champaign,
Altgeld Hall, 1409 W. Green Street,
Urbana, IL, 61801, USA}
\email{ccasti30@illinois.edu}

\address{
Alexandru Zaharescu: Department of Mathematics,
University of Illinois Urbana-Champaign,
Altgeld Hall, 1409 W. Green Street,
Urbana, IL, 61801, USA and Simion Stoilow Institute of Mathematics of the Romanian Academy, 
P. O. Box 1-764, RO-014700 Bucharest, Romania}
\email{zaharesc@illinois.edu}  

\begin{document}
\dedicatory{
Dedicated to Peter Sarnak on the occasion of his seventieth birthday \vspace{-1cm}}
\maketitle
\makeatother
\begin{abstract}
Assuming the Riemann Hypothesis, Montgomery established results concerning pair correlation of zeros of the Riemann zeta function. Rudnick and Sarnak extended these results to automorphic \( L \)-functions and all level correlations. We discover additional geometric structures associated to the zeros of automorphic \( L \)-functions. In the case of pair correlation, these structures form certain surfaces which display Gaussian behavior. For triple correlation, these structures exhibit Laplace and Chi-squared distributions, revealing an unexpected phase transition.  
\end{abstract}
\setcounter{tocdepth}{1}
\tableofcontents
\section{Introduction}\label{sec: Introduction}
\subsection{Motivation}\label{sec 1.1 Motivation}
In his seminal work \cite{montgomeryPairCorrelationZeros1973},
\citeauthor{montgomeryPairCorrelationZeros1973} conducted a detailed study of the pair correlation of non-trivial zeros of the Riemann zeta function $\zeta(s)$. To understand the Fourier transform of the distribution function of the numbers $\gamma-\gamma'$, where $\gamma, \gamma'$ are imaginary parts of the non-trivial zeros of $\zeta(s)$, Montgomery considered the function
\begin{align}\label{Montgomery's Sum 1}
F(\alpha):=F(\alpha, T)=\bigg(\frac{T}{2 \pi} \log T\bigg)^{-1} \sum_{0<\gamma^{\prime} \leqslant T} \sum_{0<\gamma \leqslant T} T^{i \alpha(\gamma-\gamma^{\prime})} w(\gamma-\gamma^{\prime}).
\end{align}
Here $\alpha \in \mathbb{R}$, $T \geqslant2$ and $w(u)$ is a suitable weight function defined on $\mathbb{R}$ by $w(u)=4/(4+u^2)$. Assuming the Riemann Hypothesis (RH), Montgomery showed that
\begin{align}\label{Montgomery's Result}
F(\alpha,T)=(1+o(1)) T^{-2 \alpha} \log T+\alpha+o(1),
\end{align}
as $T$ tends to infinity, uniformly for $0 \leqslant \alpha < 1$.
\par
\citeauthor{rudnickZerosPrincipalLfunctions1996} \cite{rudnickZerosPrincipalLfunctions1996}
vastly generalized Montgomery's work by establishing corresponding results for all higher level correlations of non-trivial zeros of general automorphic $L$-functions. Their work shows that the fluctuations of zeros of an automorphic $L$-function align precisely with the predictions of the GUE model as suggested by Dyson \cite{Dyson}.
\par
In this paper, we are interested in the fine scale distribution of zeros of general automorphic $L$-functions. We discover additional geometric structures in the form of surfaces related to the correlation of these zeros (see Figures \ref{fig:Intro_figure}--\ref{fig:TCS}). Let $m \in \mathbb{N}$ and $\mathscr{A}_m$ be the set of irreducible cuspidal automorphic representations of
$\mathrm{GL}_m$ over $\mathbb{Q}$ with unitary central character. For $\pi \in \mathscr{A}_m$, let $L(s, \pi)$ be its standard $L$-function. For $X,T \geqslant 2$, we consider the sum
\begin{align}
 \label{gXT}
\mathcal{S}_{\textnormal{av},\pi}(X,T)&= \frac{1}{N_{\pi}( T )}
 \sum_{\lvert \Im{\rho_{\pi}'} \rvert \leqslant T} \Re \sum_{\lvert \Im{\rho_{\pi}} \rvert \leqslant T} X^{\rho_{\pi}-\rho_{\pi}'}w(\rho_{\pi}-\rho_{\pi}'),
 \end{align}
where
 \begin{align}
 w &: \mathbb{C} \setminus \{2,-2\} \to \mathbb{C},\quad w(u) := \frac{4}{4 - u^2}, \label{weight} \\
 N_{\pi}( T ) & = \# \left\{ \rho_{\pi} = \beta_{\pi}+i\gamma_{\pi} \colon 0 < \beta_{\pi} < 1, \lvert \gamma_{\pi} \rvert \leqslant T, L(\rho_{\pi},\pi)=0 \right \}, \label{numOfZeros}
\end{align}
and $\rho_{\pi}, \rho_{\pi}'$ run over the non-trivial zeros of $L(s, \pi)$. Here the function $w(u)$ in \eqref{weight} corresponds to the weight function introduced by Montgomery in \eqref{Montgomery's Sum 1}. We shall discuss a more general class of weight functions in Section \ref{sec: general weight functions}. We rewrite $\mathcal{S}_{\textnormal{av},\pi}(X,T)$ as
\begin{align}\label{rewriting G}
 \mathcal{S}_{\textnormal{av},\pi}(X,T)&= \frac{1}{N_{\pi}( T )}
 \sum_{\substack{\lvert \Im{\rho_{\pi}'} \rvert \leqslant T}} \Re \mathcal{S}_{\pi}(X, T, \rho_{\pi}'),
\end{align}
 where
\begin{align}\mathcal{S}_{\pi}(X,T,\rho_{\pi}') &= \sum_{\substack{ \lvert \Im{\rho_{\pi}} \rvert \leqslant T}} X^{\rho_{\pi}-\rho_{\pi}'}w(\rho_{\pi}-\rho_{\pi}'). \label{S definition}
\end{align}
Given a fixed zero $\rho_{\pi}'$, $\mathcal{S}_{\pi}(X,T,\rho_{\pi}')$ encodes information about the gaps $\{(\rho_{\pi}-\rho_{\pi}')\}$ as $\rho_{\pi}$ varies over the neighboring zeros of $\rho_{\pi}'$. More precisely, it can be viewed as an exponential sum whose phases are determined by the spacings $\{(\rho_{\pi}-\rho_{\pi}')\}$ and hence, it captures the local configuration of the zero set near the fixed zero $\rho_{\pi}'$. When $\pi \in \mathscr{A}_1$ is trivial, assuming RH for $\zeta(s)$, one can also view $S_{\pi}(X,T,\rho_{\pi}')$ as the contribution to $F(\alpha)$ from the inner sum in \eqref{Montgomery's Sum 1} corresponding to the zero $\rho_{\pi}'$ when $X=T^{\alpha}$.

In the present paper, we study the behavior of $S_{\pi}(X,T,\rho_{\pi}')$ as $\rho'_{\pi}$ varies over the non-trivial zeros of $L(s,\pi)$. Surprisingly, numerical evidence, as shown in \autoref{fig:Intro_figure}, suggests that in the case of $\zeta(s)$, the associated surface appears to reveal Gaussian behavior led along Montgomery's function $F(\alpha)$. Our first aim in this paper is to study the precise nature of this surface for the Riemann zeta function. Our second goal is to extend the results in the broader context of automorphic $L$-functions. Our third objective is to investigate whether there is any change in these phenomena as one transitions from pairs to triples of zeros of automorphic $L$-functions. For detailed numerical and graphical evidence in these directions, interested readers are referred to the website linked in \cite{github}.
\subsection{A Brief Survey of Montgomery's Pair Correlation Conjecture} The function $F(\alpha)$ has been the subject of extensive research over the years. Besides proving \eqref{Montgomery's Result}, in \cite{montgomeryPairCorrelationZeros1973}, Montgomery also proposed the following conjecture concerning the behavior of $F(\alpha)$ for $\alpha>1.$
\begin{conjecture} \label{Montgomery Conj 1}
Let $F(\alpha)$ be as defined in \eqref{Montgomery's Sum 1}. Then for any fixed constants $1 \leqslant a<b<\infty$, we have 
\begin{align}\label{Montgomery > 1}
 F(\alpha, T)=1+o(1) \quad \text {as } \quad T \to \infty,
\end{align}
uniformly for $\alpha \in[a, b]$.
\end{conjecture} 
For variants of Conjecture \ref{Montgomery Conj 1} pertaining to more general $L$-functions, see Murty--Zaharescu \cite[Eq. (2.4)]{murtyExplicitFormulasPair2002}. Based on Conjecture \ref{Montgomery Conj 1}, Montgomery further conjectured the following distribution for the differences between the imaginary parts of non-trivial zeros of the Riemann zeta function.
\begin{conjecture}\textnormal{(Pair Correlation Conjecture).} \label{Montgomery Conj 2}
For any fixed $\alpha<\beta<\infty$,
\begin{align}\label{PC}
 \sum_{\substack{ 0 < \gamma \leqslant T \\ 0 < \gamma' \leqslant T \\ 2 \pi \alpha(\log T)^{-1} \leqslant \gamma-\gamma^{\prime} \leqslant 2 \pi \beta(\log T)^{-1} }} 1 \sim \bigg (\int_\alpha^\beta\bigg(1-\bigg(\frac{\sin \pi u}{\pi u}\bigg)^2\bigg) \diff u +\delta(\alpha,\beta) \bigg ) \frac{T}{2 \pi} \log T,
\end{align}
as $T$ tends to $\infty$. Here $\gamma, \gamma'$ vary over the imaginary parts of non-trivial zeros of $\zeta(s)$ and $\delta(\alpha,\beta)=1$ if $0\in [\alpha,\beta]$, $\delta(\alpha,\beta)=0$ otherwise.
\end{conjecture} 
Conjecture \ref{Montgomery Conj 2}, known as Montgomery's Pair Correlation Conjecture, implies that small gaps between the zeros of $\zeta(s)$ occur very rarely, which is surprising. Indeed, the pair correlation for many other number theoretic quantities differ from \eqref{PC}.
For example, although the sequence of zeros of $\zeta(s)$ and the sequence of prime numbers are related by explicit formulas, their distributions are very different. More precisely, assuming an appropriate version of the prime $k$-tuple conjecture,
\citeauthor{gallagherDistributionPrimesShort1976} \cite{gallagherDistributionPrimesShort1976}
showed that the sequence of primes in short intervals exhibits Poisson behavior.
\par
Dyson \cite{montgomeryPairCorrelationZeros1973} made the significant observation that the Gaussian Unitary Ensemble (GUE) from Random Matrix Theory has the same pair correlation function $1-(\frac{\sin \pi u}{\pi u})^2$
as in the case of zeros of $\zeta(s)$. In mathematical physics, GUE models the distribution of energy levels in systems comprising numerous particles. When suitably normalized, the limiting pair correlation function of  eigenvalues of matrices in the Gaussian Unitary Ensemble becomes $1-(\frac{\sin \pi u}{\pi u})^2$. For more results related to Montgomery's Pair Correlation Conjecture, the reader is referred to the works of Gallagher \cite{gallagher1985pair}, Goldston \cite{goldston1985prime}, Goldston--Montgomery \cite{ goldstonPairCorrelationZeros1987}, Goldston--Heath-Brown \cite{HB-Goldston}, Heath-Brown \cite{Heath1982}, Montgomery--Odlyzko \cite{Mont-Odlyzko} and Odlyzko \cite{Odlyzko, Odlyzko2}. For recent works, see Baluyot \cite{Baluyot2016}, Baluyot--Goldston--Suriajaya--Turnage-Butterbaugh \cite{BGST2024}, Bui--Milinovich \cite{BM2018}, Goldston--Suriajaya \cite{GS2022} and the references therein. 
\begin{figure}
 \centering
 \includegraphics[width=0.5\linewidth]{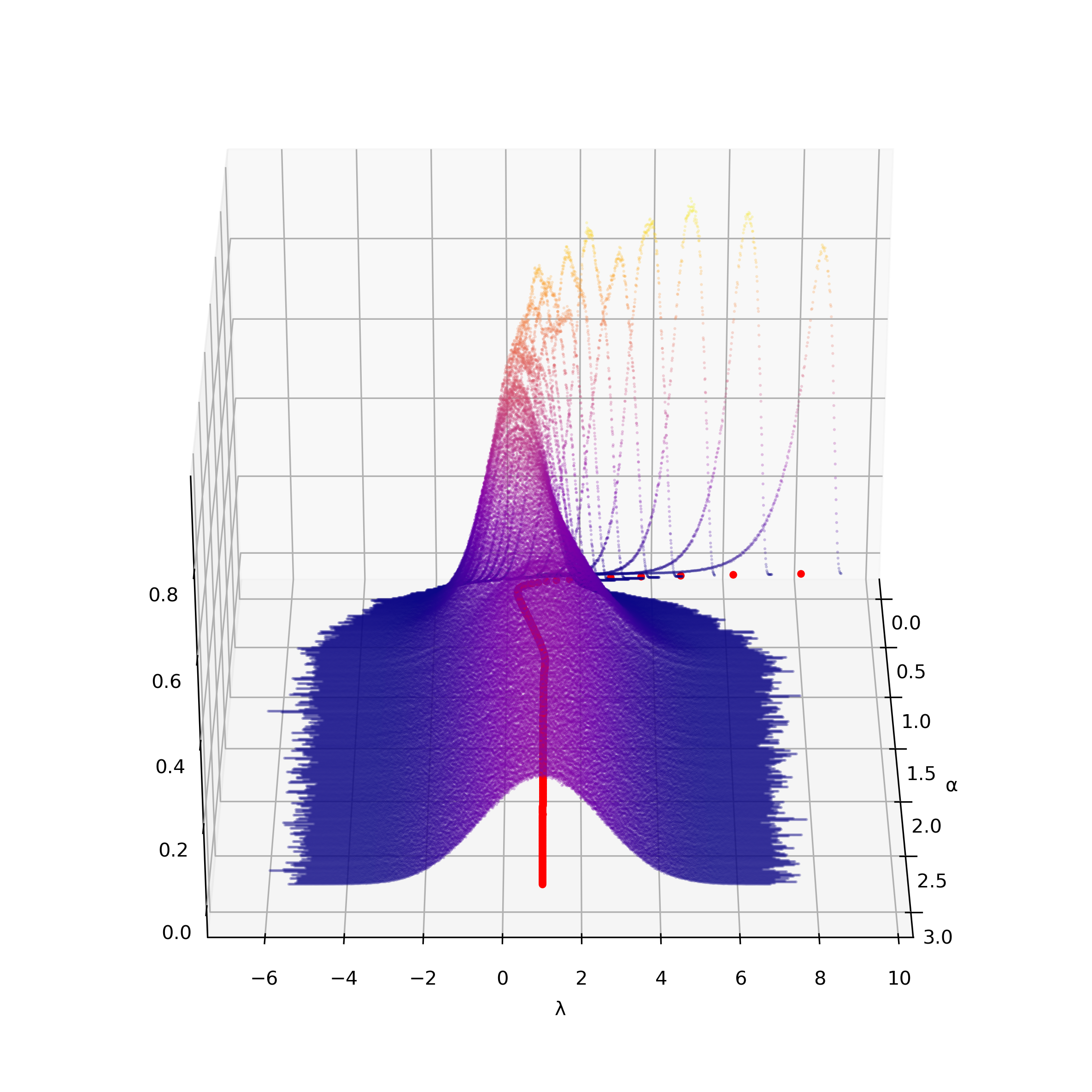}
 \caption{A surface plot for the probability density function associated to the set of values $\{\operatorname{Re}S_{\pi}( X,T,\rho_n) \colon 1 \leqslant n \leqslant N\}$, where $L(s,\pi) = \zeta(s), N=10^6$, $X=T^{\alpha}$ and $\alpha$ varies between $(0,3]$ in a discrete equi-spaced manner. The surface (in blue) appears to exhibit Gaussian behavior while shifting along Montgomery's function $F(\alpha)$ (in red).}
 \label{fig:Intro_figure}
\end{figure}
\par
\citeauthor{rudnickZerosPrincipalLfunctions1996} \cite{rudnickZerosPrincipalLfunctions1996} 
established the higher level correlations of non-trivial zeros of any given automorphic $L$-function. Their work shows that the fluctuations of the zeros of an automorphic $L$-function are universal and independent of the distribution of its coefficients. Furthermore, these fluctuations align precisely with the predictions of the GUE model as suggested by Dyson \cite{Dyson}. An essential ingredient in the work of Rudnick--Sarnak is the following technical hypothesis involving the coefficients of $L(s,\pi)$. Let
$$
\frac{L^{\prime}}{L}(s, \pi)=-\sum_{n=1}^{\infty} \frac{\Lambda_{\pi}(n)}{n^s}, \quad \Re(s)>1,
$$
where $\Lambda_{\pi}(n)$ is supported only on prime powers and is given by \eqref{Lambda Def}.\\

\noindent Hypothesis \textbf{H}$_{\pi}$ : Let $m \in \mathbb{N}$ and $\pi \in \mathscr{A}_{m}.$ For any fixed $k \geqslant 2,$
\begin{align}\label{Eq: Hypothesis H}
    \sum_{p \textnormal{ prime}}\frac{\lvert \Lambda_{\pi}(p^k) \rvert^2}{p^k} < \infty.
\end{align}
The Ramanujan conjecture for cusp forms on \(\mathrm{GL}_m\) suggests that \(|\Lambda_{\pi}(p^k)| \leqslant m \log p\), which more than adequately satisfies the requirements for Hypothesis \textbf{H}$_{\pi}$. Hypothesis \textbf{H}$_{\pi}$ was proved for $m\leqslant 4$ due to Rudnick--Sarnak \cite{rudnickZerosPrincipalLfunctions1996} and Kim \cite{kimgl4}, and in some cases when $m = 5,6$ due to Kim \cite{kim2003functoriality}, Kim--Shahidi \cite{KimShahidi} and Wu--Ye \cite{WuYe}. Recently, Hypothesis \textbf{H}$_{\pi}$ was proved in full generality for any $m \in \mathbb{N}$ by Jiang \cite{jiang2025}.
\par
Although the distribution of the coefficients of \( L(s,\pi) \) is not universal, the theory of Rankin–Selberg \( L \)-functions, together with Hypothesis \textbf{H}$_{\pi}$ ensures that the asymptotic relation 
\begin{align}\label{Key Asymptotic Relation}
    \sum_{n \leqslant X} \frac{\lvert \Lambda_{\pi}(n) \rvert^2}{n} \sim \frac{\log^2 X}{2}
\end{align}
holds regardless of the choice of $\pi$. This shared behavior contributes to the universality of the \( n \)-level correlations with respect to $\pi$ in \cite{rudnickZerosPrincipalLfunctions1996}. In this context, we also highlight the works of Katz--Sarnak \cite{katzRandomMatricesFrobenius1999,katzZeroesZetaFunctions1999} and Iwaniec--Luo--Sarnak \cite{Iwaniec-Luo-Sarnak} on the universality of distributions of low-lying zeros, that is, zeros close to the central point. Following their works, it is expected that the distribution of the low-lying zeros is predicted by 
only a few random matrix ensembles. For more on these topics, the reader is encouraged to consult
the works of Heath-Brown \cite{Heath-Brown2}, Keating--Snaith \cite{Keating-Snaith}, Iwaniec--Sarnak \cite{Iwaniec--Sarnak}, Mehta \cite{Mehta}, Young \cite{Young} and the references therein.
\subsection{A Pair Correlation Surface}\label{PCS Introduction Subsection}
In addition to Figures \ref{fig:Intro_figure}--\ref{fig:TCS}, we present further numerical evidence in \cite{github} related to the behavior of $S_{\pi}(X,T, \rho_{\pi}')$. Based on these computations and graphical evidence, we propose the following conjectures on the statistical distribution of $S_{\pi}(X,T, \rho_{\pi}')$. To proceed, we require some notation. Let $m \in \mathbb{N}$ and $\mathscr{A}_m$ be the set of irreducible cuspidal automorphic representations of
\(\mathrm{GL}_m\) over $\mathbb{Q}$ with unitary central character. For $\pi \in \mathscr{A}_m$, let $L(s, \pi)$ be its standard $L$-function. Let $X,T \geqslant 2$ and $\rho_{\pi}'$ be a non-trivial zero of $L(s, \pi)$. Define 
\begin{equation} \label{Defining Q}
\mathcal{Q}_{\pi}(X,T,\rho_{\pi}') : =
 \frac{\Re \mathcal{S}_{\pi}(X, T, \rho_{\pi}')-(m \log T)^{-1}(\min \{ \log X, m\log T\}) 
 }{(1/2) (\min \{ \log X, m\log T\})^{\frac{1}{2}}},
\end{equation}
where $\mathcal{S}_{\pi}(X,T,\rho'_{\pi})$ is defined by \eqref{S definition}. 
\begin{conjecture}
\label{conj:one}
Let $m \in \mathbb{N}$ and $\pi \in \mathscr{A}_m$. For any $\alpha \in \mathbb{R}\setminus \{0\}$ and any \( \lambda \in \mathbb{R} \),
 \begin{equation} \label{conjecture Equation}
  \lim_{T \to \infty} \frac{1}{N_{\pi}( T)}
 \# \left\{ \lvert \Im(\rho_{\pi}') \rvert \leqslant T \colon \mathcal{Q}_{\pi}(T^{\lvert \alpha \rvert m},T,\rho_{\pi}')
   < \lambda \right\} = \frac{1}{\sqrt{2 \pi}} \int_{-\infty}^{\lambda} e^{- \frac{1}{2}u^{2}} \diff u,
 \end{equation}
where $N_{\pi}(T)$ is as in \eqref{numOfZeros}, $\rho_{\pi}'$ runs over non-trivial zeros of $L(s,\pi)$ and $\mathcal{Q}_{\pi}(X,T,\rho_{\pi}')$ is defined by \eqref{Defining Q}.
\end{conjecture}

We now formulate the following definition of the Pair Correlation Surface (PCS) (if it exists). We define it pointwise as follows.
\begin{definition}\textnormal{(Definition of the Pair Correlation Surface).}
Let $m \in \mathbb{N}$ and $\pi \in \mathscr{A}_m$. Fix $\alpha \in \mathbb{R}\setminus \{0\}$, $\lambda \in \mathbb{R}$ and let
\begin{align}\label{g defn 1}
g_{\pi}(\alpha, \lambda) : = \lim_{\delta \to 0} \frac{1}{2 \delta}
 \lim_{T \to \infty} \frac{1}{N_{\pi}( T )}
 \# \left\{ \lvert \Im(\rho_{\pi}') \rvert \leqslant T \colon
\mathcal{Q}_{\pi}(T^{\lvert \alpha \rvert m},T,\rho_{\pi}') \in [\lambda -\delta, \lambda+\delta] \right\},
\end{align}
provided that the limit exists. Let $E_1$ be the set of points $(\alpha, \lambda) \in (\mathbb{R}\setminus \{0\} ) \times \mathbb{R}$ where $g_{\pi}(\alpha, \lambda)$ is defined. Let $E_0$ be the set of points $(0,\lambda) \in \{0\} \times \mathbb{R}$ for which the limit
\begin{align}\label{g defn 2}
g_{\pi}(0,\lambda) := \lim_{\substack{\alpha \to 0 \\ \alpha \neq 0 \\ (\alpha, \lambda) \in E_1 }} g_{\pi}(\alpha, \lambda)
\end{align}
exists. Let $E= E_0 \cup E_1$ and consider the function $g_{\pi}$ defined on $E$ by \eqref{g defn 1} and \eqref{g defn 2} above. We define the graph of the function $g_{\pi} : E \to \mathbb{R}$ to be the Pair Correlation Surface associated to the zeros of $L(s, \pi)$.
\end{definition}
\begin{figure}[t]
    \centering
    {{\includegraphics[width = .35\textwidth]{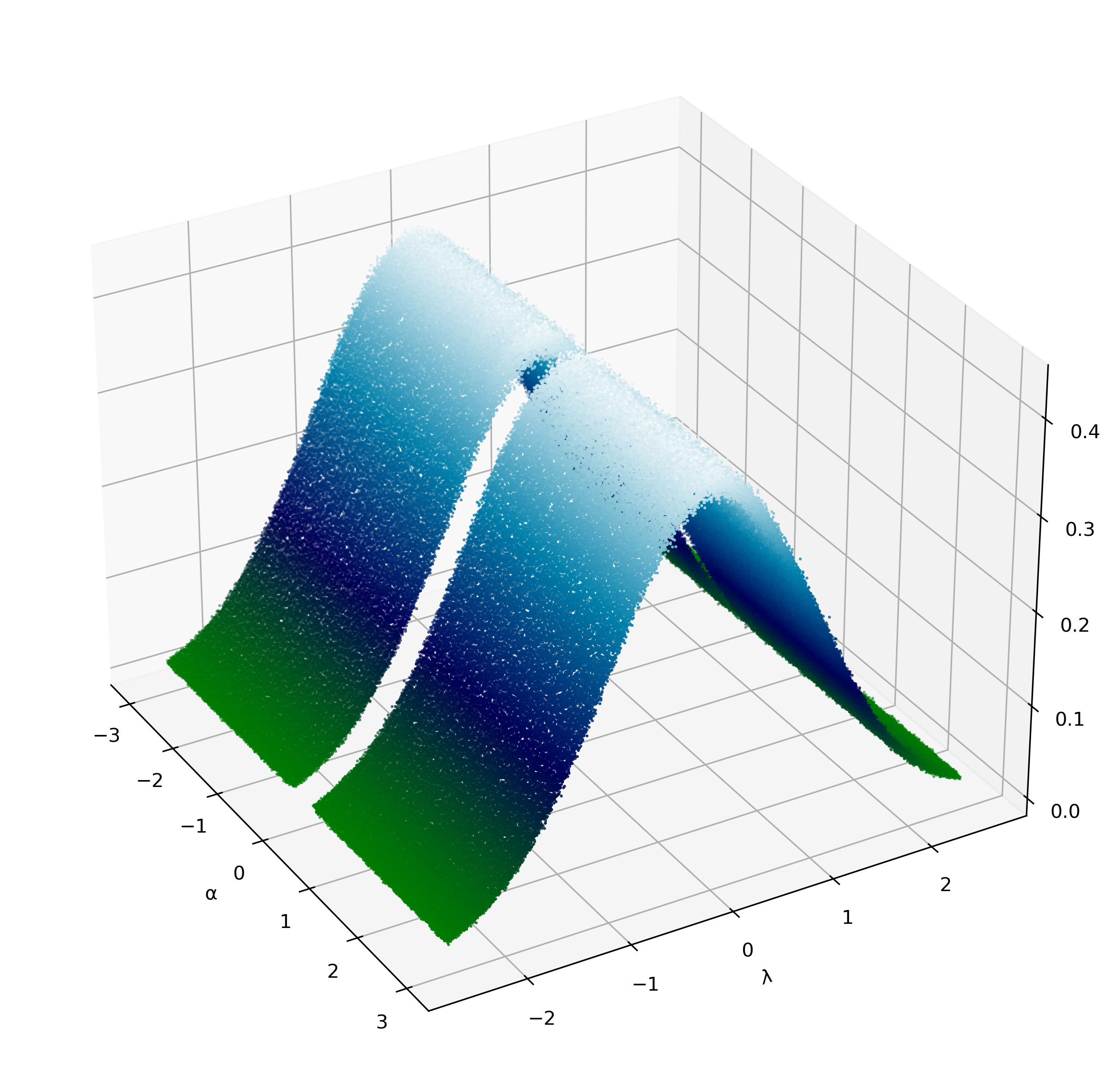}}}
     {{\includegraphics[width=.35\textwidth]{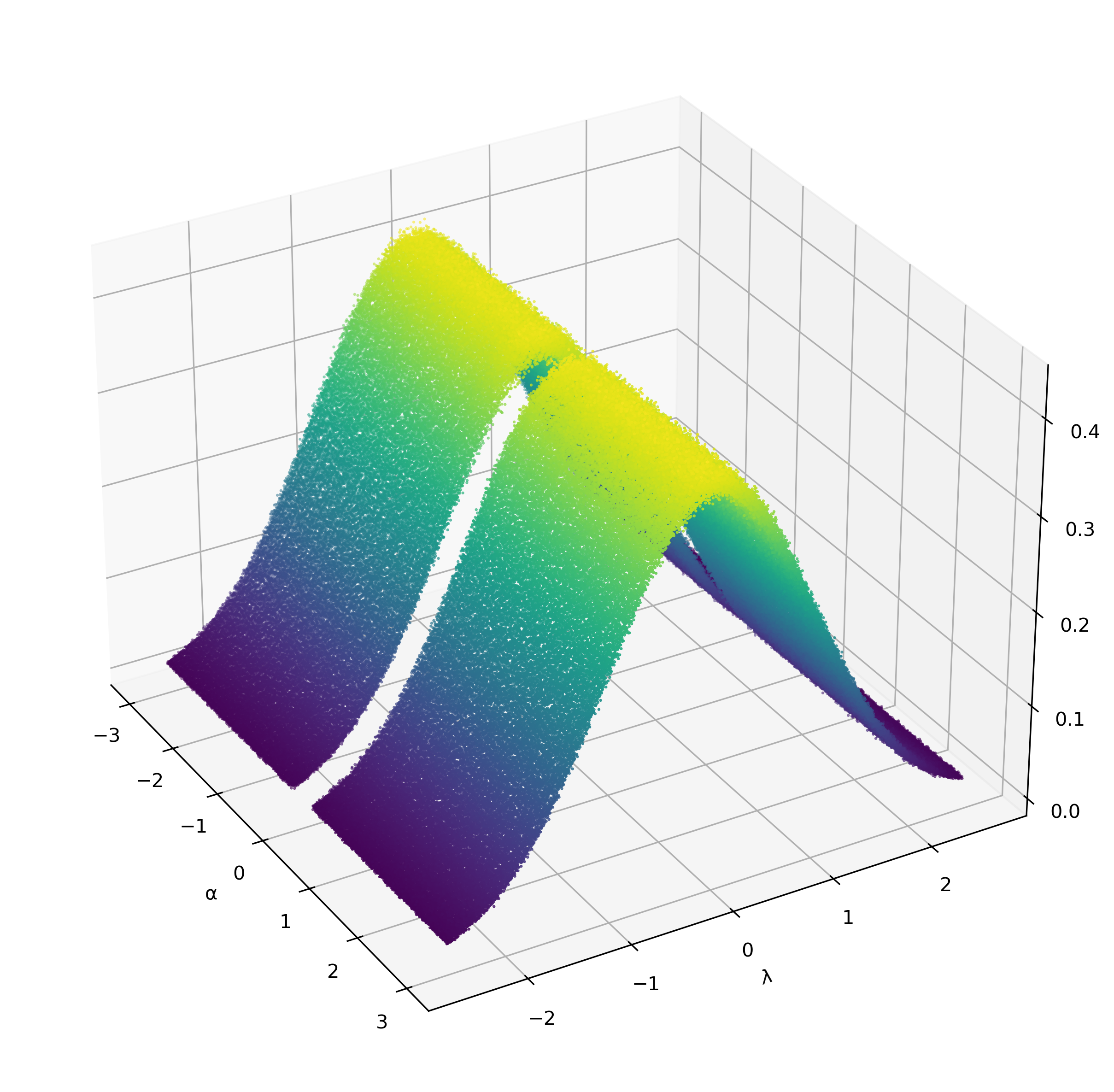} }}%
    \caption{\underline{Left:} A discrete approximation of the PCS using the zeros of $L(s,\Delta)$, the $L$-function associated with the modular discriminant,  with $|\alpha| \geqslant 0.2.$ \underline{Right:} A discrete approximation of the PCS using the zeros of $L(s,E)$ where $E/\mathbb{Q}$ is the elliptic curve over $\mathbb{Q}$ defined by $ E: y^2 +y = x^3-x$, with $|\alpha| \geqslant 0.2.$}
    \label{fig:PCS_GL2}
\end{figure}
We make the following conjecture on the existence of the Pair Correlation Surface.
\begin{conjecture}
 \label{conj:two}
 \textnormal{(Existence of the Pair Correlation Surface).}
Let $m \in \mathbb{N}$ and $\pi \in \mathscr{A}_m$. For any $\alpha \in \mathbb{R}\setminus \{0\}$ and any \( \lambda \in \mathbb{R} \), we have
\begin{equation}\label{PCS Limit Definition}
 \lim_{\delta \to 0} \frac{1}{2 \delta}
 \lim_{T \to \infty} \frac{1}{N_{\pi}( T )}
 \# \left\{ \lvert \Im(\rho_{\pi}')\rvert  < T \colon
\mathcal{Q}_{\pi}(T^{\lvert \alpha \rvert m},T,\rho_{\pi}') \in [\lambda -\delta, \lambda+\delta] \right\} = \frac{1}{\sqrt{2 \pi}}e^{-\frac{1}{2}\lambda^2},
\end{equation}
where $N_{\pi}(T)$ is given by \eqref{numOfZeros}, $\rho_{\pi}'$ runs over non-trivial zeros of $L(s,\pi)$ and $\mathcal{Q}_{\pi}(X,T,\rho_{\pi}')$ is defined by \eqref{Defining Q}. Moreover, the Pair Correlation Surface associated to the zeros of $L(s, \pi)$ is the graph of the function $g_{\pi} : \mathbb{R}^2 \to \mathbb{R}$ defined by
\begin{align}\label{PCS Equation}
g_{\pi}(\alpha, \lambda) = \frac{1}{\sqrt{2 \pi}} e^{- \frac{1}{2}\lambda^2}.
\end{align}
\end{conjecture}
We now present our main results. We focus on the case when $\alpha>0$. The case when $\alpha<0$ can be treated similarly. One can see that \autoref{conj:two} follows from \autoref{conj:one}. In the direction of \autoref{conj:one}, we prove the following result. 

\begin{theorem}\label{main thm}
Let $m \in \mathbb{N}$ and $\pi \in \mathscr{A}_m$. Assume the Riemann Hypothesis for $L(s,\pi)$. Let $\{\alpha_j \}_{j \in \mathbb{N}}$ be any sequence of
 strictly positive real numbers decreasing to $0$.
 Then there exists an increasing sequence $\{V_j\}_{j \in \mathbb{N}}$
 tending to $\infty$ with the following property.
 For any sequence $\{T_j\}_{j \in \mathbb{N}}$ such that
 $T_j \geqslant V_j$ for all $j \in \mathbb{N}$, we have for any $\lambda \in \mathbb{R}$,
\begin{align}\label{Theorem Equation}
  \lim_{j \to \infty} \frac{1}{N_{\pi}( T_j)}
 \# \left\{ \lvert \Im(\rho_{\pi}') \rvert \leqslant T_j \colon \mathcal{Q}_{\pi}(T_j^{\alpha_j m},T_j,\rho_{\pi}')
   < \lambda \right\}  = \frac{1}{\sqrt{2 \pi}} \int_{-\infty}^{\lambda} e^{- \frac{1}{2}u^{2}} \diff u,   \end{align}
where $N_{\pi}(T)$ is as in \eqref{numOfZeros}, $\rho_{\pi}'$ runs over non-trivial zeros of $L(s,\pi)$ and $\mathcal{Q}_{\pi}(X,T,\rho_{\pi}')$ is defined by \eqref{Defining Q}.
\end{theorem}
In order to prove Theorem \ref{main thm}, we utilize the method of moments. For $r \in \mathbb{N}$, we compute the moments 
\begin{align}\label{Moments Definition}
\mathcal{M}_{\pi, r}(X,T) := \frac{1}{N_{\pi}( T )}\sum_{\substack{ \lvert \Im \rho_{\pi}' \rvert  \leqslant T}} \left(\Re \mathcal{S}_{\pi}(X, T, \rho_{\pi}') -\frac{\min \{\log X, m \log T \}}{m \log T}\right)^{r}.
\end{align}
These moments will furnish the distribution of $ \mathcal{S}_{\pi}(X, T, \rho_{\pi}')$. For $r \in \mathbb{N}$, let
\begin{align}\label{Defining Normal Distribution Coefficients}
\boldsymbol{\mu}_r:=\begin{cases}
    1 \cdot 3 \cdots (r-1) & \text{if $r$ is even,}\\
    0 &\text{if $r$ is odd.}
    \end{cases} 
\end{align}
\begin{theorem}
\label{thm:twomMomentConditional}
Let $m \in \mathbb{N}, \pi \in \mathscr{A}_m$ and $\theta_{m} \in [0, \frac{1}{2}-\frac{1}{m^2+1}]$ be an admissible exponent towards the Ramanujan conjecture for $L(s,\pi)$. Assume the Riemann Hypothesis for $L(s,\pi)$. Let $X,T \geqslant 2, r \in \mathbb{N}$ and $\mathcal{M}_{\pi, r}(X,T)$ be as defined in \eqref{Moments Definition}. Fix $\alpha \in \mathbb{R}$ such that 
\[0< \alpha < \frac{1}{mr(1+\frac{4}{3}\theta_m)}.
\]
If $r$ is even, then
\[
\mathcal{M}_{\pi,r}(T^{\alpha m},T) = \boldsymbol{\mu}_{r}\left ( \frac{\alpha m \log T}{4} \right)^{\frac{r}{2}} + O_{\pi, r, \alpha}((\log T)^{\frac{r-1}{2}}),
\]
where $\boldsymbol{\mu}_r$ is given by \eqref{Defining Normal Distribution Coefficients}. If $r$ is odd, then
\[
\mathcal{M}_{\pi, r}(T^{\alpha m},T) \ll_{\pi,r,\alpha} (\log T)^{\frac{r-1}{2}}.
\]
\end{theorem}
We expect the results in Theorem \ref{thm:twomMomentConditional} to hold for all fixed $\alpha \in \mathbb{R}\setminus \{0\}$. This leads us to make the following conjecture, which is evidently stronger than Conjectures \ref{conj:one} and \ref{conj:two}.
\begin{conjecture}\label{conj:gaussian}
Let $m \in \mathbb{N}$ and $\pi \in \mathscr{A}_m$. Let $X,T \geqslant 2, r \in \mathbb{N}$ and $\mathcal{M}_{\pi, r}(X,T)$ be as defined in \eqref{Moments Definition}. Fix $\alpha \in \mathbb{R} \setminus \{0 \}$. If $r$ is even, then
\begin{align*}
\mathcal{M}_{\pi, r}(T^{\lvert \alpha \rvert m },T) = (\boldsymbol{\mu}_{r}+o(1))\left ( \frac{\min\{\lvert \alpha \rvert,1\} m \log T}{4} \right)^{\frac{r}{2}}, \quad \textnormal{as} \quad T \to \infty,
\end{align*}
where $\boldsymbol{\mu}_r$ is defined by \eqref{Defining Normal Distribution Coefficients}. If $r$ is odd, then 
\begin{align*}
\mathcal{M}_{\pi, r}(T^{\lvert \alpha \rvert m },T) =o( {(\log T)^{\frac{r}{2}}}) \quad \textnormal{as} \quad T \to \infty.
\end{align*}
\end{conjecture}
We make the following remarks.
\begin{remark}
In some sense, \autoref{main thm} is as close as possible to \autoref{conj:one},
 while failing to prove it for any single fixed $\alpha$.
 We do need in \autoref{main thm} to let  $\{\alpha_j \}_{j \in \mathbb{N}}$ tend to $0$. But we can do so as slowly as we want. 
\end{remark}
\begin{remark}
The universality in the shape of the Pair Correlation Surface with respect to $\pi$, as seen from \eqref{PCS Equation} and Theorem \ref{main thm}, is unexpected at first sight, but admits a natural explanation. While studying the pair correlation of zeros of \(L\)-functions, one typically uses an explicit formula to connect the zeros to primes, leading us to examine moments of sums of the form 
\[ \sum_{n \leqslant X} \frac{a_n \Lambda_{\pi}(n)}{n^{\rho_{\pi}}}, \quad \textnormal{where} \quad \frac{L'}{L}(s,\pi) = \sum_{n=1}^{\infty} \frac{\Lambda_{\pi}(n)}{n^s}.
\]
Here the sequence $\{a_n\}$ corresponds to Montgomery's weight function $w(u)$, or to more general weight functions. Importantly, if we write $\Lambda_{\pi}(p)= \lambda_{\pi}(p) \log p$ then the distribution of the coefficients \(\lambda_{\pi}(p)\) is not universal across different choices of \(\pi\). If $\pi \in \mathscr{A}_2$, then there are two conjectured limiting distributions for the \(\lambda_{\pi}(p)\)'s: Sato--Tate or a uniform distribution with a Dirac mass term, see Serre \cite{Serre}. As the degree increases, the number of possible limiting distributions increases. However, the Rankin-Selberg \(L\)-function theory shows that all these distributions share the same second moment, see Lemma \ref{Rudnick-Sarnak Lemma}. This common second moment is key to the universality in the \(n\)-level correlations, as shown by Rudnick--Sarnak \cite{rudnickZerosPrincipalLfunctions1996} as well as in the shape of the Pair Correlation Surface.
\end{remark}

\begin{remark}\label{remark:small_alpha}
The restriction for the range of $\alpha$ in Theorem \ref{thm:twomMomentConditional} naturally arises in the calculation of the $r$-th moments and reflects upon Montgomery's result, in which case, the asymptotic \eqref{Montgomery's Result} holds for $0\leqslant \alpha < 1$. This restriction may be surpassed if one considers averaging over suitable families of $L$-functions. For instance, by averaging over certain families of Dirichlet $L$-functions and $\mathrm{GL}_2$ $L$-functions, analogous versions of \eqref{Montgomery's Result} have been established for extended ranges of $\alpha$, see the works of \"Ozl\"uk \cite{ozluk1996theq}, Chandee--Lee--Liu--Radziwi\l{}{\l} \cite{chandee2014simple}, and Chandee--Klinger-Logan--Li \cite{chandee2022pair}. It would be interesting to investigate whether the Pair Correlation Surface exhibits distinctive features associated with the symmetry type of the family of $L$-functions considered.
\end{remark}

\begin{remark}
As mentioned in Section \ref{sec 1.1 Motivation}, $\mathcal{S}_{\pi}(X, T, \rho_{\pi}')$ can be viewed as an exponential sum involving the differences $\rho_{\pi}-\rho_{\pi}'$ as $\rho_{\pi}$ varies over the non-trivial zeros of $L(s,\pi)$. Assuming RH for $L(s,\pi)$ and taking into account the attached weight $w(\rho_{\pi}-\rho_{\pi}')$ in \eqref{S definition}, a trivial bound for $\Re \mathcal{S}_{\pi}(X, T, \rho_{\pi}')$ is $O_{\pi}(\log T)$. Theorem \ref{main thm} shows that for most $\rho_{\pi}'$, $\Re \mathcal{S}_{\pi}(X, T, \rho_{\pi}')$ once normalized exhibit square-root cancellation. Moreover, they also satisfy an underlying Gaussian distribution. 
\end{remark}
\par
A significant observation is that Theorem \ref{main thm} does not entirely address the uniformity in \(\alpha\) with respect to \(T\). Specifically, in \eqref{Defining Q} if $X$ is small compared to $T$, say $X \asymp (\log T)^{\alpha m}$ for some fixed $0<\alpha<1/(2m)$ then as we shall see later, the asymptotic estimate for $\mathcal{S}_{\textnormal{av},\pi}(X,T)$ is influenced by a term of order $X^{-2} \log T$. This mirrors the asymptotic relation \eqref{Montgomery's Result} obtained by Montgomery for \(\zeta(s)\). This transition in the behavior of $\mathcal{S}_{\textnormal{av},\pi}(X,T)$ accounts for the distortion observed in the distribution of $S_{\pi}(X,T, \rho_{\pi}')$ for small ranges of $X$, as seen in \autoref{fig:Intro_figure} and also in \cite{github}. To explore this in more detail, we study a second-order distribution of $S_{\pi}(X,T, \rho_{\pi}')$ when \(X\) is of size \((\log T)^\alpha\). We consider the shifted sums 
\begin{align}
\widetilde{\mathcal{S}}_{\pi}(X,T, \rho'_{\pi}) &:= \mathcal{S}_{\pi}(X,T, \rho'_{\pi}) +X^{-2} \frac{L'}{L}( \rho'_{\pi} - 2, \pi) \label{Shifted S}, \\
\Re \widetilde{\mathcal{S}}_{\text{av}, \pi}(X,T) &:= \frac{1}{N_{\pi}( T )}  \sum_{\substack{\lvert \Im{\rho_{\pi}'} \rvert \leqslant T}} \Re \widetilde{\mathcal{S}}_{\pi}(X,T, \rho'_{\pi}) \label{Shifted Mean}
\end{align}
and define 
\begin{equation}\label{Defining Q Shifted}
\widetilde{\mathcal{Q}_{\pi}}(X,T,\rho'_{\pi}) : =\frac{\Re \widetilde{\mathcal{S}_{\pi}}(X, T, \rho'_{\pi})-(m \log T)^{-1}(\min \{ \log X, m\log T\})
 }{(1/2) (\min \{ \log X, m\log T\})^{\frac{1}{2}}}.
\end{equation}
For small ranges of $X$, we can relax the assumption of RH for $L(s,\pi)$ and instead adopt a weaker hypothesis regarding the density of zeros of $L(s,\pi)$ near the critical line.\\

\noindent \textbf{Hypothesis \textbf{Z$_{\pi}$}}. Let $m \in \mathbb{N}$ and $\pi \in \mathscr{A}_m$. There exist constants $A_{\pi}>0, B_{\pi}>0$ (depending on $\pi$) such that
\[
N_{\pi}(\sigma, T) =\#\left\{\rho_{\pi}= \beta_{\pi}+i \gamma_{\pi}: \sigma \leqslant \beta_{\pi}, \lvert \gamma_{\pi} \rvert \leqslant T, L(\rho_{\pi}, \pi) =0 \right \} \ll_{\pi} T^{1-A_{\pi}(\sigma-\frac{1}{2})} \log^{B_{\pi}} T,
\]
uniformly for $\sigma \geqslant \frac{1}{2}$ and $T \geqslant 2$.\\

\begin{remark}\label{Zero Density Remark}
Hypothesis \textbf{Z$_{\pi}$} is known for the Riemann zeta function and Dirichlet $L$-functions due to Selberg \cite{MR0022872, MR0020594} and for any $\pi \in \mathscr{A}_{2}$ due to Luo \cite{Luo} and Beckwith--Liu--Thorner--Zaharescu \cite{MR4574646}.
\end{remark} 
Assuming Hypothesis \textbf{Z$_{\pi}$}, our main result in this direction is as follows.
\begin{theorem}\label{main thm 2}
Let $m \in \mathbb{N}$ and $\pi \in \mathscr{A}_m$. Assume Hypothesis \textnormal{\textbf{Z$_{\pi}$}} for $L(s,\pi)$ with $B_{\pi}<3$. Let $\{\alpha_j \}_{j \in \mathbb{N}}$ be any sequence of
 strictly positive real numbers decreasing to $0$.
 Then there exists an increasing sequence $\{V_j\}_{j \in \mathbb{N}}$
 tending to $\infty$ with the following property.
 For any sequence $\{T_j\}_{j \in \mathbb{N}}$ such that
 $T_j \geqslant V_j$ for all $j \in \mathbb{N}$, we have for all $\lambda \in \mathbb{R}$,
\begin{align}\label{Theorem Equation}
  \lim_{j \to \infty} \frac{1}{N_{\pi}( T_j)}
 \# \left\{ \lvert \Im(\rho_{\pi}') \rvert \leqslant T_j \colon \mathcal{\widetilde{Q}}_{\pi}((\log T_j)^{\alpha_j m},T_j,\rho_{\pi}')
   < \lambda \right\}  = \frac{1}{\sqrt{2 \pi}} \int_{-\infty}^{\lambda} e^{- \frac{1}{2}u^{2}} \diff u,   \end{align}
where $N_{\pi}(T)$ is as in \eqref{numOfZeros}, $\rho_{\pi}'$ runs over non-trivial zeros of $L(s,\pi)$ and $\mathcal{\widetilde{Q}}_{\pi}(X,T,\rho_{\pi}')$ is defined by \eqref{Defining Q Shifted}.
\end{theorem}
To prove Theorem \ref{main thm 2}, we establish asymptotic estimates for the moments
\begin{align}\label{Shifted Moments Definition}
\mathcal{\widetilde{M}}_{\pi, r}(X,T) := \frac{1}{N_{\pi}( T )}\sum_{\substack{ \lvert \Im \rho_{\pi}' \rvert  \leqslant T}} \left(\Re \mathcal{\widetilde{S}}_{\pi}(X, T, \rho_{\pi}') - \frac{\min \{\log X, m \log T \}}{m \log T}\right)^{r}.
\end{align}
\begin{theorem}\label{thm:twomMomentUnconditional}
Let $m \in \mathbb{N}$ and $\pi \in \mathscr{A}_m$. Assume Hypothesis \textnormal{\textbf{Z$_{\pi}$}} for $L(s,\pi)$ with $B_{\pi}<3$. Let $X,T \geqslant 3, r \in \mathbb{N}$ and $\mathcal{\widetilde{M}}_{\pi, r}(X,T)$ be as defined in \eqref{Shifted Moments Definition}. Fix $\alpha \in \mathbb{R}$ such that $0<r\alpha m  <2(3-B_{\pi})$. If $r$ is even, then
\[
\mathcal{\widetilde{M}}_{\pi,r}((\log T)^{\alpha m},T) = \boldsymbol{\mu}_{r}\left ( \frac{\alpha m \log \log T}{4} \right)^{\frac{r}{2}} + O_{\pi,r, \alpha}((\log \log T)^{\frac{r-1}{2}}),
\]
where $\boldsymbol{\mu}_{r}$ is given by \eqref{Defining Normal Distribution Coefficients}. If $r$ is odd, then
\[
\widetilde{\mathcal{M}}_{\pi, r}((\log T)^{\alpha m},T) \ll_{\pi,r,\alpha} (\log \log T)^{\frac{r-1}{2}}.
\]
\end{theorem}
\begin{remark}
Theorems \ref{main thm 2} and \ref{thm:twomMomentUnconditional} hold unconditionally when $m=1$ or $m=2$.
\end{remark}
\subsection{Smooth Weight Functions} \label{sec: Smooth weight functions} Fix $X \geqslant 2$ and consider the function $f$ on $(0,\infty)$ defined by
\begin{align}\label{Defining Montgomery function}
f(x) := \begin{cases}
 \left (\dfrac{x}{X} \right )^2, & x \leqslant X \\
 \left (\dfrac{X}{x} \right )^2, & x \geqslant X.
\end{cases}
\end{align} 
Then \eqref{S definition} can be viewed as 
\begin{align}\mathcal{S}_{\pi}(X,T,\rho_{\pi}') &= \sum_{\substack{ \lvert \Im{\rho_{\pi}} \rvert \leqslant T}} \widehat{f}(\rho_{\pi}-\rho_{\pi}'), \label{S definition General Weights}
\end{align}
where $\widehat{f}$ denotes the Mellin transform of $f$. More generally, let $\Psi \in C_{c}^{\infty}(0,\infty)$ be a fixed non-negative compactly supported smooth function. Fix $X \geqslant 2$ and define
\begin{equation}\label{eqn:psi_W}
    \Psi_X(x) := \Psi\left(\frac{x}{X}\right), \quad x \in \mathbb{R}.
\end{equation}
Let $T \geqslant 2$ and define
\begin{align}\label{rewriting G general}
 \mathcal{S}_{\textnormal{av},\pi, \Psi}(X,T)&:= \frac{1}{N_{\pi}( T )}
 \sum_{\substack{\lvert \Im{\rho_{\pi}'} \rvert \leqslant T}} \Re \mathcal{S}_{\pi,\Psi}(X, T, \rho_{\pi}')
\end{align}
 where
\begin{align}\mathcal{S}_{\pi,\Psi}(X,T,\rho_{\pi}') &:= \sum_{\substack{ \lvert \Im{\rho_{\pi}} \rvert \leqslant T}}\widehat{\Psi}_X(\rho_{\pi} - \rho'_{\pi}), \label{eqn:S_test}
\end{align}
and $\rho_{\pi}, \rho_{\pi}'$ vary over the non-trivial zeros of $L(s, \pi)$. 
For $r \in \mathbb{N}$, consider the moments 
\begin{equation}\label{eqn:moments_test}
\mathcal{M}_{\pi, \Psi, r}(X,T) := \frac{1}{N_{\pi}(T)}\sum_{\substack{ \lvert \Im \rho_{\pi}' \rvert  \leqslant T}} \left(\Re \mathcal{S}_{\pi,\Psi}(X, T, \rho_{\pi}') -\widehat{\Psi}(0) \frac{\min \{\log X, m \log T \}}{m \log T}  \right)^{r}.
\end{equation}
We show that the results from Section \ref{PCS Introduction Subsection} can be extended for the class of smooth compactly supported weight functions. Furthermore, we can improve upon the admissible range for $\alpha$ in Theorem \ref{thm:twomMomentConditional}.
\begin{theorem}\label{Thm:Main_for_test_funcs}
Let $m \in \mathbb{N}, \pi \in \mathscr{A}_m$ and $\theta_{m} \in [0, \frac{1}{2}-\frac{1}{m^2+1}]$ be an admissible exponent towards the Ramanujan conjecture for $L(s,\pi)$. Assume the Riemann Hypothesis for $L(s,\pi)$. Let $\Psi \in C_{c}^{\infty}(0,\infty)$ be a fixed non-negative compactly supported smooth function. Let $X,T \geqslant 2, r \in \mathbb{N}$ and $\mathcal{M}_{\pi,\Psi, r}(X,T)$ be as defined in \eqref{eqn:moments_test}. Fix $\alpha \in \mathbb{R}$ such that 
\[0< \alpha < \frac{1}{mr(1+\theta_m)}.
\]
If $r$ is even, then
\[
\mathcal{M}_{\pi,\Psi, r}(T^{\alpha m},T) = \boldsymbol{\mu}_{r}\left( \widehat{\Psi^2}(0)\frac{\alpha m \log T}{2}\right)^{\frac{r}{2}} + O_{\pi,\Psi, r,\alpha}((\log T)^{\frac{r-1}{2}}),
\]
where $\boldsymbol{\mu}_{r}$ is given by \eqref{Defining Normal Distribution Coefficients}. If $r$ is odd, then
\[
\mathcal{M}_{\pi,\Psi, r}(T^{\alpha m},T) \ll_{\pi,\Psi,r, \alpha} (\log T)^{\frac{r-1}{2}}.
\]
\end{theorem}
Similar to Conjecture \ref{conj:gaussian}, we expect the results in Theorem \ref{Thm:Main_for_test_funcs} to hold for all fixed $\alpha \in \mathbb{R}\setminus \{0\}$. Define 
\begin{equation} \label{Defining Q General}
\mathcal{Q}_{\pi, \Psi}(X,T,\rho_{\pi}') : =
\frac{\Re \mathcal{S}_{\pi,\Psi}(X, T, \rho_{\pi}')-\widehat{\Psi}(0)(m \log T)^{-1} \min \{ \log X, m \log T\}  }{ \left( (1/2) \min \{ \log X, m \log T \} \widehat{\Psi^2}(0)\right)^{\frac{1}{2}}},
\end{equation}
where $\mathcal{S}_{\pi, \Psi}(X,T,\rho'_{\pi})$ is defined by \eqref{S definition General Weights}. As a consequence of Theorem \ref{Thm:Main_for_test_funcs}, we also obtain the analogue of Theorem \ref{main thm} for $ \mathcal{Q}_{\pi,\Psi}(X,T,\rho_{\pi}')$. 

\subsection{Distributions Associated to Triple Correlation} \label{sec: triple cor intro} This section concerns the distribution of non-trivial zeros of automorphic $L$-functions with respect to their triple correlation.  Let $\Psi \in C_{c}^{\infty}(0,\infty)$ be a fixed non-negative compactly supported smooth function. Let $X_1,X_2,T \geqslant 2$ and consider the sum
\begin{align}
\widehat{\mathcal{S}}_{\pi,\Psi}(X_1,X_2,T,\rho_{\pi}) &:= \sum_{\substack{ \lvert \Im{\rho_{\pi,1}} \rvert \leqslant T}} \sum_{\substack{ \lvert \Im{\rho_{\pi,2}} \rvert \leqslant T}}\widehat{\Psi}_{X_1}(\rho_{\pi,1} - \rho_{\pi}) \widehat{\Psi}_{X_2}(\rho_{\pi} - \rho_{\pi,2}), \label{eqn:S triple}
\end{align}

where $\rho_{\pi,1}, \rho_{\pi,2}$ vary over the non-trivial zeros of $L(s, \pi)$. Then $\widehat{\mathcal{S}}_{\pi,\Psi}(X_1,X_2,T,\rho_{\pi})$ simultaneously captures information about the gaps $\{ (\rho_{\pi,1} - \rho_{\pi})\}$ and $\{ (\rho_{\pi} - \rho_{\pi,2})\}$ as $\rho_{\pi,1}$ and $\rho_{\pi,2}$ vary over neighboring zeros of $\rho_{\pi}$. For $r \in \mathbb{N}$, consider the moments
\begin{equation}\label{Definition of Moments: Triple Correlation}
\widehat{\mathcal{M}}_{\pi, \Psi, r}(X_1,X_2,T) := \frac{1}{N_{\pi}(T)}\sum_{\substack{ \lvert \Im \rho_{\pi} \rvert  \leqslant T}} (\Re \widehat{\mathcal{S}}_{\pi,\Psi}(X_1,X_2, T, \rho_{\pi}))^{r}.
\end{equation}
For $r \in \mathbb{N}$, let
\begin{align}\label{Laplace Coefficients}
\boldsymbol{\mathcal{L}}_r:=\begin{cases}
    r! & \text{if $r$ is even,}\\
    0 &\text{if $r$ is odd,}
    \end{cases} 
\end{align}
and
\begin{align}\label{Chi-Square Coefficients}
\boldsymbol{\mathcal{\chi}}_r=  2^r \cdot r!.
\end{align}
\begin{figure}[t]
    \centering
     {{\includegraphics[width=.5\linewidth]{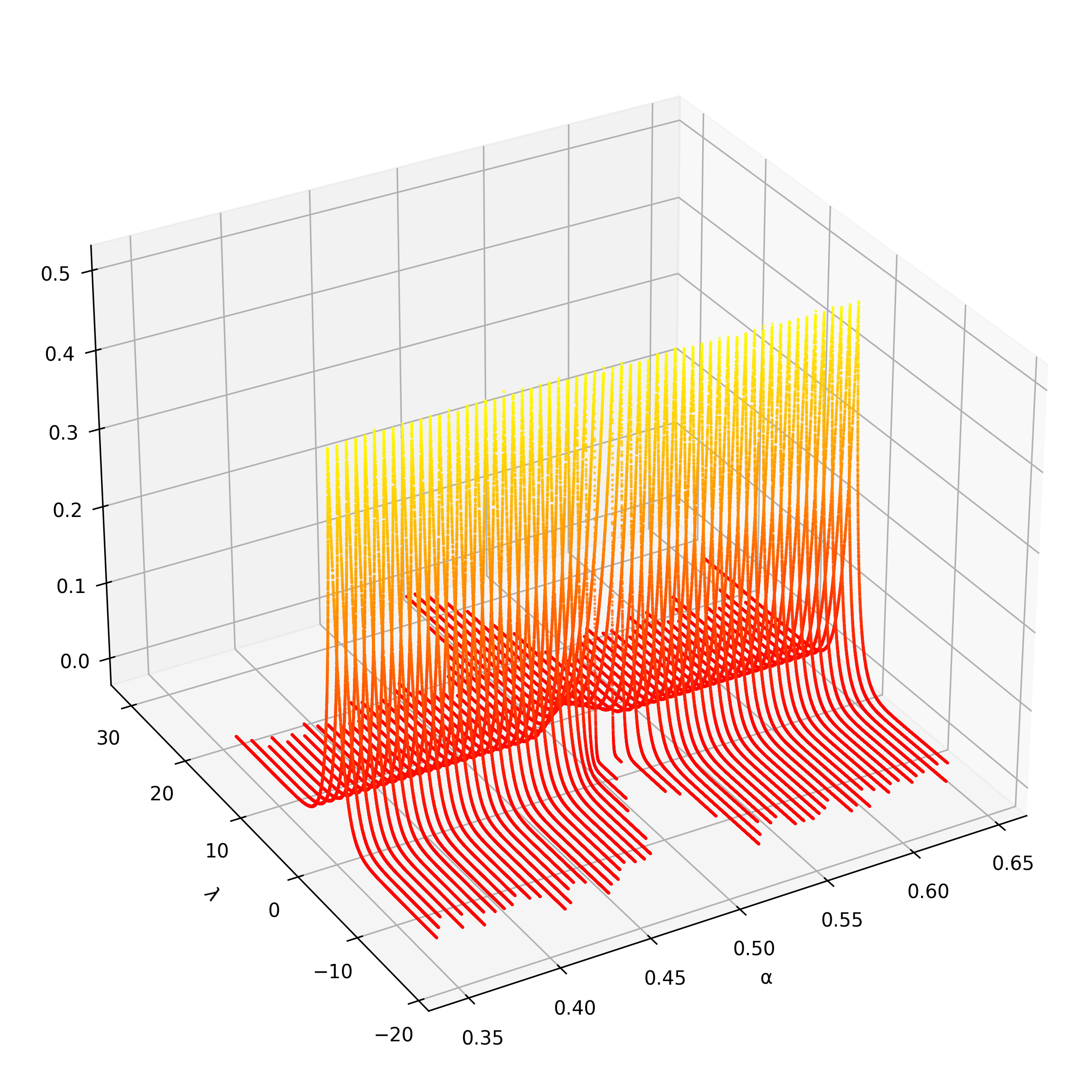}}}%
    \caption{ A discrete approximation of the probability density function associated to $\widehat{\mathcal{S}}_{\pi,\Psi}(X_1,X_2,T,\rho_{\pi})$ where $L(s,\pi)=\zeta(s)$, $X_1= T^{\frac{1}{2}}$ and $X_2=T^{\alpha_2}$ with $\alpha_2$ varying between $(0.35,0.65)$ in a discrete equi-spaced manner. As $\alpha_2$ approaches $\frac{1}{2}$, the graph gradually transitions from Laplace to a Chi-squared distribution.}
    \label{fig:TCS Plot 2}
\end{figure}
\begin{theorem}\label{Theorem : Triple Correlation}
Let $m \in \mathbb{N}, \pi \in \mathscr{A}_m$ and $\theta_{m} \in [0, \frac{1}{2}-\frac{1}{m^2+1}]$ be an admissible exponent towards the Ramanujan conjecture for $L(s,\pi)$. Assume the Riemann Hypothesis for $L(s,\pi)$. Let $\Psi \in C_{c}^{\infty}(0,\infty)$ be a fixed non-negative compactly supported smooth function. Let $X_1,X_2,T \geqslant 2, r \in \mathbb{N}$ and $\widehat{\mathcal{M}}_{\pi,\Psi, r}(X_1,X_2,T)$ be as defined in \eqref{Definition of Moments: Triple Correlation}. Fix $\alpha_1, \alpha_2 >0$ such that
\[0< \alpha_1+\alpha_2 < \frac{1}{mr(1+\theta_m)}.
\]
Then the following holds.
\begin{itemize}
\item[(1)] Suppose $\alpha_1 \neq \alpha_2$. If $r$ is even, then
\[
\widehat{\mathcal{M}}_{\pi,\Psi, r}(T^{\alpha_1 m},T^{\alpha_2 m},T) = \boldsymbol{\mathcal{L}}_{r}\left(\widehat{\Psi^2}(0)\frac{m \sqrt{\alpha_1 \alpha_2} \log T}{2}  \right)^{r} + O_{\pi,\Psi, r,\alpha_1,\alpha_2}((\log T)^{r-\frac{1}{2}}),
\]
where $\boldsymbol{\mathcal{L}}_{r}$ is given by \eqref{Laplace Coefficients}.
If $r$ is odd, then
\[
\widehat{\mathcal{M}}_{\pi,\Psi, r}(T^{\alpha_1 m },T^{\alpha_2 m},T) \ll_{\pi,\Psi,r, \alpha_1,\alpha_2} (\log T)^{\frac{r-1}{2}}.
\]
\item[(2)] Suppose $\alpha_1 = \alpha_2 = \alpha$. Then for any $r \in \mathbb{N}$,
\[
\widehat{\mathcal{M}}_{\pi,\Psi, r}(T^{\alpha m},T^{\alpha m},T) = \boldsymbol{\mathcal{\chi}}_{r}\left(\widehat{\Psi^2}(0)\frac{\alpha m \log T}{2} \right)^{r} + O_{\pi,\Psi, r,\alpha}((\log T)^{r-\frac{1}{2}}).
\]
\end{itemize}
\end{theorem}

\begin{remark}
The moments in the case of triple correlation are not Gaussian. In fact, the elements of the sequence $\boldsymbol{\mathcal{L}}_{r}$ defined by \eqref{Laplace Coefficients} are the moments of a Laplace distribution $\textnormal{Laplace }(\mu,b)$ with mean $\mu=0$ and the scaling parameter $b=1$. On the other hand, the elements of the sequence $\boldsymbol{\mathcal{\chi}}_{r}$ defined by \eqref{Chi-Square Coefficients} are the moments of a Chi-squared distribution $\chi^2_k$ with degree of freedom $k=2$. In particular, the distribution of $\Re \widehat{\mathcal{S}}_{\pi,\Psi}(T^{\alpha_1 m },T^{\alpha_2 m}, T, \rho_{\pi}')$ exhibits the characteristics of a Laplace distribution when $\alpha_1 \neq \alpha_2$ and that of a Chi-squared distribution when $\alpha_1=\alpha_2$.
\end{remark}
Theorem \ref{Theorem : Triple Correlation} now implies the following result on the distribution of $ \widehat{\mathcal{S}}_{\pi,\Psi}(X_1,X_2,T,\rho_{\pi}')$. Define
\begin{equation} \label{Defining Q Triple}
\widehat{\mathcal{Q}}_{\pi, \Psi}(X_1,X_2,T,\rho_{\pi}') : =
\frac{\Re \widehat{\mathcal{S}}_{\pi,\Psi}(X_1,X_2,T,\rho_{\pi}')}{(1/2) \cdot \widehat{\Psi^2}(0)\cdot (\min \{ \log X_1, m \log T \} \cdot \min \{ \log X_2, m \log T \})^{\frac{1}{2}}},
\end{equation}
\begin{theorem}\label{Theorem : Triple Correlation CDF} Let $m \in \mathbb{N}$ and $\pi \in \mathscr{A}_m$. Assume the Riemann Hypothesis for $L(s,\pi)$. Let $\{\alpha_{1j} \}_{j \in \mathbb{N}}$ and $\{\alpha_{2j} \}_{j \in \mathbb{N}}$ be two sequences of
 strictly positive real numbers decreasing to $0$. Then there exists an increasing sequence $\{V_j\}_{j \in \mathbb{N}}$
 tending to $\infty$ with the following property.
 For any sequence $\{T_j\}_{j \in \mathbb{N}}$ such that
 $T_j \geqslant V_j$ for all $j \in \mathbb{N}$, we have for any $\lambda \in \mathbb{R}$,
 \begin{itemize}
\item[(1)] If $\alpha_{1j} \neq \alpha_{2j}$ for all $j \geqslant J_0$, then
\begin{align*}
  \lim_{j \to \infty} \frac{1}{N_{\pi}( T_j)}
 \# \left\{ \lvert \Im(\rho_{\pi}') \rvert \leqslant T_j \colon \widehat{\mathcal{Q}}_{\pi, \Psi}(T_j^{\alpha_{1j}},T_j^{\alpha_{2j}},T_j,\rho_{\pi}') 
   < \lambda \right\}  = \frac{1}{2}\int_{-\infty}^{\lambda} e^{-|u|} \, \diff  u,   
\end{align*}
\item[(2)] If $\alpha_{1j} = \alpha_{2j}$ for all $j \geqslant J_0$ then
\[
\lim_{j \to \infty} \frac{1}{N_{\pi}( T_j)}
 \# \left\{ \lvert \Im(\rho_{\pi}') \rvert \leqslant T_j \colon \widehat{\mathcal{Q}}_{\pi, \Psi}(T_j^{\alpha_{1j}},T_j^{\alpha_{2j}},T_j,\rho_{\pi}') 
   < \lambda \right\}  = \frac{1}{2}\int_{0}^{\lambda} e^{-\frac{u}{2}}\, \mathrm d u,
   \]
\end{itemize}
where $N_{\pi}(T)$ is as in \eqref{numOfZeros}, $\rho_{\pi}'$ runs over non-trivial zeros of $L(s,\pi)$ and $\widehat{\mathcal{Q}}_{\pi,\Psi}(X_1,X_2,T,\rho_{\pi}')$ is defined by \eqref{Defining Q Triple}.
\end{theorem}

\begin{remark}
In Figure \ref{fig:TCS Plot 2}, we observe that the Laplace distribution begins transitioning into a Chi-squared distribution as $\alpha_2$ approaches $\alpha_1$. A natural question is what specific insights can be drawn about this phase transition, where the Laplace distribution evolves into the Chi-squared distribution. In this context, our proof of Theorem \ref{Theorem : Triple Correlation} actually yields a slightly stronger result.  For $T \geqslant 2$ sufficiently large, if
\[
\lvert \alpha_1 - \alpha_2 \rvert \geqslant \frac{\log \log T}{\log T},
\]
then we have Conclusion (1) from the statement of Theorem \ref{Theorem : Triple Correlation}. Conversely, if
\[
\lvert \alpha_1 - \alpha_2 \rvert \leqslant \frac{1}{(\log \log T)(\log T)},
\]
then Conclusion (2) holds. Similar refinements apply to the statement of Theorem \ref{Theorem : Triple Correlation CDF}.
\end{remark}

\begin{figure}[t]
    \centering
    {{\includegraphics[width = .35\textwidth]{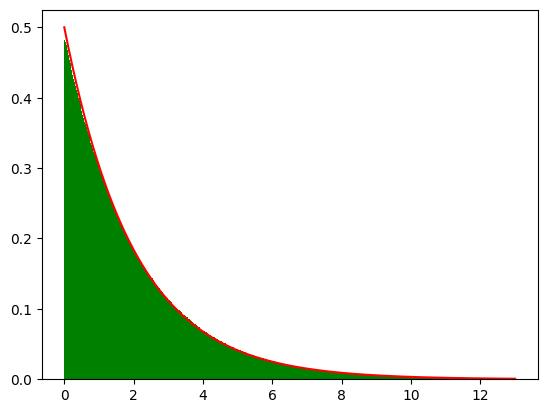}}}
     {{\includegraphics[width=.35\textwidth]{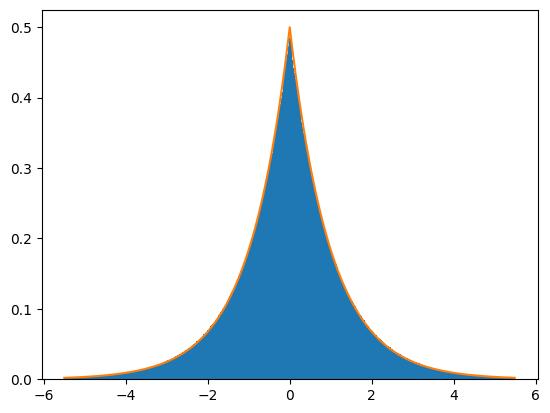}}}%
    \caption{\underline{Left:} A discrete approximation of the probability density function associated to $\widehat{\mathcal{S}}_{\pi,\Psi}(X_1,X_2,T,\rho_{\pi})$ where $L(s,\pi)=\zeta(s)$, $X_1=X_2= T^{\frac{1}{2}}$ (in green). The probability distribution of a Chi-squared distribution with two degrees of freedom (in red). \underline{Right:} A discrete approximation of the probability density function associated to $\widehat{\mathcal{S}}_{\pi,\Psi}(X_1,X_2,T,\rho_{\pi})$ where $L(s,\pi)=\zeta(s)$, $X_1= T^{\frac{1}{2}}$ and $X_2 = T^{\frac{2}{5}}$ (in blue). The probability distribution of a Laplace distribution with mean zero and scaling parameter one (in orange).}
    \label{fig:TCS}
\end{figure}
\subsection{Structure of the Paper}  The paper is organized as follows. In Section \ref{sec: Proof Outline}, we provide brief outlines of the proofs of our results and discuss some potential directions for future research. In Section \ref{sec:preliminaries}, we provide some background on automorphic \( L \)-functions and introduce preliminary lemmas which include explicit formulas for \( L(s,\pi) \), Landau--Gonek type results, and key asymptotic estimates related to Hypothesis \textbf{H$_{\pi}$}. In Section \ref{sec: Conditional 1}, we carry out majority of our work towards estimating the moments in the case of pair correlation. Section \ref{sec: Conditional 2} is devoted to the proofs of Theorems \ref{main thm} and \ref{thm:twomMomentConditional}. In Section \ref{sec: unconditional 1}, we address the issue of uniformity in $\alpha$ with respect to $T$ and prove results using Hypothesis \textbf{Z$_{\pi}$} instead of the Riemann Hypothesis. The proofs of Theorems \ref{main thm 2} and \ref{thm:twomMomentUnconditional} are provided in Section \ref{sec : unconditional 2}. In Section \ref{sec: general weight functions}, we move to a general class of smooth compactly supported weight functions and prove Theorem \ref{Thm:Main_for_test_funcs}. Finally, in Section \ref{sec: Triple Correlation}, we study the distributions in the case of triple correlation and prove Theorems \ref{Theorem : Triple Correlation} and \ref{Theorem : Triple Correlation CDF}.

\section{Outline of the Proofs and Future Directions} \label{sec: Proof Outline}
In this section, we provide a brief overview of the proofs of our main results, and discuss some directions for future research. As mentioned in Section \ref{sec: Introduction}, the proof of Theorem \ref{main thm} relies on establishing the moment estimates in Theorem \ref{thm:twomMomentConditional}. The first step in the proof of Theorem \ref{thm:twomMomentConditional} is an application of the explicit formula for $L(s,\pi)$, see Lemma \ref{lem:explicit_formula_for_L(s,pi)}. We then show that, apart from an exceptional set of zeros of cardinality at most $O_{\pi}(\sqrt{T} \log T)$, the Riemann Hypothesis for $L(s,\pi)$ implies that
\begin{equation}\label{Eq: Zeros to Primes}
\mathcal{S}_{\pi}(X,T, \rho'_{\pi}) = - \sum_{n =1}^{\infty} \frac{a_n\Lambda_{\pi}(n)}{n^{\frac{1}{2}+i\gamma'_{\pi}}}+\frac{m \log( \lvert \gamma'_{\pi} \rvert+3)}{X^2}+ O_{\pi}\bigg(\frac{\sqrt{X}}{T} + \frac{1}{X^2}+\frac{\log T}{\sqrt{T}}\bigg).
\end{equation}
where $\{a_n\}$ is the sequence given by \eqref{Defining coefficients}, see Lemma \ref{Connecting S to sum over primes}. Thus, in the setting of Theorem \ref{thm:twomMomentConditional}, the principal contribution to the distribution of $\mathcal{S}_{\pi}(X,T,\rho'_{\pi})$ will arise from the prime sum over $n$ appearing on the right-hand side of \eqref{Eq: Zeros to Primes}. To evaluate these prime sums and their moments, we require Landau--Gonek type formulas for general automorphic $L$-functions, see Corollaries \ref{Gonek assuming RH} and \ref{when b>a}. We shall also need asymptotic estimates for sums of the form
\begin{align}\label{Eq: Variance Estimation}
\sum_{n \leqslant V} \frac{a_n \lvert \Lambda_{\pi}(n) \rvert^2}{n}, \quad \textrm{and} \quad  \sum_{n \leqslant V} \frac{a_n^2 \lvert \Lambda_{\pi}(n) \rvert^2}{n},
\end{align}
see Lemma \ref{Rudnick-Sarnak Lemma}. We derive these estimates from the asymptotic formula
\begin{align}\label{Eq: Main Rudnick-Sarnak Lemma}
\sum_{n \leqslant X} \frac{|\Lambda_{\pi}(n)|^2}{n}
= \frac{\log^2 X}{2} + O_{\pi}(1).
\end{align}
Note that the asymptotic estimate for the left-hand side of \eqref{Eq: Main Rudnick-Sarnak Lemma} from \cite[Prop. 2.3]{rudnickZerosPrincipalLfunctions1996} comes with an error term of $O_{\pi}(\log X)$, which is insufficient to deduce the asymptotics for the sums in \eqref{Eq: Variance Estimation} needed in our setting. To obtain the necessary relations and establish the stronger estimate \eqref{Eq: Main Rudnick-Sarnak Lemma}, we exploit the currently known zero-free regions for $L(s,\pi)$. This is done across Lemma \ref{lem:prime_number_theorem} and \ref{Rudnick-Sarnak Lemma}.
\par
Following this, in Section \ref{sec: Conditional 1}, we begin by evaluating the mixed moments
\begin{align}\label{Eq: Mixed Moments}
\frac{1}{N_{\pi}(T)}\sum_{\lvert \gamma_{\pi} \rvert \leqslant T}
\bigg(\sum_{p \leqslant V} \frac{a_p \Lambda_{\pi}(p)}
{p^{\frac{1}{2}+i\gamma_{\pi}}}\bigg)^k
\bigg(\sum_{q \leqslant V} \frac{a_q \overline{\Lambda_{\pi}(q)}}
{q^{\frac{1}{2}-i\gamma_{\pi}}}\bigg)^\ell,
\end{align}
where $p$ and $q$ run over primes. In Lemma \ref{Moment Calculations Assuming RH Step 1}, we show that the diagonal terms produce the principal contributions to \eqref{Eq: Mixed Moments}, while
the off-diagonal terms are shown to be negligible through an
application of the Landau--Gonek formulas. In Lemma
\ref{Moment Calculations Assuming RH Step 3}, we further establish that
the contribution from prime powers is negligible. To this end, we divide
the prime powers \(p^k\) into the ranges \(2 \leqslant k \leqslant m^2+1\) and
\(k>m^2+1\). The contribution from the latter range is controlled
directly using the bound $0 \leqslant \theta_m \leqslant \frac{1}{2}-\frac{1}{m^2+1},$ where $\theta_{m} \in [0, \frac{1}{2}-\frac{1}{m^2+1}]$ is an admissible exponent towards the Ramanujan conjecture for $L(s,\pi)$. For the former range, we start by using Jensen's inequality and expanding the
resulting sums. The contribution from these sums are then bounded by employing a combination of tools including Hypothesis \textbf{H$_{\pi}$}, the Landau--Gonek formulas and
asymptotic relations for the sums in \eqref{Eq: Variance Estimation}. Our final step in the proof of Theorem \ref{thm:twomMomentConditional} is to show
\[
\frac{1}{N_{\pi}(T)} \sum_{\lvert \gamma'_{\pi} \rvert \leqslant T} (\Re \mathcal{S}_{\pi}(X, T, \rho'_{\pi}) )^{2d} = \frac{1}{N_{\pi}(T)}\sum_{\lvert \gamma'_{\pi} \rvert \leqslant T} \bigg (\Re \sum_{n \leqslant V} \frac{a_n\Lambda_{\pi}(n)}{n^{\frac{1}{2}+i\gamma'_{\pi}}}\bigg)^{2d} +\textnormal{Smaller Order Terms}.
\]
There is an obstruction here as far as the size of $V$ relative to $X$ is concerned. It is crucial that we choose $V=X^{\frac{4}{3}+\varepsilon}$ since the sequence $\{a_n\}$ does not decay fast enough, see Lemma \ref{S_av to Mean} and Remark \ref{Rem: xi>4/3}. Finally we use our estimates for the moments over prime sums and Lemma \ref{S_av to Mean} to deduce the desired result. 
\par
For the proofs of Theorems \ref{main thm 2} and
\ref{thm:twomMomentUnconditional}, in the
short range $X=(\log T)^{\alpha m},$ the second term on the right-hand side of
\eqref{Eq: Main Rudnick-Sarnak Lemma} becomes dominant. Therefore, in order to still be able to extract the desired Gaussian distribution, we work with the shifted sums
\(\widetilde{\mathcal{S}}_{\pi}(X,T,\rho'_{\pi})\) defined in
\eqref{Shifted S}. The key new ingredient in this setting is Proposition \ref{Crucial RH removing Lemma}, which allows us to establish
the required moment estimates for the associated prime sums under Hypothesis \textbf{Z$_{\pi}$}, instead of assuming RH. For the proof of Theorem \ref{Thm:Main_for_test_funcs}, which treats
general weight functions, we employ an appropriate explicit
formula (see Lemma \ref{lem:explicit_formula_test}) and prove that the
shorter truncation $V=X^{1+\varepsilon}$
is sufficient. This ultimately
allows us to extend the admissible range of \(\alpha\) beyond that
obtained in Theorem \ref{thm:twomMomentConditional}. 
\par
For the more complicated case of triple correlations, that is, in the context of Theorems \ref{Theorem : Triple Correlation} and \ref{Theorem : Triple Correlation CDF}, the prime sums to consider are of the form
\begin{align*}
\sum_{\substack{m_{11}, \dots, m_{k1} \\ m_{12} ,\dots, m_{k2} \\ n_{11}, \dots, n_{\ell1}\\
n_{12}, \dots, n_{\ell 2}}} &\frac{\prod_{i=1}^{k} \Psi_{X_1}(m_{i1})   \Psi_{X_2}(m_{i2}) \Lambda_{\pi}(m_{i1}) \Lambda_{\widetilde{\pi}}(m_{i2}) \prod_{j=1}^{\ell} \Psi_{X_1}(n_{j1}) \Psi_{X_2}(n_{j2}) \Lambda_{\widetilde{\pi}}(n_{j1}) \Lambda_{\pi}(n_{j2})}{(M_1M_2N_1N_2)^{\frac{1}{2}}}\notag \\
&\quad \times \frac{1}{N_{\pi}(T)} \sum_{\lvert \gamma_{\pi} \rvert \leqslant T} \bigg(\frac{M_2N_1}{M_1N_2} \bigg )^{i \gamma_{\pi}},
\end{align*}
where we set $M_1 = \prod_{i=1}^{k}m_{i1}$, $M_2 = \prod_{i=1}^{k}m_{i2}$, $N_1 = \prod_{j=1}^{\ell} n_{j1}$ and $N_2 = \prod_{j=1}^{\ell} n_{j2}$. The main contribution arises from the diagonal, which, after restricting
to primes, is given by the relation $M_2N_1=M_1N_2$.
The combinatorial structure of these diagonal configurations depends
crucially on whether \(\alpha_1=\alpha_2\) or
\(\alpha_1\neq\alpha_2\). These two regimes exhibit distinct diagonal structures and lead to different limiting
distributions, namely the Laplace and the Chi-squared distributions. Bounding the contribution from the non-diagonal terms and the prime powers is more delicate in this scenario. In particular, if $\alpha_1,\alpha_2$ satisfy certain linear relations over $\mathbb{Q}$, some specific configurations involving prime
powers can prevent any cancellation in the sum over
\(\gamma_{\pi}\). We address this situation in Lemma \ref{Moment Calculations Assuming RH Step 3 3-Cor}, which allows us to establish the diagonal contributions as the main term and complete the proofs of Theorems \ref{Theorem : Triple Correlation} and \ref{Theorem : Triple Correlation CDF}.
\par
We conclude this section by outlining some directions for future research related to Conjecture \ref{conj:one} and the results of
this paper. A natural first direction is to formulate and investigate a family aspect analogue of Conjecture \ref{conj:one}, in the spirit of the
Katz--Sarnak philosophy, with the goal of understanding how the pair correlation surfaces vary across different families of
\(L\)-functions. A second  direction is the study of higher order correlation
surfaces and hypersurfaces, where one works with any given finite number of parameters $\alpha_i$. In this setting, one may expect newer distributions to emerge, depending on the arithmetic relations among the \(\alpha_i\)'s. A third possible direction is to establish a direct
connection between these correlation surfaces and the distribution of
eigenvalues from Random Matrix Theory. Such a connection could provide a
random matrix model for the distributions studied in this work and allow one
to compare the results with the predictions arising from
the corresponding random matrix ensembles. Finally, it may be worth investigating if the new techniques introduced recently by Claude \cite{alpöge2026thirdszetazerossimple} can help in the study of correlation surfaces considered in this work. 

\section{Preliminaries}\label{sec:preliminaries}
\subsection{General Notations} For most part of the paper, notation will be introduced when it is needed. Aside from that, we employ some standard notations.
\begin{itemize}
    \item Throughout the paper, the expressions $f(X)=O(g(X))$, $f(X) \ll g(X)$, and $g(X) \gg f(X)$ are equivalent to the statement that $|f(X)| \leqslant C|g(X)|$ for all sufficiently large $X$, where $C>0$ is an absolute constant. A subscript of the form $\ll_{\alpha}$ means the implied constant may depend on the parameter $\alpha$. Dependence on several parameters is indicated in an analogous manner, as in $\ll_{\alpha, \lambda}$.
\item For any set $\mathcal{A}$, $\#\mathcal{A}$ denotes the cardinality of the set $\mathcal{A}$.
\item For $s \in \mathbb{C},$ we denote the Mellin transform of $\varphi: (0,\infty) \to \mathbb{C}$ by 
\begin{equation*}
    \widehat{\varphi}(s) := \int_{0}^{\infty}\varphi(x)x^{s-1} \diff x 
\end{equation*}
when the integral exists.
    \end{itemize}
\subsection{Background on $L$-functions.} We recall some standard facts about $L$-functions arising from automorphic representations. For more details, see Iwaniec--Kowalski \cite{IK2004} and Rudnick--Sarnak \cite{rudnickZerosPrincipalLfunctions1996}.
\par
Let $m \in \mathbb{N}$ and  $\mathscr{A}_m$ be the set of irreducible cuspidal automorphic representations of
$\mathrm{GL}_m$ over $\mathbb{Q}$ with unitary central character. Given $\pi \in \mathscr{A}_m$ with conductor $q_\pi$, let $L(s, \pi)$ be its standard $L$-function. There exist suitable complex numbers $\alpha_{1, \pi}(p), \alpha_{2, \pi}(p), \dots, \alpha_{m, \pi}(p)$ such that
\begin{align}\label{Euler Product}
L(s, \pi)=\prod_{p \text { prime }} \prod_{j=1}^m (1-\alpha_{j, \pi}(p) p^{-s})^{-1}=\sum_{n=1}^{\infty} \frac{\lambda_\pi(n)}{n^s} .  
\end{align}
Both the sum and the product in \eqref{Euler Product} converge absolutely for $\Re(s)>1$. The $\alpha_{j,\pi}(p), 1 \leqslant j \leqslant m$ are called the local roots or local parameters of $L(s,\pi)$ at the prime $p$ and they satisfy $\lvert \alpha_{j,\pi}(p) \rvert < p$ for all primes $p$. We also have spectral parameters $\kappa_\pi(1), \kappa_\pi(2), \dots, \kappa_\pi(m) \in \mathbb{C}$ such that if we define
\[
L(s, \pi_{\infty}):=\pi^{-\frac{ms}{2}} \prod_{j=1}^m \Gamma\bigg (\frac{s+\kappa_\pi(j)}{2}\bigg),
\]
then the completed $L$-function 
\[
\Lambda(s, \pi):=(s(1-s))^{\delta_{\pi}} q_\pi^{\frac{s}{2}} L(s, \pi) L(s, \pi_{\infty})
\]
is entire of order one. Here we let $\delta_{\pi} = 1$ if $\pi \in \mathscr{A}_1 $ is trivial; otherwise $\delta_{\pi} = 0$. Let $\widetilde{\pi} \in \mathscr{A}_m$ be the contragredient representation. Then $\alpha_{j, \widetilde{\pi}}(p)=\overline{\alpha_{j, \pi}(p)}$ and $\kappa_{\widetilde{\pi}}(j)=\overline{\kappa_\pi(j)}$ for $j=1,2, \dots, m$. Moreover, there exists a complex number $\varepsilon_{\pi}$ of modulus one such that for all $s \in \mathbb{C}$,
\begin{align}\label{Functional Equation}
\Lambda(s, \pi)=\varepsilon_{\pi} \Lambda(1-s, \widetilde{\pi}).
\end{align}
The bounds
\begin{align}\label{Towards Ramanujan Bounds}
\log_p\left|\alpha_{j, \pi}(p)\right| \leqslant \theta_m, \quad \operatorname{Re}\left(\kappa_\pi(j)\right) \geqslant -\theta_m
\end{align}
hold for some 
\begin{align}\label{Luo-Rudnick-Sarnak Bound}
0 \leqslant \theta_m \leqslant \frac{1}{2}-\frac{1}{m^2+1}
\end{align}
due to Luo--Rudnick--Sarnak \cite{luo1995selberg, luo1999generalized} and Rudnick--Sarnak \cite{rudnickZerosPrincipalLfunctions1996}. The Ramanujan conjecture and the Selberg eigenvalue conjecture assert that \eqref{Towards Ramanujan Bounds} holds with $\theta_m=0$. This is not known except in specific cases, such as when $m = 1$, and for $\mathrm{GL}_2$ $L$-functions corresponding to holomorphic forms due to Deligne \cite{Deligne}. For $\operatorname{Re}(s)>1$, we take the logarithmic derivative of \eqref{Euler Product} to obtain
\begin{align}\label{Log-Derivative}
\frac{L^{\prime}}{L}(s, \pi)=-\sum_{n=1}^{\infty} \frac{\Lambda_{\pi}(n)}{n^s}, 
\end{align}
where $\Lambda_{\pi}(n)$ is supported on prime powers and is given by
\begin{align}\label{Lambda Def}
\Lambda_\pi(p^k)=\sum_{j=1}^m \alpha_{j,\pi}(p)^k \log p.
\end{align}
Denote by $\rho_{\pi}=\beta_{\pi}+i\gamma_{\pi}$ the non-trivial zeros of $L(s,\pi)$. Define
\begin{align}\label{Number of Zeros}
N_\pi(\sigma, T):=\#\{\rho_{\pi}=\beta_{\pi}+i \gamma_{\pi}: \beta_{\pi} \geqslant \sigma,|\gamma_{\pi}| \leqslant T, L(\rho_{\pi}, \pi)=0\}.
\end{align}
As with $\zeta(s)$, following a standard winding number argument, we can show that
\begin{align}\label{Zero Counting General}
N_\pi(T):=N_\pi(0, T) \sim \frac{T}{\pi} \log q_{\pi}T^{m}.
\end{align}
The Riemann Hypothesis $(\mathrm{RH})$ for $L(s,\pi)$ asserts that $\Re(\rho_{\pi})=\beta_{\pi} = \frac{1}{2}$.
\par 
\subsection{Sums over Zeros to Sums over Primes} Landau \cite[p. 353]{landau1909handbuch} showed that 
\begin{align}\label{Landau's Explicit Formula}
\sum_{n \leqslant X} \frac{\Lambda(n)}{n^s}=-\frac{\zeta'}{\zeta}(s)+\frac{X^{1-s}}{1-s}-\sum_{\rho}\frac{X^{\rho-s}}{\rho-s}+\sum_{n=1}^\infty \frac{X^{-2n-s}}{2n+s},
\end{align}
where $X>1$ is not a prime power, $\rho$ runs over the non-trivial zeros of $\zeta(s)$ and $s \neq 1, s \neq \rho, s \neq -2n$. We shall need a truncated version of \eqref{Landau's Explicit Formula} for $L$-functions arising from automorphic representations. The proof follows standard arguments and for the sake of completeness, we provide it in the Appendix.
\begin{lemma} \label{lem:explicit_formula_for_L(s,pi)}
Let $m \in \mathbb{N}, \pi \in \mathscr{A}_m$ and $\theta_{m} \in [0, \frac{1}{2}-\frac{1}{m^2+1}]$ be an admissible exponent towards the Ramanujan conjecture for $L(s,\pi)$. Let \( X \geqslant 2 \) and 
\[T \geqslant 2+\lvert \operatorname{Im}(s) \rvert+\lvert \operatorname{Im}(\kappa_{\pi}(j)) \rvert,
\]
for each $1 \leqslant j \leqslant m$ where $\kappa_{\pi}(j) \in \mathbb{C}$ are the spectral parameters. Let $s \in \mathbb{C}$ such that \( L(s, \pi) \ne 0 \). Then
\begin{align}\label{Explicit Formula Equation 1}
\sum_{n \leqslant X} \frac{\Lambda_{\pi}(n)}{n^{s}} &=-\frac{L^{\prime}}{L}(s, \pi )+ \delta_{\pi}\frac{X^{1-s}}{1-s} -\sum_{\lvert \Im(\rho_{\pi}-s) \rvert <T} \frac{X^{\rho_{\pi}-s}}{\rho_{\pi}-s}+\underset{\substack{k \in \mathbb{N}\cup \{0\} \\ 1 \leqslant j \leqslant m}}{\sum \sum} \frac{X^{-2 k-\kappa_{\pi}(j)-s}}{2 k+\kappa_{\pi}(j)+s} + E_{\pi}(X,T,s),
\end{align}
where $\delta_{\pi}=1$ if $L(s,\pi)=\zeta(s)$ and zero otherwise,
\begin{align*}
E_{\pi}(X,T,s) \ll_{\pi} \frac{4^{\lvert \operatorname{Re}(s) \rvert } 
 X^{\theta_m} \log X}{X^{\operatorname{Re}(s)}} \min \left(1, \frac{X}{T\langle X \rangle}\right)+ \frac{4^{\lvert \operatorname{Re}(s) \rvert } \log^2(XT)}{T} \left( \frac{X^{1+\theta_m}}{X^{\operatorname{Re}(s)}}+\frac{1}{\log X} \right),
\end{align*}
and $ \langle X \rangle$ denotes the distance of $X$ from the nearest prime power other than $X$ itself. Furthermore, if $X$ is not a prime power, we have the explicit formula
\begin{align}\label{Explicit Formula Equation 2}
\sum_{n \leqslant X} \frac{\Lambda_{\pi}(n)}{n^{s}} &=-\frac{L^{\prime}}{L}(s, \pi )+ \delta_{\pi}\frac{X^{1-s}}{1-s} -\sum_{\rho_{\pi}} \frac{X^{\rho_{\pi}-s}}{\rho_{\pi}-s} +\underset{\substack{k \in \mathbb{N}\cup \{0\} \\ 1 \leqslant j \leqslant m}}{\sum \sum} \frac{X^{-2 k-\kappa_{\pi}(j)-s}}{2 k+\kappa_{\pi}(j)+s}.
\end{align}
In \eqref{Explicit Formula Equation 1} and \eqref{Explicit Formula Equation 2}, if $s=1$ and $\delta_{\pi}=1$ then the term $X^{1-s}(1-s)^{-1}-L^{\prime} / L(s,\pi)$ should be interpreted as $\log X-\gamma$, where $\gamma$ is the Euler--Mascheroni constant.
\end{lemma}
\begin{proof}
    See Appendix.
\end{proof}
Let $\rho_{\pi}$ and $\rho'_{\pi}$ denote non-trivial zeros of $L(s, \pi)$. Recall from Section \ref{sec 1.1 Motivation}, for $X\geqslant 2$ and $T\geqslant 2$, 
\begin{align*}
\mathcal{S}_{\pi}(X,T,\rho_{\pi}') &= \sum_{\substack{ \lvert \Im{\rho_{\pi}} \rvert \leqslant T}} X^{\rho_{\pi}-\rho_{\pi}'}w(\rho_{\pi}-\rho_{\pi}'),\\
\textrm{and} \quad \mathcal{S}_{\text{av},\pi}(X,T)&= \frac{1}{N_{\pi}( T )}
 \sum_{\substack{\lvert \Im{\rho_{\pi}'} \rvert \leqslant T}} \Re \mathcal{S}_{\pi}(X, T, \rho_{\pi}'),
\end{align*}
where $w(u)$ and $N_{\pi}(T)$ are given by $\eqref{weight}$ and \eqref{numOfZeros}. Fix $X \geqslant 2$. Define the sequence
\begin{align}\label{Defining coefficients}
a_{n} := \begin{cases}
 \left (\dfrac{n}{X} \right )^2, & n \leqslant X \\
 \left (\dfrac{X}{n} \right )^2, & n \geqslant X.
\end{cases}
\end{align}
Equipped with Lemma \ref{lem:explicit_formula_for_L(s,pi)}, we are able to establish the following result.
\begin{lemma}
 \label{Sum over Zeros to Sum over Primes}
Let $m \in \mathbb{N}$ and $\pi \in \mathscr{A}_m$. Let \( X,T \geqslant 2 \) and $\{a_n\}$ be the sequence given by \eqref{Defining coefficients}. Then for all non-trivial zeros \( \rho'_{\pi} = \beta'_{\pi} + i \gamma'_{\pi} \) of $L(s,\pi)$ with \( \lvert \gamma'_{\pi} \rvert \leqslant T \) outside an exceptional set of size at most $O_{\pi}(\sqrt{T}\log T)$, we have
\begin{align}\label{Lemma 2.2 Main Result}
 \sum_{\rho_{\pi}} X^{\rho_{\pi}-\rho'_{\pi}} w(\rho_{\pi}-\rho'_{\pi})= - \sum_{n =1}^{\infty} \frac{a_n\Lambda_{\pi}(n)}{n^{\rho'_{\pi}}}+\frac{m\log (\lvert \gamma'_{\pi} \rvert+3)}{X^2}+ O_{\pi}\bigg(\frac{X^{1 - \beta'_{\pi}}}{T} + \frac{1}{X^2}\bigg),
 \end{align}
\end{lemma}
\begin{proof}
Fix a non-trivial zero \( \rho'_{\pi} = \beta'_{\pi} + i \gamma'_{\pi} \) of $L(s,\pi)$. We may assume that \( \lvert \gamma'_{\pi} \rvert > \sqrt{T} \). For \( X \) not a prime power, \( s \ne 1, \) and \( L(s,\pi) \ne 0 \), Lemma \ref{lem:explicit_formula_for_L(s,pi)} shows that
\begin{align}\label{Explicit Formula for L(s,pi)}
\sum_{n \leqslant X} \frac{\Lambda_{\pi}(n)}{n^{s}} &=-\frac{L^{\prime}}{L}(s, \pi )+ \delta_{\pi}\frac{X^{1-s}}{1-s} -\sum_{\rho_{\pi}} \frac{X^{\rho_{\pi}-s}}{\rho_{\pi}-s} +\underset{\substack{k \in \mathbb{N}\cup \{0\} \\ 1 \leqslant j \leqslant m}}{\sum \sum} \frac{X^{-2 k-\kappa_{\pi}(j)-s}}{2 k+\kappa_{\pi}(j)+s},
\end{align}
where $\delta_{\pi}=1$ if $L(s,\pi)=\zeta(s)$ and zero otherwise. Choosing \( s = 2 + \rho'_{\pi} \) in \eqref{Explicit Formula for L(s,pi)} and using the Dirichlet series for $L'/L(s,\pi)$,
we have
\begin{equation}
 \label{s=5/2 L(s,pi)}
 \sum_{\rho_{\pi}} \frac{X^{\rho_{\pi}-\rho'_{\pi}}}{-2 + (\rho_{\pi}-\rho'_{\pi})}
 = \sum_{n > X} \frac{a_n\Lambda_{\pi}(n)}{n^{\rho'_{\pi}}}
 - \delta_{\pi}\frac{X^{1 - \rho'_{\pi}}}{1 + \rho'_{\pi}}
 +\underset{\substack{k \in \mathbb{N}\cup \{0\} \\ 1 \leqslant j \leqslant m}}{\sum \sum} \frac{X^{-2 k-\kappa_{\pi}(j)-\rho'_{\pi}}}{2 k+2+\kappa_{\pi}(j)+\rho'_{\pi}}.
\end{equation}
Similarly, choosing \( s = -2 + \rho'_{\pi} \) in \eqref{Explicit Formula for L(s,pi)} yields
\begin{align}
 \label{s=-3/2 L(s,pi)}
 \sum_{\rho_{\pi}} \frac{X^{\rho_{\pi}-\rho'_{\pi}}}{2 + (\rho_{\pi}-\rho'_{\pi})}
 &= - \sum_{n \leqslant X} \frac{a_n\Lambda_{\pi}(n)}{n^{\rho'_{\pi}}}
 - X^{-2} \frac{L'}{L}( \rho'_{\pi} - 2 , \pi) + \delta_{\pi}\frac{X^{1 - \rho'_{\pi}}}{3 - \rho'_{\pi}}
 + \underset{\substack{k \in \mathbb{N}\cup \{0\} \\ 1 \leqslant j \leqslant m}}{\sum \sum} \frac{X^{-2 k-\kappa_{\pi}(j)-\rho'_{\pi}}}{2 k-2+\kappa_{\pi}(j)+\rho'_{\pi}}.
\end{align}
Subtracting \eqref{s=5/2 L(s,pi)} from \eqref{s=-3/2 L(s,pi)}, we arrive at
\begin{align}
 \sum_{\rho_{\pi}} X^{\rho_{\pi}-\rho'_{\pi}} w(\rho_{\pi}-\rho'_{\pi})
 & = - \sum_{n =1}^{\infty} \frac{a_n\Lambda_{\pi}(n)}{n^{\rho'_{\pi}}} - X^{-2} \frac{L'}{L}( \rho'_{\pi} - 2, \pi)
 + \delta_{\pi}\frac{4X^{1 - \rho'_{\pi}}}{( 3 - \rho'_{\pi} )( 1 + \rho' _{\pi})} \notag \\
&\quad + \underset{\substack{k \in \mathbb{N}\cup \{0\} \\ 1 \leqslant j \leqslant m}}{\sum \sum} \frac{4X^{-2 k-\kappa_{\pi}(j)-\rho'_{\pi}}}{(2 k+\kappa_{\pi}(j)+\rho'_{\pi})^2-4} \label{Estimating the Sums over the Residues}.
 \end{align}
Both sides of \eqref{Estimating the Sums over the Residues} are continuous when $X$ is a prime power so we no longer exclude such cases. Note that if $a \leqslant \sigma \leqslant b<0$ in a fixed vertical strip $[a,b] \times \mathbb{R}$, then
\begin{equation}\label{Derivative Estimate}
-\frac{L'}{L}(\sigma+it, \pi) = m\log (\lvert t \rvert+3) +O_{\pi}(1),
\end{equation}
uniformly in $\sigma$. This follows by taking the logarithmic derivative of the functional equation \eqref{Functional Equation} and applying Stirling's formula. Substituting \eqref{Derivative Estimate} into \eqref{Estimating the Sums over the Residues}, we obtain
\begin{align*}
 \sum_{\rho_{\pi}} X^{\rho_{\pi}-\rho'_{\pi}} w(\rho_{\pi}-\rho'_{\pi})= - \sum_{n =1}^{\infty} \frac{a_n\Lambda_{\pi}(n)}{n^{\rho'_{\pi}}}+\frac{m\log( \lvert \gamma'_{\pi} \rvert+3)}{X^2}+ O_{\pi}\bigg(\frac{X^{1 - \beta'_{\pi}}}{T} + \frac{1}{X^2}\bigg)
 \end{align*}
which completes the proof.
\end{proof}
\begin{remark}
The secondary main term $mX^{-2}\log( \lvert \gamma'_{\pi} \rvert+3)$ in \eqref{Lemma 2.2 Main Result} is the key reason behind the distortion in the shape of the Pair Correlation Surface when $X$ is small compared to $T$. We address this situation in more detail in Sections \ref{sec: unconditional 1} and \ref{sec : unconditional 2}.
\end{remark}
Lemma \ref{Sum over Zeros to Sum over Primes} essentially allows us to replace $\mathcal{S}_{\pi}(X,T,\rho'_{\pi})$ by a sum over primes and prime powers.
\begin{lemma}\label{Connecting S to sum over primes}
Let $m \in \mathbb{N}$ and $\pi \in \mathscr{A}_m$. Let \( X,T \geqslant 2 \) and $\{a_n\}$ be the sequence given by \eqref{Defining coefficients}. Then for all non-trivial zeros \( \rho'_{\pi} = \beta'_{\pi} + i \gamma'_{\pi} \) of $L(s,\pi)$ with \( \lvert \gamma'_{\pi} \rvert \leqslant T \) outside an exceptional set of size at most $O_{\pi}(\sqrt{T}\log T)$, we have
\begin{equation*}
\mathcal{S}_{\pi}(X,T, \rho'_{\pi}) = - \sum_{n =1}^{\infty} \frac{a_n\Lambda_{\pi}(n)}{n^{\rho'_{\pi}}}+\frac{m \log( \lvert \gamma'_{\pi} \rvert+3)}{X^2}+ O_{\pi}\bigg(\frac{X^{1 - \beta'_{\pi}}}{T} + \frac{1}{X^2}+\frac{X\log T}{\sqrt{T}}\bigg).
\end{equation*}
Moreover, assuming the Riemann Hypothesis for $L(s,\pi)$, we have
\begin{equation*}
\mathcal{S}_{\pi}(X,T, \rho'_{\pi}) = - \sum_{n =1}^{\infty} \frac{a_n\Lambda_{\pi}(n)}{n^{\frac{1}{2}+i\gamma'_{\pi}}}+\frac{m \log( \lvert \gamma'_{\pi} \rvert+3)}{X^2}+ O_{\pi}\bigg(\frac{\sqrt{X}}{T} + \frac{1}{X^2}+\frac{\log T}{\sqrt{T}}\bigg).
\end{equation*}
\end{lemma}
\begin{proof} Except for $O_{\pi}(\sqrt{T}\log T)$ non-trivial zeros $\rho'_{\pi}=\beta'_{\pi} + i \gamma'_{\pi}$ of $L(s,\pi)$, we may assume $\sqrt{T} \leqslant \lvert \gamma'_{\pi} \rvert \leqslant T-\sqrt{T}$. For such $\rho'_{\pi}$, we write
\begin{align}\label{Breaking Zeros into Intervals}
\sum_{\rho_{\pi}}&X^{\rho_{\pi}-\rho'_{\pi}}w(\rho_{\pi}-\rho'_{\pi}) -\mathcal{S}_{\pi}(X,T, \rho'_{\pi}) \notag \\
&=\sum_{m=0}^{\infty} \sum_{\substack{\gamma_{\pi} -T \in (m,m+1] }} \frac{X^{\rho_{\pi}-\rho'_{\pi}}}{4+(\gamma_{\pi}-\gamma'_{\pi})^2}+\sum_{m=0}^{\infty} \sum_{\substack{\gamma_{\pi} +T \in (-m,-m-1] }} \frac{X^{\rho_{\pi}-\rho'_{\pi}}}{4+(\gamma_{\pi}-\gamma'_{\pi})^2}.
\end{align}
For the first sum on the right hand side of \eqref{Breaking Zeros into Intervals}, i.e. the case when $\gamma_{\pi}>T$, we use estimates on the local density of zeros of $L(s,\pi)$ (see \cite[Prop. 5.7]{IK2004}) to obtain the bound
\begin{align}
&\ll_{\pi} \frac{X\log (T+1)}{4+(T-\gamma')^2}+\sum_{m=1}^{\infty} \sum_{\substack{\gamma_{\pi} -T \in (m,m+1] }} \frac{X^{\rho_{\pi}-\rho'_{\pi}}}{4+(\gamma_{\pi}-\gamma'_{\pi})^2} \notag \\
&\ll_{\pi} \frac{X\log T}{T}+ \int_{T+1}^{\infty} \frac{X\log u}{(u-\gamma')^2} \diff u \ll_{\pi} \frac{X\log T}{\sqrt{T}}. \label{Extending the Sum over all Zeros}
\end{align}
A similar argument holds for the second sum on the right hand side of \eqref{Breaking Zeros into Intervals}, i.e. the case when $\gamma_{\pi}<-T$. Moreover, assuming RH for $L(s,\pi)$, we have $\lvert X^{\rho_{\pi}-\rho'_{\pi}} \rvert \leqslant 1$ and the bound for \eqref{Extending the Sum over all Zeros} is $\ll_{\pi} T^{-1/2} \log T$. Now applying Lemma \ref{Sum over Zeros to Sum over Primes}, the desired conclusion follows.
\end{proof}
\subsection{Landau--Gonek Generalizations for $L(s,\pi)$} We shall require a refined form of Gonek's result \cite{gonekExplicitFormulaLandau1993} on sums over zeros of $L$-functions,
due to Murty--Zaharescu \cite{murtyExplicitFormulasPair2002} and Ford--Soundararajan--Zaharescu \cite{fordDistributionImaginaryParts2009}.
\begin{lemma}\label{KZ Gonek}
Let $m \in \mathbb{N}, \pi \in \mathscr{A}_m$ and $\theta_{m} \in [0, \frac{1}{2}-\frac{1}{m^2+1}]$ be an admissible exponent towards the Ramanujan conjecture for $L(s,\pi)$. Let $X >1, T \geqslant 2$ and denote by $n_X$ the nearest prime power to $X$. Then for any $\varepsilon>0$
\begin{align}\label{fzEquation}
\sum_{\lvert \gamma_{\pi} \rvert \leqslant T} X^{\rho_{\pi}}=-\frac{\Lambda_{\pi}(n_X)}{\pi} \frac{\sin(T \log (X / n_X))}{\log (X / n_X)}+O_{\pi, \varepsilon}\bigg(X^{1+\theta_m+\varepsilon}\log^2 T+\frac{\log T}{\log X}\bigg),
\end{align}
where if $X=n_X$, the first term on the right hand side of \eqref{fzEquation} is $-\frac{T\Lambda_{\pi}(X)}{\pi}$.
\end{lemma}
\begin{proof} See \cite[Proposition 1]{murtyExplicitFormulasPair2002} and  \cite[Lemma 2]{fordDistributionImaginaryParts2009}.
\end{proof}
A quick consequence of the above lemma are the following two corollaries.
\begin{corollary}\label{Gonek assuming RH}
Let $m \in \mathbb{N}, \pi \in \mathscr{A}_m$ and $\theta_{m} \in [0, \frac{1}{2}-\frac{1}{m^2+1}]$ be an admissible exponent towards the Ramanujan conjecture for $L(s,\pi)$. Let $T \geqslant 2$ and suppose $a,b \in \mathbb{N}$ with $1 \leqslant b<a\leqslant T^{\xi}$ for some $\xi>0$. Then for any $\varepsilon>0$
\begin{align}\label{Gonek RH Equation}
\sum_{\lvert \gamma_{\pi} \rvert \leqslant T}\left(\frac{a}{b}\right)^{\rho_{\pi}}=-\frac{T}{\pi} \Lambda_{\pi}\left(\frac{a}{b}\right)+O_{\pi,\xi, \varepsilon}(a^{1+\theta_m+\varepsilon}\log^2 T).
\end{align}
Moreover, assuming the Riemann Hypothesis for $L(s,\pi)$, for any $\varepsilon>0$
\begin{align}\label{Gonek RH Equation 1}
\sum_{\lvert \gamma_{\pi} \rvert \leqslant T}\left(\frac{a}{b}\right)^{i\gamma_{\pi}}=-\frac{T}{\pi} \sqrt{\frac{b}{a}}\Lambda_{\pi} \left(\dfrac{a}{b}\right)+O_{\pi, \xi, \varepsilon}((ab)^{\frac{1}{2}+\theta_m+\varepsilon}\log^2 T).
\end{align}
\end{corollary}
\begin{proof}
First assume that $a/b$ is not a prime power. Choose $X=a/b$ in \eqref{fzEquation}. Then $X \neq n_X$ and $\lvert \log (X/n_X) \rvert \gg a^{-1}$. Therefore the first term on the right hand side of \eqref{fzEquation} is $\ll_{\pi, \xi} a \cdot n_X^{\theta_m} \log T.$ The error term in \eqref{fzEquation} in this case is
\[
\ll_{\pi, \xi, \varepsilon} \left(\frac{a}{b}\right)^{1+\theta_m+\varepsilon}\log^2 T + a \log T,
\]
for any $\varepsilon>0$ since $\log X \geqslant a^{-1}$. Therefore, we obtain
\[
\sum_{\lvert \gamma_{\pi} \rvert \leqslant T}\left(\frac{a}{b}\right)^{\rho_{\pi}} \ll_{\pi, \xi, \varepsilon} a^{1+\theta_m+\varepsilon}\log^2 T.
\]
When $a/b$ is a prime power, the main term in \eqref{fzEquation} is $- \frac{T\Lambda_{\pi}(a/b)}{\pi}$ and the proof of \eqref{Gonek RH Equation} is complete. Assuming RH for $L(s,\pi)$, we have $\rho_{\pi} = \frac{1}{2}+i\gamma_{\pi}$ and thus \eqref{Gonek RH Equation 1} follows from \eqref{Gonek RH Equation}.
\end{proof}
\begin{corollary}\label{when b>a}
Let $m \in \mathbb{N}, \pi \in \mathscr{A}_m$ and $\theta_{m} \in [0, \frac{1}{2}-\frac{1}{m^2+1}]$ be an admissible exponent towards the Ramanujan conjecture for $L(s,\pi)$. Let $T \geqslant 2$ and suppose $a,b \in \mathbb{N}$ with $1 \leqslant a<b\leqslant T^{\xi}$ for some $\xi>0$.  Then for any $\varepsilon>0$
\begin{align}\label{Gonek RH Equation P<Q}
\sum_{\lvert \gamma_{\pi} \rvert \leqslant T} \left(\frac{a}{b}\right)^{\rho_{\pi}}=-\frac{T}{ \pi } \frac{a}{b} \Lambda_{\widetilde{\pi}}\left(\frac{b}{a}\right)+O_{\pi,\xi, \varepsilon}(ab^{\theta_m+\varepsilon}\log^2 T).
\end{align}
Moreover, assuming the Riemann Hypothesis for $L(s,\pi)$, for any $\varepsilon>0$
\begin{align}\label{Gonek RH Equation P<Q 2}
\sum_{\lvert \gamma_{\pi} \rvert \leqslant T}\left(\frac{a}{b}\right)^{i\gamma_{\pi}}=-\frac{T}{\pi} \sqrt{\frac{a}{b}}\Lambda_{\widetilde{\pi}} \left(\dfrac{b}{a}\right)+O_{\pi, \xi, \varepsilon}((ab)^{\frac{1}{2}+\theta_m+\varepsilon}\log^2 T).
\end{align}
\end{corollary}
\begin{proof}
Note that 
\begin{equation*}
 \sum_{\lvert \gamma_{\pi} \rvert \leqslant T} \left(\frac{a}{b}\right)^{\rho_{\pi}} = \frac{a}{b} \sum_{\lvert \gamma_{\pi} \rvert \leqslant T} \left(\frac{b}{a}\right)^{1-\rho_{\pi}} = \frac{a}{b}\sum_{\lvert \gamma_{\widetilde{\pi}} \rvert \leqslant T} \left(\frac{b}{a}\right)^{\rho_{\widetilde{\pi}}}.
\end{equation*}
Applying Corollary \ref{Gonek assuming RH}, we obtain
\begin{align}\label{Gonek RH Equation P<Q}
\sum_{\lvert \gamma_{\pi} \rvert \leqslant T} \left(\frac{a}{b}\right)^{\rho_{\pi}}=-\frac{T}{ \pi } \frac{a}{b} \Lambda_{\widetilde{\pi}}\left(\frac{b}{a}\right)+O_{\pi,\xi, \varepsilon}(ab^{\theta_m+\varepsilon}\log^2 T).
\end{align}
Assuming RH for $L(s,\pi)$, we have $\rho_{\pi} = \frac{1}{2}+i\gamma_{\pi}$ and thus \eqref{Gonek RH Equation P<Q 2} follows from \eqref{Gonek RH Equation P<Q}.
\end{proof}
\subsection{Some Key Lemmas and Estimates.} We record the following result for later use.
\begin{lemma}\label{Complex Number Identity}
Let $z \in \mathbb{C}$ and $k \in \mathbb{N}$ . Then
\begin{align*}
\Re(z)^{2k} &= \frac{1}{2^{2 k}}\binom{2k}{k} \Re(z^{k}\overline{z}^{k}) + \frac{1}{2^{2k-1}} \sum_{i=0}^{k-1}\binom{2k}{i} \Re(z^{2k-i}\overline{z}^{i}),\\
\textrm{and} \quad \Re(z)^{2k+1} &= \frac{1}{4^k} \sum_{i=0}^{k}\binom{2k+1}{i} \Re(z^{2k+1-i}\overline{z}^{i}).
\end{align*}
\end{lemma}
\begin{proof}
The proof follows by writing $\Re(z)=\frac{1}{2}(z+\overline{z})$ and applying the Binomial Theorem.
\end{proof}

It is essential for later analysis to understand variants of the following sum 
\begin{equation}\label{def:A(x)}
\sum_{n \leqslant X} \frac{|\Lambda_{\pi}(n)|^2}{n}.
\end{equation} 
We have the following lemmas.
\begin{lemma}\label{lem:prime_number_theorem}
Let $m \in \mathbb{N}$ and $\pi \in \mathscr{A}_{m}.$ Let $X \geqslant 2$. Then we have 
\begin{align*}
\sum_{n \leqslant X} \frac{|\Lambda_{\pi}(n)|^2}{n} = \frac{\log ^2 X}{2} + O_{\pi}(1).
\end{align*}
\end{lemma}
\begin{proof}
As in \cite[Eq. (2.25)]{rudnickZerosPrincipalLfunctions1996}, we use the upper bound for $\theta_m$ and Hypothesis \textbf{H}$_{\pi}$ to show that the contribution from the prime powers is $O_{\pi}(1).$ Hence we are just concerned with the sum over primes. When $\pi$ is self-contragredient, this follows from \cite[Theorem 1.1]{GillespieYe}.
Due to the zero-free region of \citeauthor{ThornerHumphriesHoheisel} \cite[Theorem 2.1] {ThornerHumphriesHoheisel}, the self-contragredient condition is no longer necessary. 
\end{proof}
\begin{lemma}\label{Rudnick-Sarnak Lemma}
Let $m \in \mathbb{N}$ and $ \pi \in \mathscr{A}_m$. Let $2 \leqslant X \leqslant V$ and suppose $\{a_n\}$ be the sequence in \eqref{Defining coefficients}. Then 
\begin{align*}
    \sum_{n \leqslant V} \frac{a_n \lvert \Lambda_{\pi}(n) \rvert^2}{n} &= \log X- \frac{\log V}{2}\left(\frac{X}{V}\right)^2+O_{\pi}(1),\\
\textrm{and} \quad 
    \sum_{n \leqslant V} \frac{a_n^2 \lvert \Lambda_{\pi}(n) \rvert^2}{n} &= \frac{\log X}{2}- \frac{\log V}{4}\bigg(\frac{X}{V}\bigg)^4+O_{\pi}(1).
\end{align*}
\end{lemma}
\begin{proof}
The proof follows by partial summation and Lemma~\ref{lem:prime_number_theorem}.
\end{proof}
\begin{lemma}\label{Convergent Sums}
Let $m \in \mathbb{N}$ and $\pi \in \mathscr{A}_{m}.$ For any $j \in \mathbb{N}$ with $j\geqslant 2$, we have
\begin{align*}
\sum_{p \leqslant X} \frac{|\Lambda_{\pi}(p)|^{2j}}{p^j} < \infty. 
\end{align*} 
\end{lemma}
\begin{proof}
By \eqref{Luo-Rudnick-Sarnak Bound}, we have
\begin{align*}
\sum_{p\leqslant x} \frac{|\Lambda_{\pi}(p)|^{2j}}{p^j} \ll \sum_{p\leqslant x} \frac{|\Lambda_{\pi}(p)|^{2} p^{2\theta_m(j-1)}}{p^j} \ll \sum_{p \leqslant x} \frac{|\lambda_{\pi}(p)|^2}{p^{1+\frac{2(j-1)}{m^2+1}-\epsilon}} \ll 1,
\end{align*}
using the Rankin--Selberg framework, see for instance \cite[Proof of Theorem 5.41]{IK2004}. 
\end{proof}
\section{Moments of Sums over Primes: Part I}\label{sec: Conditional 1} 
By Lemma \ref{Connecting S to sum over primes}, it suffices to focus on sums over primes to study the distribution of $\mathcal{S}_{\pi}(X,T,\rho'_{\pi})$. Fix $\xi>1$. Suppose $T\geqslant 2$ and let $\alpha, X,V$ satisfy
\begin{align}
\label{Defining Parameters}
0<\alpha< \frac{1}{m},\quad X=T^{\alpha m} \quad \textrm{and} \quad V=X^{\xi}.
\end{align}
We begin by estimating certain mixed moments of sums over primes.
\begin{lemma}\label{Moment Calculations Assuming RH Step 1}
Let $m \in \mathbb{N}, \pi \in \mathscr{A}_m$ and $\theta_{m} \in [0, \frac{1}{2}-\frac{1}{m^2+1}]$ be an admissible exponent towards the Ramanujan conjecture for $L(s,\pi)$. Assume the Riemann Hypothesis for $L(s,\pi)$. Let $k, \ell$ be non-negative integers with $k+\ell=2d$ for some $d \in \mathbb{N}$. Let $T \geqslant 2$ and $\alpha, \xi, X, V$ be as in \eqref{Defining Parameters} satisfying 
\[\alpha < \frac{1}{2dm(1+\xi \theta_m)}.
\]
Let $\{a_n\}$ be the sequence defined by \eqref{Defining coefficients}. Then
\begin{align*}
\frac{1}{N_{\pi}(T)}\sum_{\lvert \gamma_{\pi} \rvert \leqslant T}\bigg(\sum_{p \leqslant V} \frac{a_p \Lambda_{\pi}(p) }{p^{\frac{1}{2}+i\gamma_{\pi}}}\bigg)^k\bigg(\sum_{q \leqslant V} \frac{a_q \overline{\Lambda_{\pi}(q)}}{q^{\frac{1}{2}-i\gamma_{\pi}}}\bigg)^\ell &=\boldsymbol{\delta}_{k\ell} d!\bigg( \sum_{p \leqslant V} \frac{a_p^2 \lvert \Lambda_{\pi}(p)\rvert^2}{p}\bigg)^d+O( \max \{1, (\log T)^{d-2}\}).
\end{align*}
Here $p,q$ run over primes up to $V$ and $\boldsymbol{\delta}_{k\ell}=1$ if $k=\ell$, and zero otherwise. The implied constant in the error term depends at most on $\pi, d, \alpha$ and $\xi$.
\end{lemma}
\begin{proof}
All implied constants in the proof may depend at most on $\pi, d, \alpha$ and $\xi$. Let
\begin{equation}\label{A_k,l Notation Assuming RH}
\mathcal{A}_{k,\ell} := \frac{1}{N_{\pi}(T)}\sum_{\lvert \gamma_{\pi} \rvert \leqslant T} \bigg(\sum_{p \leqslant V} \frac{a_p \Lambda_{\pi}(p)}{p^{\frac{1}{2}+i\gamma_{\pi}}}\bigg)^k\bigg(\sum_{q \leqslant V} \frac{a_q \overline{\Lambda_{\pi}(q)}}{q^{\frac{1}{2}-i\gamma_{\pi}}}\bigg)^\ell.
\end{equation}
and $\eta := \eta(\pi, d, \alpha, \xi) = 2\alpha dm(1+\xi \theta_m)$. Note that $\eta<1$. We break our proof into two cases.\\

\noindent \textbf{Case 1 : When $k \neq \ell$.}
Assume $k > \ell$. Consider the situation when $\ell=0$. Then
\begin{align}
\mathcal{A}_{k,0} = \frac{1}{N_{\pi}(T)} \sum_{p_1,p_2, \dots, p_k \leqslant V} \bigg(\prod_{i=1}^{k}a_{p_i} \Lambda_{\pi}(p_i) \bigg) \sum_{\lvert \gamma_{\pi} \rvert \leqslant T} P^{-\frac{1}{2}-i\gamma_{\pi}}, \notag
\end{align}
where $P=p_1p_2 \cdots p_k$. By Corollary \ref{when b>a},
\begin{align}
\sum_{\lvert \gamma_{\pi} \rvert \leqslant T} P^{-i\gamma_{\pi}} =-\frac{T\Lambda_{\pi}(P)}{\pi\sqrt{P}} + O_{\varepsilon}(P^{\frac{1}{2}+\theta_m+\varepsilon} \log^2 T). \notag
\end{align}
for any $\varepsilon>0$. Note $\Lambda_{\pi}(P)=0$ for those tuples $(p_1,p_2,\dots, p_k)$ where all $p_i$'s are not equal. Hence we get
\begin{align}\label{1st Gonek Application}
\mathcal{A}_{k,0} &=-\frac{T}{\pi N_{\pi}(T)}\sum_{p \leqslant V} \frac{a_p^k \Lambda^k_{\pi}(p) \Lambda_{\pi}(p^k)}{p^k} + O_{\varepsilon}\bigg{(} \frac{V^{2d\theta_m+\varepsilon}\log T}{T} \sum_{p_1,\dots, p_k \leqslant V} \prod_{i=1}^{k}a_{p_i} \lvert \Lambda_{\pi}(p_i) \rvert \bigg{)}.
\end{align}
By Cauchy--Schwarz and Lemma \ref{Rudnick-Sarnak Lemma}, the error term in \eqref{1st Gonek Application} is
\[\ll_{\varepsilon} \frac{V^{2d\theta_m+\varepsilon} \log T}{T} \bigg( \sum_{p \leqslant V} pa_{p} \bigg )^{d} \bigg ( \sum_{p \leqslant V} \frac{a_p \lvert \Lambda_{\pi}(p) \rvert^2}{p}\bigg )^{d} \ll_{\varepsilon} T^{\eta-1+\varepsilon}.
\]
Since $k \geqslant 2$, by Cauchy--Schwarz, Hypothesis \textnormal{\textbf{H}$_{\pi}$}  and Lemma \ref{Convergent Sums}, the first term on the right hand side of \eqref{1st Gonek Application} is $\ll (\log T)^{-1}$. Hence by choosing $\varepsilon>0$ sufficiently small in terms of $\eta$, we obtain
\begin{equation}\label{Case 1, when l=0}
\mathcal{A}_{k,0} \ll (\log T)^{-1}. 
\end{equation}
\par
Now assume that $\ell > 0$. Then
\begin{align}
\mathcal{A}_{k,\ell}=\frac{1}{N_{\pi}(T)}\sum_{\substack{p_1, p_2,\dots, p_k \leqslant V \\ q_1,q_2,\dots, q_\ell \leqslant V}} \dfrac{(\prod_{i=1}^k a_{p_i} \Lambda_{\pi}(p_i))(\prod_{j=1}^\ell a_{q_j} \overline{\Lambda_{\pi}(q_j)} )}{(PQ)^{\frac{1}{2}}}\sum_{\lvert \gamma_{\pi} \rvert \leqslant T} \left(\frac{Q}{P} \right )^{i \gamma_{\pi}}, \notag
\end{align}
where $P=p_1p_2 \cdots p_k$ and $Q=q_1q_2 \cdots q_\ell$. Note that $P \neq Q$. We break the sum $\mathcal{A}_{k,\ell}$ by writing
\[ \mathcal{A}_{k,\ell} = \mathcal{A}_{k,\ell}(P>Q)+\mathcal{A}_{k,\ell}(P<Q),
\]
depending on whether $P>Q$ or $P<Q$. Since $Q/P \notin \mathbb{Z}$, using Corollary $\ref{Gonek assuming RH}$, we have
\begin{equation}\label{Q>P Case}
\mathcal{A}_{k,\ell}(P<Q) \ll_{\varepsilon} T^{\eta-1+\varepsilon}.
\end{equation}
When $P>Q$, using Corollary \ref{when b>a}, it follows that
\begin{align*}
\sum_{\lvert \gamma_{\pi} \rvert \leqslant T} \bigg(\frac{Q}{P} \bigg )^{i \gamma_{\pi}} = -\frac{T}{\pi}\sqrt{\frac{Q}{P}} \Lambda_{\pi} \bigg (\frac{P}{Q} \bigg )+ O_{\varepsilon} ((PQ)^{\frac{1}{2}+\theta_m+\varepsilon} \log^2 T).
\end{align*}
Note that $\Lambda_{\pi} (P/Q) \neq 0$ if and only if $P/Q$ is a prime power. Hence we deduce that
\begin{align*}
\mathcal{A}_{k,\ell}(P>Q)\ll_{\varepsilon} \frac{1}{\log T}\bigg(\sum_{p \leqslant V}\frac{(a_p \Lambda_{\pi}(p))^{k-\ell} \Lambda_{\pi}(p^{k-\ell})}{{p}^{k-\ell}}\bigg)\bigg(\sum_{q \leqslant V}\frac{a_q^2 \lvert \Lambda_{\pi}(q) \rvert ^2}{q}\bigg)^\ell +T^{\eta-1+\varepsilon}.
\end{align*}
Since $k-\ell \geqslant 2$, the sum over $p$ is $O(1)$ following the same argument how we established \eqref{Case 1, when l=0}. The sum over $q$ is $\ll (\log X)^{\ell}$ by Lemma \ref{Rudnick-Sarnak Lemma}. Therefore choosing $\varepsilon>0$ sufficiently small in terms of $\eta$, we obtain
\begin{align} \label{Case 1, when l neq 0}
   \mathcal{A}_{k,\ell}(P>Q) \ll (\log T)^{\ell-1}.
\end{align}
In conclusion, from \eqref{Case 1, when l=0}, \eqref{Q>P Case} and \eqref{Case 1, when l neq 0}, the overall contribution in Case 1 is $\ll (\log T)^{d-2}$. A concomitant argument follows when $\ell >k$.\\

\noindent \textbf{Case 2 : When $k=\ell$.} For $k=\ell=d$, we have
\begin{align}
\mathcal{A}_{d,d}=\frac{1}{N_{\pi}(T)}\sum_{\substack{p_1, p_2,\dots, p_d \leqslant V \\ q_1,q_2,\dots, q_d \leqslant V}} \dfrac{(\prod_{i=1}^d a_{p_i} \Lambda_{\pi}(p_i))(\prod_{j=1}^d a_{q_j} \overline{\Lambda_{\pi}(q_j)})}{(PQ)^{\frac{1}{2}}}\sum_{\lvert \gamma_{\pi} \rvert \leqslant T} \left(\frac{Q}{P} \right )^{i \gamma_{\pi}},
\end{align}
where $P=p_1p_2 \cdots p_d$ and $Q=q_1q_2 \cdots q_d$. The main terms occur when $P =Q$. If all the $p_i$'s are distinct, there are exactly $d!$ ways to rewrite $P$ as $Q$. Applying Lemma \ref{Rudnick-Sarnak Lemma} and Lemma \ref{Convergent Sums}, the contribution from the case $P=Q$ is
\begin{align}\label{Inner Sum Estimate}
\sum_{\substack{p_1, p_2,\dots, p_d \leqslant V \\ q_1,q_2,\dots, q_d \leqslant V \\ P=Q}}& \dfrac{(\prod_{i=1}^d a_{p_i} \Lambda_{\pi}(p_i))(\prod_{j=1}^d a_{q_j} \overline{\Lambda_{\pi}(q_j)})}{P} \notag \\
&= d! \sum_{p_1, p_2, \dots, p_d \leqslant
 V}\frac{(\prod_{i=1}^{d} a_{p_i}^2 \lvert \Lambda_{\pi}(p_i) \rvert ^2 )}{P} \notag \\
 &\hspace{0.5cm}- \sum_{j=2}^{d}\bigg \{ d! \bigg (1-\frac{1}{j!} \bigg )\bigg (\sum_{p \leqslant
 V}\frac{ a_p^{2j} \lvert \Lambda_{\pi}(p) \rvert ^{2j}}{p^j} \bigg ) \sum_{\substack{p_1, p_2, \dots, p_{d-j} \leqslant V \\ p_i \neq p}}\frac{ \prod_{i=1}^{d-j} a_{p_i}^2 \lvert \Lambda_{\pi}(p_i) \rvert ^2 }{p_1p_2 \cdots p_{d-j}} \bigg \} \notag \\
 &= d!\bigg ( \sum_{p\leqslant
 V}\frac{a_p^2 \lvert \Lambda_{\pi}(p) \rvert ^2}{p} \bigg )^d + O( (\log T)^{d-2} ).
\end{align}
Note that the error term in \eqref{Inner Sum Estimate} doesn't exist when $d=1$.
\par
Now we concentrate on the error terms, which occur when $P \neq Q$. Here $Q/P$ or $P/Q$ is never a prime power. Applying Corollaries \ref{Gonek assuming RH} and \ref{when b>a}, we see that
\begin{equation}\label{P not equal Q}
\mathcal{A}_{d,d}(P> Q)+\mathcal{A}_{d,d}(P< Q) \ll  T^{\eta-1+\varepsilon}. 
\end{equation}
Choosing $\varepsilon>0$ sufficiently small in terms of $\eta$, from \eqref{Inner Sum Estimate} and \eqref{P not equal Q}, the overall contribution in Case 2 is
\begin{align}
d! \bigg ( \sum_{p\leqslant
 V}\frac{a_p^2 \lvert \Lambda_{\pi}(p) \rvert ^2 }{p} \bigg )^d + O (\max\{1,(\log T)^{d-2}\} ). \notag
\end{align}
Combining the two cases, we arrive at our desired conclusion.
\end{proof}
Our next lemma follows directly from Lemma \ref{Moment Calculations Assuming RH Step 1} combined with Lemma \ref{Complex Number Identity}.
\begin{lemma}\label{Moment Calculations Assuming RH Step 2}
Let $m \in \mathbb{N}, \pi \in \mathscr{A}_m$ and $\theta_{m} \in [0, \frac{1}{2}-\frac{1}{m^2+1}]$ be an admissible exponent towards the Ramanujan conjecture for $L(s,\pi)$. Assume the Riemann Hypothesis for $L(s,\pi)$. Let $d \in \mathbb{N}$, $T \geqslant 2$ and $\alpha, \xi, X, V$ be as in \eqref{Defining Parameters} satisfying 
\[\alpha < \frac{1}{2dm(1+\xi \theta_m)}.
\]
Let $\{a_n\}$ be the sequence defined by \eqref{Defining coefficients}. Then
\begin{align*}
\frac{1}{N_{\pi}(T)}\sum_{\lvert \gamma_{\pi} \rvert \leqslant T} \bigg(\Re \sum_{p \leqslant V} \frac{a_p \Lambda_{\pi}(p)}{p^{\frac{1}{2}+i \gamma_{\pi}}}\bigg)^{2d}&=\boldsymbol{\mu}_{2d}\bigg(\frac{1}{2}\sum_{p \leqslant V} \frac{a_p^2  \lvert \Lambda_{\pi}(p) \rvert ^2 }{p}\bigg)^d+O( \max \{1, (\log T)^{d-2}\} ).
\end{align*}
where $\boldsymbol{\mu}_{2d}$ is defined by \eqref{Defining Normal Distribution Coefficients}. Also, if
\[ \alpha < \frac{1}{m(2d-1)(1+\xi \theta_m)},
\]
then we have
\begin{align*}
\frac{1}{N_{\pi}(T)}\sum_{\lvert \gamma_{\pi} \rvert \leqslant T} \bigg(\Re \sum_{p \leqslant V} \frac{a_p \Lambda_{\pi}(p)}{p^{\frac{1}{2}+i\gamma_{\pi}}}\bigg)^{2d-1} \ll (\log T)^{d-1}.
\end{align*}
The implied constants in the error term and the $\ll$ notation depend at most on $\pi, d, \alpha$ and $\xi$.
\end{lemma}
\begin{proof}
All implied constants in the proof may depend at most on $\pi,d,\alpha$ and $\xi$. First consider the case when the exponent is even. Using Lemma \ref{Complex Number Identity}, we have
\begin{align*}
\frac{1}{N_{\pi}(T)}\sum_{\lvert \gamma_{\pi} \rvert \leqslant T} \bigg(\Re \sum_{p \leqslant V} \frac{a_p \Lambda_{\pi}(p)}{p^{\frac{1}{2}+i \gamma_{\pi}}}\bigg)^{2d}= \frac{1}{2^{2 d}}\binom{2d}{d} \Re (\mathcal{A}_{d,d}) + \frac{1}{2^{2d-1}} \sum_{i=0}^{d-1}\binom{2d}{i} \Re (\mathcal{A}_{2d-i,i}),
\end{align*}
where $\mathcal{A}_{k,\ell}$ is given by \eqref{A_k,l Notation Assuming RH}. An application of Lemma \ref{Moment Calculations Assuming RH Step 1} shows that
\begin{align*}
\frac{1}{N_{\pi}(T)} \sum_{\lvert \gamma_{\pi} \rvert \leqslant T} \bigg(\Re \sum_{p \leqslant V} \frac{a_p \Lambda_{\pi}(p)}{p^{\frac{1}{2}+i \gamma_{\pi}}}\bigg)^{2d} &= \frac{d!}{2^{2d}}\binom{2d}{d} \bigg( \sum_{p \leqslant V} \frac{a_p^2 \lvert \Lambda_{\pi}(p) \rvert ^2}{p}\bigg)^d+O ( \max \{1, (\log T)^{d-2}\} ) \\
&= \boldsymbol{\mu}_{2d} \bigg(\frac{1}{2} \sum_{p \leqslant V} \frac{a_p^2 \lvert \Lambda_{\pi}(p) \rvert ^2}{p}\bigg)^d+O ( \max \{1, (\log T)^{d-2}\} ).
\end{align*}
\par
When the exponent is odd, in Lemma \ref{Complex Number Identity} none of the terms $z^i\overline{z}^j$ are such that $i=j$. Following the proof of Lemma \ref{Moment Calculations Assuming RH Step 1} and writing $k+\ell=2d-1$, the contributions arise only from Case 1. Assuming $k>\ell$, the critical situation is when $k-\ell =1$. The bounds corresponding to \eqref{Case 1, when l=0} and \eqref{Case 1, when l neq 0} are then amplified by an extra factor of $\log T$ following an application of Lemma \ref{Rudnick-Sarnak Lemma}. A similar argument holds when $\ell >k$.
\end{proof}
Now we extend our sums to include prime powers. Define
\begin{align}\label{Defining Mean}
\text{Av}_{\pi}(T, X, V, \{a_n\}) : &= \frac{1}{N_{\pi}(T)}\sum_{\lvert \gamma_{\pi} \rvert \leqslant T}\Re \sum_{n \leqslant V} \frac{a_n \Lambda_{\pi}(n)}{n^{ \rho_{\pi}}}\\
\label{Defining SD}
\textrm{and} \quad \text{Var}_{\pi}(T, X, V, \{a_n\}) : &= \frac{1}{2}\sum_{n \leqslant V} \frac{a_n^2 \lvert \Lambda_{\pi} (n) \rvert ^2}{n} - \text{Av}_{\pi}^2(T, X, V, \{a_n\}).
\end{align}
We first prove asymptotic formulas for $\text{Av}_{\pi}$ and $\text{Var}_{\pi}$ without assuming RH for $L(s,\pi)$.
\begin{lemma}\label{Calculating mean}
Let $m \in \mathbb{N}, \pi \in \mathscr{A}_m$ and $\theta_{m} \in [0, \frac{1}{2}-\frac{1}{m^2+1}]$ be an admissible exponent towards the Ramanujan conjecture for $L(s,\pi)$. Suppose $T \geqslant 2$ and $\alpha, \xi, X, V$ be as in \eqref{Defining Parameters} satisfying 
\[\alpha < \frac{1}{m(1+\xi \theta_m)}.\]
Let $\{a_n\}$ be the sequence defined by \eqref{Defining coefficients}. Then
\begin{align}
\textnormal{Av}_{\pi}(T, X, V, \{a_n\}) &= -\alpha+O\left(\frac{1}{\log T}\right), \label{Mean Estimate}\\
\textrm{and} \quad \textnormal{Var}_{\pi}(T, X, V, \{a_n\}) &= \frac{\alpha m \log T}{4}+O(1) \label{SD Estimate}.
\end{align}
The implied constants in the error term depend at most on $\pi, \alpha$ and $\xi$.
\end{lemma} 
\begin{proof}
All implied constants in the proof may depend at most on $\pi,\alpha$ and $\xi$. By Corollary \ref{when b>a}, we have
\begin{align}\label{Mean Estimate Step 1}
\text{Av}_{\pi} &= -\frac{T}{\pi N_{\pi}(T)} \sum_{n \leqslant V} \frac{a_n \lvert \Lambda_{\pi} (n) \rvert ^2}{n}+O_{\varepsilon} (T^{\alpha m(1+\xi \theta_m)-1+\varepsilon}),
\end{align}
for any $\varepsilon>0$. By Lemma \ref{Rudnick-Sarnak Lemma},
\begin{equation} \label{Mean Estimate Step 2}
\sum_{n \leqslant V} \frac{a_n \lvert \Lambda_{\pi} (n) \rvert ^2}{n} = \alpha m \log T+O(1).
\end{equation}
We choose $\varepsilon>0$ sufficiently small depending on $\pi,\alpha$ and $\xi$. Substituting \eqref{Mean Estimate Step 2} into \eqref{Mean Estimate Step 1} and applying a more precise form of \eqref{Zero Counting General} (see \cite[Theorem 5.8]{IK2004}), \eqref{Mean Estimate} follows. Now using \eqref{Mean Estimate} and Lemma \ref{Rudnick-Sarnak Lemma} in \eqref{Defining SD}, \eqref{SD Estimate} holds.
\end{proof}
We are now ready to estimate our moments involving sums over prime powers.
\begin{lemma}\label{Moment Calculations Assuming RH Step 3}
Let $m \in \mathbb{N}, \pi \in \mathscr{A}_m$ and $\theta_{m} \in [0, \frac{1}{2}-\frac{1}{m^2+1}]$ be an admissible exponent towards the Ramanujan conjecture for $L(s,\pi)$. Assume the Riemann Hypothesis for $L(s,\pi)$. Let $d \in \mathbb{N}$, $T \geqslant 2$ and $\alpha,\xi, X, V$ be as in \eqref{Defining Parameters} satisfying 
\[\alpha < \frac{1}{2dm(1+\xi \theta_m)}.\]
Let $\{a_n\}$ be the sequence defined by \eqref{Defining coefficients}. Then
\begin{align*}
\frac{1}{N_{\pi}(T)}\sum_{\lvert \gamma_{\pi} \rvert \leqslant T}\bigg(\Re \sum_{n \leqslant V} \frac{a_n \Lambda_{\pi}(n)}{n^{\frac{1}{2}+i \gamma_{\pi}}}\bigg)^{2d}=\boldsymbol{\mu}_{2d}\textnormal{Var}_{\pi}^{d} + O((\log T)^{d-\frac{1}{2}}),
\end{align*}
where $\boldsymbol{\mu}_{2d}$ is given by \eqref{Defining Normal Distribution Coefficients} and $\textnormal{Var}_{\pi}$ is defined by \eqref{Defining SD}. Also, if 
\[\alpha < \frac{1}{m(2d-1)(1+\xi \theta_m)},\]
then we have
\begin{align*}
\frac{1}{N_{\pi}(T)}\sum_{\lvert \gamma_{\pi} \rvert \leqslant T}\bigg(\Re \sum_{n \leqslant V} \frac{a_n \Lambda_{\pi}(n)}{n^{\frac{1}{2}+i\gamma_{\pi}}}\bigg)^{2d-1} \ll (\log T)^{d-1}.
\end{align*}
The implied constants in the error term and the $\ll$ notation depend at most on $\pi, d, \alpha$ and $\xi$.
\end{lemma}
\begin{proof}
All implied constants in the proof may depend at most on $\pi, d, \alpha$ and $\xi$. First consider the case when the exponent is even. We write
\begin{align}\label{Binomial Expansion Step}
\frac{1}{N_{\pi}(T)}&\sum_{\lvert \gamma_{\pi} \rvert \leqslant T}\bigg(\Re \sum_{n \leqslant V} \frac{a_n \Lambda_{\pi}(n)}{n^{\frac{1}{2}+i \gamma_{\pi}}}\bigg)^{2d}\notag\\
&= \frac{1}{N_{\pi}(T)}\sum_{j=0}^{2d}\binom{2d}{j} \sum_{\lvert \gamma_{\pi} \rvert \leqslant T} \bigg(\Re \sum_{p \leqslant V} \frac{a_p \Lambda_{\pi}(p)}{p^{\frac{1}{2}+i\gamma_{\pi}}}\bigg)^j\bigg(\Re \sum_{\substack{n \leqslant V \\ n =p^k, k \geqslant2}} \frac{a_n\Lambda_{\pi}(n)}{n^{\frac{1}{2}+i\gamma_{\pi} }}\bigg)^{2d-j}.
\end{align}
By Lemma \ref{Moment Calculations Assuming RH Step 2}, the contribution to \eqref{Binomial Expansion Step} from $j=2d$ is
\begin{equation}\label{Transition 1}
\boldsymbol{\mu}_{2d}\bigg(\frac{1}{2}\sum_{p \leqslant V} \frac{a_p^2  \lvert \Lambda_{\pi}(p) \rvert ^2}{p}\bigg)^d+O( \max \{1, (\log T)^{d-2}\})
\end{equation}
where $\boldsymbol{\mu}_{2d}$ is defined by \eqref{Defining Normal Distribution Coefficients}. Using Hypothesis \textnormal{\textbf{H}$_{\pi}$} (see \cite[Eq. 2.25]{rudnickZerosPrincipalLfunctions1996}),
\begin{equation*}
\textnormal{Var}_{\pi} = \frac{1}{2}\sum_{p \leqslant V} \frac{a_p^2  \lvert \Lambda_{\pi}(p) \rvert ^2}{p}+O(1).
\end{equation*}
Therefore applying Lemma \ref{Calculating mean}, \eqref{Transition 1} is equal to
\begin{align*}
\boldsymbol{\mu}_{2d}\textnormal{Var}_{\pi}^{d}+O((\log T)^{d-1}).
\end{align*}
\par
For $j=0$, note that
\[
\bigg \lvert \sum_{\substack{n \leqslant V \\ n =p^k , k \geqslant2}} \frac{a_n\Lambda_{\pi}(n)}{n^{\frac{1}{2}+i\gamma_{\pi} }} \bigg \rvert^{2d} \ll \bigg \lvert \sum_{\substack{n=p^k \leqslant V \\ 2 \leqslant k \leqslant m^2+1 }} \frac{a_{p^k}\Lambda_{\pi}(p^k)}{p^{\frac{k}{2}+ki\gamma_{\pi} }} \bigg \rvert^{2d}+\bigg \lvert \sum_{\substack{n \leqslant V \\ n =p^k , k > m^2+1}} \frac{a_n\Lambda_{\pi}(n)}{n^{\frac{1}{2}+i\gamma_{\pi} }} \bigg \rvert^{2d}.
\]
Using \eqref{Luo-Rudnick-Sarnak Bound}, it follows that
\[
\sum_{k >m^2+1} \sum_{\substack{n \leqslant V \\ n =p^k}} \frac{a_n\Lambda_{\pi}(n)}{n^{\frac{1}{2}+i\gamma_{\pi} }} \ll 1.
\]
Therefore we obtain
\begin{align}\label{Separating the Prime Powers RH Case Step 1}
\frac{1}{N_{\pi}(T)}\sum_{\lvert \gamma_{\pi} \rvert \leqslant T} \bigg(\Re \sum_{\substack{n \leqslant V \\ n =p^k , k \geqslant2}} \frac{a_n\Lambda_{\pi}(n)}{n^{\frac{1}{2}+i\gamma_{\pi} }}\bigg)^{2d} & \notag \ll \frac{1}{N_{\pi}(T)}\sum_{\lvert \gamma_{\pi} \rvert \leqslant T} \bigg \lvert \sum_{\substack{n=p^k \leqslant V \\ 2 \leqslant k \leqslant m^2+1 }} \frac{a_{p^k}\Lambda_{\pi}(p^k)}{p^{\frac{k}{2}+ki\gamma_{\pi} }} \bigg \rvert^{2d}+1 \\
&\ll \sum_{k=2}^{m^2+1} \frac{1}{N_{\pi}(T)}\sum_{\lvert \gamma_{\pi} \rvert \leqslant T} \bigg \lvert \sum_{\substack{n=p^k \leqslant V}} \frac{a_{p^k}\Lambda_{\pi}(p^k)}{p^{\frac{k}{2}+ki\gamma_{\pi} }} \bigg \rvert^{2d}+1,
\end{align}
where in the last step, we apply Jensen's inequality. To bound the contribution from the prime powers with $2 \leqslant k \leqslant m^2+1$, we follow the approach from Case 2 of Lemma \ref{Moment Calculations Assuming RH Step 1}. Fixing $k$ and opening the $2d$-th power, we obtain
\begin{align*}
\frac{1}{N_{\pi}(T)}\sum_{\substack{p_1, p_2,\dots, p_d \leqslant V^{1/k} \\ q_1,q_2,\dots, q_d \leqslant V^{1/k}}} \frac{(\prod_{i=1}^d a_{p_i^k} \Lambda_{\pi}(p_i^k))(\prod_{j=1}^d a_{q_j^k} \overline{\Lambda_{\pi}(q_j^k)})}{(PQ)^{k/2}}\sum_{\lvert \gamma_{\pi} \rvert \leqslant T} \bigg(\frac{Q}{P} \bigg )^{ki \gamma_{\pi}}, 
\end{align*}
where $P=p_1p_2 \cdots p_d$ and $Q=q_1q_2 \cdots q_d$. The contribution from the diagonal terms $P=Q$ is $O(1)$ by Hypothesis \textnormal{\textbf{H}$_{\pi}$}. For the non-diagonal terms $P \neq Q$, note that $Q/P$ or $P/Q$ is never a prime power and thus by Corollaries \ref{Gonek assuming RH} and \ref{when b>a}, their contribution is also $O( 1)$. Therefore the contribution when $j = 0$ is $O( 1)$. 
\par
Now suppose $j\neq 0, 2d$. By H\"{o}lder's inequality, the sum corresponding to each such $j$ is
\begin{equation*}
\ll \frac{1}{N_{\pi}(T)} \bigg(\sum_{\lvert \gamma_{\pi} \rvert \leqslant T}\bigg(\Re \sum_{p \leqslant V} \frac{a_p \Lambda_{\pi}(p)}{p^{\frac{1}{2}+i \gamma_{\pi}}}\bigg)^{2d}\bigg)^{\frac{j}{2d}}\bigg(\sum_{\lvert \gamma_{\pi} \rvert \leqslant T}\bigg(\Re \sum_{\substack{n \leqslant V \\ n =p^k, k \geqslant2}} \frac{a_n\Lambda_{\pi}(n)}{n^{\frac{1}{2}+i\gamma_{\pi} }} \bigg)^{2d}\bigg)^{\frac{2d-j}{2d}}.
\end{equation*}
Using the estimates for $j=0,2d$, the above is $\ll (\log T)^{d-\frac{1}{2}}$. Combining the estimates for different values of $j$ in \eqref{Binomial Expansion Step}, we arrive at our desired result when the exponent is even.
\par
Now, suppose the exponent is odd. We consider two cases. When $d=1$, by definition and Lemma \ref{Calculating mean},
\[
\frac{1}{N_{\pi}(T)}\sum_{\lvert \gamma_{\pi} \rvert \leqslant T}\Re \sum_{n \leqslant V} \frac{a_n \Lambda_{\pi}(n)}{n^{\frac{1}{2}+\gamma_{\pi}}} = \text{Av}_{\pi} \ll 1.
\]
Now assume $d>1$. For $j=2d-1$, by Lemma \ref{Moment Calculations Assuming RH Step 2},
\begin{align*}
\frac{1}{N_{\pi}(T)}\sum_{\lvert \gamma_{\pi} \rvert \leqslant T}\bigg(\Re \sum_{p \leqslant V} \frac{a_p \Lambda_{\pi}(p)}{p^{\frac{1}{2}+i\gamma_{\pi}}}\bigg)^{2d-1} \ll (\log T)^{d-1}.
\end{align*}
For $j=0$, by Cauchy--Schwarz and our previous case on prime powers for even exponents, we get
\begin{align*}
&\frac{1}{N_{\pi}(T)} \sum_{\lvert \gamma_{\pi} \rvert \leqslant T}\bigg(\Re \sum_{\substack{n \leqslant V \\ n =p^k, k \geqslant 2}} \frac{a_n\Lambda_{\pi}(n)}{n^{\frac{1}{2}+i\gamma_{\pi}}}\bigg)^{2d-1} \\
& \ll \frac{1}{N_{\pi}(T)} \bigg( \sum_{\lvert \gamma_{\pi} \rvert \leqslant T}\bigg(\Re \sum_{\substack{n \leqslant V \\ n =p^k, k \geqslant 2}} \frac{a_n\Lambda_{\pi}(n)}{n^{\frac{1}{2}+i\gamma_{\pi}}}\bigg)^{2d} \bigg)^{\frac{1}{2}} \bigg( \sum_{\lvert \gamma_{\pi} \rvert \leqslant T}\bigg(\Re \sum_{\substack{n \leqslant V \\ n =p^k, k \geqslant 2}} \frac{a_n\Lambda_{\pi}(n)}{n^{\frac{1}{2}+i\gamma_{\pi}}}\bigg)^{2d-2} \bigg)^{\frac{1}{2}} \ll 1.
\end{align*}
When \( j \neq 0 \) or \( 2d-1 \), we first use H\"{o}lder's inequality to bound the contribution by
\begin{align}\label{Holder Application}
\frac{1}{N_{\pi}(T)} \bigg(\sum_{\lvert \gamma_{\pi} \rvert \leqslant T}\bigg \lvert \Re \sum_{p \leqslant V} \frac{a_p \Lambda_{\pi}(p)}{p^{\frac{1}{2}+i \gamma_{\pi}}}\bigg \rvert ^{2d-1}\bigg)^{\frac{j}{2d-1}}\bigg(\sum_{\lvert \gamma_{\pi} \rvert \leqslant T}\bigg \lvert \Re \sum_{\substack{n \leqslant V \\ n =p^k, k \geqslant2}} \frac{a_n\Lambda_{\pi}(n)}{n^{\frac{1}{2}+i\gamma_{\pi} }} \bigg \rvert^{2d-1}\bigg )^{\frac{2d-1-j}{2d-1}}.
\end{align}
For the sum over primes, by Cauchy--Schwarz we raise it to even powers and apply Lemma \ref{Moment Calculations Assuming RH Step 2} to obtain
\begin{align}
\sum_{\lvert \gamma_{\pi}\rvert \leqslant T}&\bigg \lvert \Re \sum_{p \leqslant V} \frac{a_p \Lambda_{\pi}(p)}{p^{\frac{1}{2}+i \gamma_{\pi}}}\bigg \rvert ^{2d-1} \ll (\log T)^{d-\frac{1}{2}}. \label{CS Application}
\end{align}
A similar argument works for the sum over higher prime powers and thus, \eqref{Holder Application} is 
\[
\ll \frac{1}{N_{\pi}(T)} (N_{\pi}(T)(\log X)^{d-\frac{1}{2}} )^{\frac{j}{2d-1}} (N_{\pi}(T))^{\frac{2d-1-j}{2d-1}} \ll (\log T)^{d-1},
\]
when $j \neq 0,2d-1$. This completes the proof.
\end{proof}
\section{Proofs of Theorems \ref{main thm} and \ref{thm:twomMomentConditional}}\label{sec: Conditional 2}
In this section, our objective is to prove Theorems \ref{main thm} and \ref{thm:twomMomentConditional}. We will slightly modify \eqref{Defining Parameters} by imposing \(\xi > \frac{4}{3}\). This additional constraint is necessary for the arguments to follow, and we discuss it further at the end of the section, see Remark \ref{Rem: xi>4/3}. We start with the following lemma.
\begin{lemma}\label{S_av to Mean}
Let $m \in \mathbb{N}, \pi \in \mathscr{A}_m$ and $\theta_{m} \in [0, \frac{1}{2}-\frac{1}{m^2+1}]$ be an admissible exponent towards the Ramanujan conjecture for $L(s,\pi)$. Assume the Riemann Hypothesis for $L(s,\pi)$. Let $T \geqslant 2$ and $\alpha, \xi, X, V$ be as in \eqref{Defining Parameters} satisfying 
\[\alpha< \frac{1}{m(1+\xi \theta_m)} \quad \textrm{and} \quad \xi> \frac{4}{3}.\]
Let $\{a_n\}$ be the sequence defined by \eqref{Defining coefficients}. Then
\begin{equation*}
\Re \mathcal{S}_{\textnormal{av},\pi}(X,T)= \alpha + O\bigg(\frac{1}{\log T} \bigg).
\end{equation*}
The implied constant in the error term depends at most on $\pi, \alpha$ and $\xi$.
\end{lemma}
\begin{proof}
All implied constants in the proof may depend at most on $\pi, \alpha$ and $\xi$. Applying Lemma \ref{Connecting S to sum over primes}, we have
\begin{align}
\Re \mathcal{S}_{\textnormal{av},\pi}(X,T) &= \frac{1}{N_{\pi}(T)} \sum_{\lvert \gamma_{\pi'} \rvert \leqslant T} \Re \mathcal{S}_{\pi}(X, T, \rho'_{\pi}) \notag \\
& =  - \frac{1}{N_{\pi}(T)} \sum_{\lvert \gamma'_{\pi} \rvert \leqslant T} \Re\sum_{n =1}^{\infty} \frac{a_n\Lambda_{\pi}(n)}{n^{\frac{1}{2}+i\gamma'_{\pi}}}+ \frac{1}{N_{\pi}( T )} \sum_{\substack{\rho_{\pi}' \in \mathcal{E}}} \bigg (  \Re \mathcal{S}_{\pi}(X, T, \rho'_{\pi})+ \Re \sum_{n =1}^{\infty} \frac{a_n\Lambda_{\pi}(n)}{n^{\frac{1}{2}+i\gamma'_{\pi}}} \bigg )\notag \\
&\quad + O_{\pi}\bigg( \frac{\log T}{X^2} + \frac{\sqrt{X}}{T} + \frac{\log T}{\sqrt{T}}\bigg)\label{S-AV Step 1},
\end{align}
where $\mathcal{E}$ is an exceptional set of non-trivial zeros $\rho'_{\pi}=\beta_{\pi}+i\gamma'_{\pi}$ of $L(s,\pi)$ with $\lvert \gamma'_{\pi} \rvert \leqslant T$ and $\# \mathcal{E} \ll_{\pi} \sqrt{T}\log T$. The second term on the right hand side of \eqref{S-AV Step 1} is
\begin{equation}\label{S-AV Step 2}
    \ll \frac{1}{N_{\pi}( T )} \sum_{\substack{\rho'_{\pi} \in \mathcal{E}}} (\log T+(X \log X)^{1/2}) \ll T^{-\delta_1},
\end{equation}
for some $\delta_1 = \delta_1(\pi,\alpha, \xi)>0$. To see this, we first remark that here we don't use any cancellation on the sum over zeros. By Cauchy--Schwarz and Lemma \ref{Rudnick-Sarnak Lemma}, the sum over $n$ is $\ll_{\pi}(X \log X)^{1/2}$. To bound the expression $\Re \mathcal{S}(X, T, \rho'_{\pi})$, we assume RH for $L(s,\pi)$ and apply standard estimates on sums over zeros (see \cite[Eq. 5.32]{IK2004}). For the first term on the right hand side of \eqref{S-AV Step 1}, we write
\begin{align}\label{Tail Convergence 1}
 \frac{1}{N_{\pi}(T)} \sum_{\lvert \gamma'_{\pi} \rvert \leqslant T} \Re\sum_{n =1}^{\infty} \frac{a_n\Lambda_{\pi}(n)}{n^{\frac{1}{2}+i\gamma'_{\pi}}}- \textnormal{Av}_{\pi} = \frac{1}{N_{\pi}(T)} \sum_{\lvert \gamma'_{\pi} \rvert \leqslant T} \Re\sum_{n>V} \frac{a_n\Lambda_{\pi}(n)}{n^{\frac{1}{2}+i\gamma'_{\pi}}}.
\end{align}
where $\textnormal{Av}_{\pi} = \textnormal{Av}_{\pi}(T, X, V, \{a_n\})  $ is as defined in \eqref{Defining Mean}.  
Consider the inner sum over $n>V$ in the right hand side of \eqref{Tail Convergence 1}. By an application of Cauchy--Schwarz and Lemma \ref{lem:prime_number_theorem},
\begin{align}\label{Tail Convergence 2}
\sum_{n>V} \frac{a_n\Lambda_{\pi}(n)}{n^{\frac{1}{2}}} \ll \sum_{k=0}^{\infty} \bigg(  \sum_{n \in (2^kV,2^{k+1}V]} a_n^2 \bigg)^{\frac{1}{2}}\bigg (  \sum_{n \in (2^kV,2^{k+1}V]} \frac{\lvert \Lambda_{\pi}(n) \rvert^2}{n}\bigg)^{\frac{1}{2}} \ll T^{-\delta_2}
\end{align}
for some $\delta_2 = \delta_2(\pi, \alpha, \xi)>0$. Here we have used that $\xi>\frac{4}{3}$. Therefore bounding the outer sum over $\gamma'_{\pi}$ in the right hand side of \eqref{Tail Convergence 1} trivially, we arrive at 
\begin{align}\label{S-AV Step 4}
\frac{1}{N_{\pi}(T)} \sum_{\lvert \gamma'_{\pi} \rvert \leqslant T} \Re\sum_{n =1}^{\infty} \frac{a_n\Lambda_{\pi}(n)}{n^{\frac{1}{2}+i\gamma'_{\pi}}}- \textnormal{Av}_{\pi} \ll T^{-\delta_2}.
\end{align}
Putting together \eqref{S-AV Step 2} and \eqref{S-AV Step 4} into \eqref{S-AV Step 1} and applying Lemma \ref{Calculating mean}, the proof follows.
\end{proof}
\subsection*{Proof of Theorem \ref{thm:twomMomentConditional}}
First suppose that the exponent is even and write $r=2d$. We let $X=T^{\alpha m}$ and choose  $V=X^{\xi}$ for some $\xi = \xi(\pi,d,\alpha)>\frac{4}{3}$ such that 
\[\alpha< \frac{1}{2dm(1+\xi \theta_m)}.\]
Then Lemma \ref{Connecting S to sum over primes}, alongside the argument we used to establish \eqref{Tail Convergence 2}, shows that for all non-trivial zeros \( \rho'_{\pi} = \frac{1}{2} + i \gamma'_{\pi} \) of $L(s,\pi)$ with \( \lvert \gamma'_{\pi} \rvert \leqslant T \) outside an exceptional set of size at most $O_{\pi}(\sqrt{T}\log T)$, we have
\begin{align}\label{Tail Convergence 3}
\Re\mathcal{S}_{\pi}(X,T, \rho'_{\pi}) = - \Re\sum_{n \leqslant V} \frac{a_n\Lambda_{\pi}(n)}{n^{\frac{1}{2} + i \gamma'_{\pi}}} + E'_{\pi}(X,T,\rho_{\pi}'),
\end{align} 
where $E_{\pi}'(X,T,\rho_{\pi}') \ll_{\pi,d,\alpha}T^{-\delta_1}$ for some $\delta_1 = \delta_1(\pi,d,\alpha)>0$. Therefore we have
\begin{align*}
\frac{1}{N_{\pi}(T)}& \sum_{\lvert \gamma'_{\pi} \rvert \leqslant T} (\Re \mathcal{S}_{\pi}(X, T, \rho'_{\pi}) )^{2d}\\
&= \frac{1}{N_{\pi}(T)}\sum_{\lvert \gamma'_{\pi} \rvert \leqslant T} \bigg (-\Re \sum_{n \leqslant V} \frac{a_n\Lambda_{\pi}(n)}{n^{\frac{1}{2}+i\gamma'_{\pi}}} +E'_{\pi}(X,T,\rho_{\pi}')\bigg)^{2d} \\
&\quad + \frac{1}{N_{\pi}( T )} \sum_{\substack{ \rho'_{\pi} \in \mathcal{E}}} \bigg ((\Re \mathcal{S}_{\pi}(X,T, \rho'_{\pi}))^{2d}- \bigg (\Re \sum_{n \leqslant V} \frac{a_n\Lambda_{\pi}(n)}{n^{\frac{1}{2}+i\gamma'_{\pi}}} +E'_{\pi}(X,T,\rho_{\pi}')\bigg)^{2d} \bigg )\\
&= \mathcal{U}_1+\mathcal{U}_2
\end{align*}
say, where $\mathcal{E}$ is an exceptional set of non-trivial zeros $\rho'_{\pi}=\beta_{\pi}+i\gamma'_{\pi}$ of $L(s,\pi)$ with $\lvert \gamma'_{\pi} \rvert \leqslant T$ and $\# \mathcal{E} \ll_{\pi} \sqrt{T}\log T$. An application of Lemma \ref{Moment Calculations Assuming RH Step 3} shows that
\begin{align*}
\mathcal{U}_1 &=\frac{1}{N_{\pi}(T)}\sum_{\lvert \gamma'_{\pi} \rvert \leqslant T} \bigg ( \Re \sum_{n \leqslant V} \frac{a_n\Lambda_{\pi}(n)}{n^{\frac{1}{2}+i\gamma'_{\pi}}}\bigg)^{2d}+O_{\pi,d,\alpha}\bigg(\frac{(\log T)^{d-1}}{T^{\delta_1}}\bigg) \notag \\
&=\boldsymbol{\mu}_{2d}\textnormal{Var}_{\pi}^{d} + O_{\pi,d,\alpha}((\log T)^{d-\frac{1}{2}})
\end{align*}
where $\boldsymbol{\mu}_{2d}$ is given by \eqref{Defining Normal Distribution Coefficients} and $\textnormal{Var}_{\pi}$ is as defined in \eqref{Defining SD}. An argument similar to how we proved \eqref{S-AV Step 2} shows that $\mathcal{U}_2 \ll_{\pi,d, \alpha} T^{-\delta_2}$ for some $\delta_2 = \delta_2(\pi,d,\alpha)>0$. Combining the estimates for $\mathcal{U}_1$ and $\mathcal{U}_2$,
\begin{equation}\label{Raw Moments}
\frac{1}{N_{\pi}(T)}\sum_{\lvert \gamma'_{\pi} \rvert \leqslant T} (\Re \mathcal{S}_{\pi}(X, T, \rho'_{\pi}) )^{2d} =\boldsymbol{\mu}_{2d}\textnormal{Var}_{\pi}^{d} + O_{\pi,d,\alpha}((\log T)^{d-\frac{1}{2}}).
\end{equation}
Applying Lemma \ref{Calculating mean} and \ref{S_av to Mean}, it follows that
\begin{equation*}
\mathcal{M}_{\pi, 2d}(T^{\alpha m},T) = \boldsymbol{\mu}_{2d}\left ( \frac{\alpha m \log T}{4} \right)^{d} + O_{\pi,d,\alpha}((\log T)^{d-\frac{1}{2}}).
\end{equation*}
When the exponent is odd, the proof is similar except that $\mathcal{U}_1 \ll_{\pi,d,\alpha} (\log T)^{d-1}$. \qed
\subsection*{Proof of Theorem \ref{main thm}}
For each fixed $r\in \mathbb{N}$, there exists $J_r \in \mathbb{N}$ such that \[
\alpha_j < \frac{1}{rm(1+\frac{4}{3} \theta_m)}\]
for all $j \geqslant J_r$. By Theorem \ref{thm:twomMomentConditional}, for each such $\alpha_j$, there exists $U_{j,r}(\pi,\alpha_j)>0$ depending on $\pi$ and $\alpha_j$ such that for any $T_j \geqslant U_{j,r} $, we have
\begin{align*}
\frac{1}{N_{\pi}(T_j)}\sum_{\lvert \gamma'_{\pi} \rvert \leqslant T_j} \bigg ( \frac{\alpha_{j} m \log T_j}{4} \bigg)^{-\frac{r}{2}} \bigg( \Re \mathcal{S}_{\pi}(T_j^{\alpha_{j} m}, T_j, \rho'_{\pi}) -\alpha_j\bigg)^{r} =\boldsymbol{\mu}_{r} +O_{\pi,r}((\log T_j)^{-\frac{1}{2}}),
\end{align*}
where $\boldsymbol{\mu}_r$ is defined by \eqref{Defining Normal Distribution Coefficients}. Define the sequence $\{V_j\}_{j \in \mathbb{N}}$ as follows. For all $j < J_1$, let $V_j=3$. As $j$ varies over $\mathbb{N}$, we choose  
\[
V_{j}= \max\{U_{j,1},U_{j,2},\cdots,U_{j,r}\}, \quad J_{r} \leqslant j < J_{r+1}.
\]
Then for any sequence $\{T_j\}_{j \in \mathbb{N}}$ such that $T_j \geqslant V_j$
 for all $j \in \mathbb{N}$, and for every positive integer $r$, we have
\begin{align}\label{Final Distribution Statement}
\frac{1}{N_{\pi}(T_j)}\sum_{\lvert \gamma'_{\pi} \rvert \leqslant T_j} \bigg ( \frac{\alpha_{j} m \log T_j}{4} \bigg)^{-\frac{r}{2}} \bigg( \Re \mathcal{S}_{\pi}(T_j^{\alpha_{j} m}, T_j, \rho'_{\pi}) -\alpha_j \bigg)^{r} \rightarrow \boldsymbol{\mu}_{r},\hspace{0.2cm}  \textrm{as} \hspace{0.2cm}  j \to \infty.
\end{align}
The conclusion of Theorem \ref{main thm} now follows by employing the moment method along with an application of the Berry--Esseen theorem. This is a well-known argument, for example, see Davenport--Erd\H{o}s \cite{davenportDistributionQuadraticHigher1952}, Montgomery--Soundararajan \cite{montgomeryPrimesShortIntervals2004} for such applications. \qed

\begin{sectionremark}\label{Rem: xi>4/3}
The additional restriction \(\xi > \frac{4}{3}\), compared to the results in Section \ref{sec: Conditional 1}, is crucial for establishing \eqref{Tail Convergence 3}. While one may improve the results for \eqref{Tail Convergence 2} by leveraging the outer sum over zeros, \eqref{Tail Convergence 3} critically depends on the choice \(\xi > \frac{4}{3}\). This dependence arises because the sequence \(\{a_n\}\) defined by \eqref{Defining coefficients}, originating from the weight function \(w(u)\) in \eqref{weight}, does not decay rapidly enough. This issue does not occur when using smooth, compactly supported weight functions in Section \ref{sec: general weight functions}.
\end{sectionremark}
\section{Moments of Sums over Primes: Part II}\label{sec: unconditional 1}
Let $m \in \mathbb{N}, \pi \in \mathscr{A}_m$ and $\mathcal{\widetilde{S}}_{\pi}(X,T,\rho'_{\pi})$ be as defined in \eqref{Shifted S}. In this section, we focus on obtaining results about the distribution of $\mathcal{\widetilde{S}}_{\pi}(X,T,\rho'_{\pi})$ and thereby addressing the issue of uniformity of $\alpha$ with respect to $T$. To do so, we shall replace the assumption of the Riemann Hypothesis for $L(s,\pi)$ with the following zero density hypothesis for $L(s,\pi)$ which we alluded to in Section \ref{sec: Introduction}.\\

\noindent \textbf{Hypothesis Z$_{\pi}$:}
There exist constants $A_{\pi}>0,B_{\pi}>0$ (depending on $\pi$) such that
$$
\begin{aligned}
N_{\pi}(\sigma, T) & =\# \left\{\rho_{\pi}= \beta_{\pi}+i \gamma_{\pi}: \sigma \leqslant \beta_{\pi}, \lvert \gamma_{\pi} \rvert \leqslant T, L(\rho_{\pi},\pi) =0 \right \}  \ll_{\pi} T^{1-A_{\pi}(\sigma-\frac{1}{2})} \log^{B_{\pi}} T,
\end{aligned}
$$
uniformly for $\sigma \geqslant \frac{1}{2}$ and $T \geqslant 2$.\\

The following proposition is crucial towards establishing results without assuming RH.
\begin{proposition}\label{Crucial RH removing Lemma}
Let $m \in \mathbb{N}$, $ \pi \in \mathscr{A}_m$ and assume Hypothesis \textnormal{\textbf{Z$_{\pi}$}}. Suppose $T,V \geqslant 3$ satisfying
\begin{align}\label{A restriction on V}
\log V =o \left( \frac{\log T}{\log \log T} \right).
\end{align}
Let $k,\ell \in \mathbb{N}$ and suppose $\{p_1,p_2,\dots,p_k\}$ and $\{q_1,q_2,\dots,q_\ell\}$ be two sets of primes satisfying $2\leqslant p_i,q_j \leqslant V$ for all $1 \leqslant i \leqslant k$ and $1 \leqslant j \leqslant \ell$. Let $P = \prod_{i=1}^{k}p_i$ and $Q = \prod_{j=1}^{\ell}q_j$. Then
\begin{align}\label{Eq: Bpi Analysis}
\sum_{\substack{\lvert \gamma_{{\pi}} \rvert \leqslant T} } (PQ)^{\beta_{\pi}-\frac{1}{2}} \bigg (\frac{Q}{P} \bigg )^{i \gamma_{\pi}} = \sum_{\substack{ \lvert \gamma_{\pi} \rvert \leqslant T} } \bigg (\frac{Q}{P}\bigg )^{\rho_{\pi}-\frac{1}{2}}+O_{\pi, k,\ell}
\bigg (\frac{T \log^2 V}{(\log T)^{2-B_{\pi}}} \bigg ).
\end{align}
\end{proposition}
\begin{proof}
Let
\begin{align*}
 & H_1 = \sum_{\substack{\lvert \gamma_{{\pi}} \rvert \leqslant T} } (PQ)^{\beta_{\pi}-\frac{1}{2}} \bigg (\frac{Q}{P} \bigg )^{i \gamma_{\pi}}, \quad H_2 = \sum_{\substack{ \lvert \gamma_{\pi} \rvert \leqslant T} } \bigg (\frac{Q}{P}\bigg )^{i\gamma_{\pi}} \quad \textnormal{and}  \quad H_3 = \sum_{\substack{ \lvert \gamma_{\pi} \rvert \leqslant T} } \bigg (\dfrac{Q}{P}\bigg )^{\rho_{\pi}-\frac{1}{2}}.
\end{align*}
Choose $\delta = (K\log \log T)/\log T$ for some $K=K(A_{\pi}, B_{\pi})$ sufficiently large depending only on $A_{\pi}$ and $B_{\pi}$. We write
\begin{align}\label{Traingle Inequality 1}
\lvert H_1-H_2\rvert & \leqslant \bigg \lvert \sum_{\substack{ \lvert \gamma_{{\pi}} \rvert \leqslant T\\ \mid \beta_{{\pi}}-\frac{1}{2} \mid < \delta}}((PQ)^{\beta_{{\pi}}-\frac{1}{2}}-1) \bigg (\frac{Q}{P}\bigg )^{i\gamma_{{\pi}}}+\sum_{\substack{\lvert \gamma_{{\pi}} \rvert \leqslant T\\ \mid \beta_{{\pi}}-\frac{1}{2} \mid \geqslant\delta}}((PQ)^{\beta_{{\pi}}-\frac{1}{2}}-1) \bigg (\frac{Q}{P} \bigg)^{i\gamma_{{\pi}}} \bigg \rvert \notag \\
&\ll_{\pi, k,\ell} \sum_{\substack{ \lvert \gamma_{{\pi}} \rvert \leqslant T\\ \frac{1}{2}< \beta_{{\pi}} <\frac{1}{2}+ \delta}}((PQ)^{\beta_{{\pi}}-\frac{1}{2}}+(PQ)^{\frac{1}{2}-\beta_{{\pi}}}-2)  \notag \\
&\quad +\sum_{\substack{ 
 \lvert \gamma_{{\pi}} \rvert \leqslant T\\ \beta_{{\pi}} \geqslant \frac{1}{2}+ \delta}}(PQ)^{\delta}+\log V \int_{1 / 2+\delta}^1 (PQ)^{\sigma-1 / 2} N_{\pi}(\sigma, T)\,  d \sigma
\end{align}
By Hypothesis \textbf{Z$_{\pi}$}, our choice of $\delta$ and \eqref{A restriction on V}, we have
\begin{align}\label{Last Two Terms}
\sum_{\substack{ 
 \lvert \gamma_{{\pi}} \rvert \leqslant T\\ \beta_{{\pi}} \geqslant \frac{1}{2}+ \delta}}(PQ)^{\delta}+\log V \int_{1 / 2+\delta}^1 (PQ)^{\sigma-1 / 2} N_{\pi}(\sigma, T)\,  d \sigma \ll_{\pi, k,\ell} \frac{T}{(\log T)^{100}},
\end{align}
say. So we focus on the first term on the right hand side of \eqref{Traingle Inequality 1}. By Taylor expansion, this is
\begin{align}\label{Taylor Expansion}
\ll_{\pi, k, \ell} \log^2 V \sum_{\substack{ 
 \lvert \gamma_{{\pi}} \rvert \leqslant T\\  \frac{1}{2}< \beta_{{\pi}} <\frac{1}{2}+ \delta }} (\beta_{{\pi}}-\frac{1}{2})^2.
\end{align}
By partial summation and Hypothesis \textbf{Z$_{\pi}$},
\begin{align*}
\sum_{\substack{ \lvert \gamma_{{\pi}} \rvert \leqslant T\\ \frac{1}{2} < \beta_{{\pi}}<\frac{1}{2}+ \delta}}(\beta_{\pi}-\frac{1}{2})^2 \ll _{\pi, k,\ell} T \log^{B_{\pi}} T \int_{0}^{\delta} \theta T^{-A_{\pi}\theta} \diff  \theta + \frac{T}{(\log T)^{100}} \ll _{\pi, k,\ell} \frac{T}{(\log T)^{2-B_{\pi}}}.
\end{align*}
Therefore along with \eqref{Traingle Inequality 1} and \eqref{Last Two Terms}, we see that
\begin{align}\label{H1-H3}
\lvert H_1-H_2 \rvert \ll_{\pi, k,\ell}\frac{T \log^2 V}{(\log T)^{2-B_{\pi}}}.
\end{align}
Note that
\[\Big{(}\dfrac{Q}{P}\Big{)}^{i\gamma_{\pi}} = \Big{(}\dfrac{Q}{P}\Big{)}^{\rho_{\pi}-\frac{1}{2}}+\Big{(}\dfrac{Q}{P}\Big{)}^{i\gamma_{\pi}}\bigg(1-\Big{(}\dfrac{Q}{P}\Big{)}^{\beta_{\pi}-\frac{1}{2}}\bigg),
\]
and hence using techniques similar to how we established \eqref{H1-H3}, we deduce that
\begin{align}\label{H2-H3}
\lvert H_2-H_3\rvert &\ll_{\pi, k,\ell} \bigg \lvert \sum_{\substack{ \lvert \gamma_{{\pi}} \rvert \leqslant T\\ \mid \beta_{{\pi}}-\frac{1}{2} \mid < \delta}}\bigg (\bigg(\frac{Q}{P}\bigg)^{\beta_{{\pi}}-\frac{1}{2}}-1\bigg )\bigg \rvert +\frac{T}{(\log T)^{100}} \ll_{\pi, k,\ell} \frac{T \log^2 V}{(\log T)^{2-B_{\pi}}}.
\end{align}
Combining \eqref{H1-H3} and \eqref{H2-H3} and applying triangle inequality, we arrive at our desired conclusion.
\end{proof}
Fix $\xi>1$. Suppose $T\geqslant 3$ and let $\alpha, X,V$ satisfy
\begin{align}
\label{Defining Parameters II}
0<\alpha< \frac{1}{m},\quad X=(\log T)^{\alpha m} \quad \textrm{and} \quad V=X^{\xi}.
\end{align}
Equipped with Proposition \ref{Crucial RH removing Lemma}, we establish an analogue of Lemma \ref{Moment Calculations Assuming RH Step 1} for small ranges of $X$ compared to $T$. Our proof strategy is similar to Lemma \ref{Moment Calculations Assuming RH Step 1}. We highlight only the differences. 
\begin{lemma}\label{General Moments First Step}
Let $m \in \mathbb{N}$ and $\pi \in \mathscr{A}_m$. Assume Hypothesis \textnormal{\textbf{Z$_{\pi}$}} for $L(s,\pi)$ with $B_{\pi} <3$. Suppose $k, \ell$ be non-negative integers with $k+\ell=2d$ for some $d \in \mathbb{N}$. Let $T \geqslant 3$ and $\alpha,\xi, X, V$ be as in \eqref{Defining Parameters II} satisfying $d\alpha m < 3-B_{\pi}$. Let $\{a_n\}$ be the sequence defined by \eqref{Defining coefficients}. Then
\begin{align*}
\frac{1}{N_{\pi}(T)}\sum_{\lvert \gamma_{\pi} \rvert \leqslant T}\bigg(\sum_{p \leqslant V} \frac{a_p \Lambda_{\pi}(p)}{p^{\rho_{\pi}}}\bigg)^k\bigg(\sum_{q \leqslant V} \frac{a_q \overline{\Lambda_{\pi}(q)} }{q^{\overline{\rho_{\pi}}}}\bigg)^\ell &=\boldsymbol{\delta}_{k\ell} d!\bigg( \sum_{p \leqslant V} \frac{a_p^2 \lvert \Lambda_{\pi}(p)\rvert^2}{p}\bigg)^d\\ &\quad+O( \boldsymbol{\delta}_{k\ell} \max \{1, (\log \log T)^{d-2}\}+(\log T)^{-\varepsilon}),
\end{align*}
for some $\varepsilon>0$ depending at most on $\pi,d,\alpha$ and $\xi$. Here $p,q$ run over primes up to $V$ and $\boldsymbol{\delta}_{k\ell}=1$ if $k=\ell$, and zero otherwise. The implied constant in the error term depends at most on $\pi, d, \alpha$ and $\xi$.
\end{lemma}
\begin{proof}
All implied constants in the proof are allowed to depend on $\pi, d, \alpha$ and $\xi$. Let
\begin{equation}\label{B_k,l Notation 1}
\mathcal{B}_{k,\ell} = \frac{1}{N_{\pi}(T)}\sum_{\lvert \gamma_{\pi} \rvert \leqslant T} \bigg(\sum_{p \leqslant V} \frac{a_p \Lambda_{\pi}(p)}{p^{\rho_{\pi}}}\bigg)^k\bigg(\sum_{q \leqslant V} \frac{a_q  \overline{\Lambda_{\pi}(q)}}{q^{\overline{\rho_{\pi}}}}\bigg)^\ell.
\end{equation}
We break our proof into two cases.\\

\noindent \textbf{Case 1 : When $k \neq \ell$.}
Assume $k > \ell$. When $\ell=0$, similar to Lemma \ref{Moment Calculations Assuming RH Step 1} we obtain $\mathcal{B}_{k,0} \ll (\log T)^{-1}.$ Now suppose $\ell > 0$. By \eqref{Functional Equation}, we can write
\begin{align}
\mathcal{B}_{k,\ell}&=\frac{1}{N_{\pi}(T)}\sum_{\substack{p_1, p_2,\dots, p_k \leqslant V \\ q_1,q_2,\dots, q_\ell \leqslant V}} \bigg (\prod_{i=1}^k a_{p_i} \Lambda_{\pi}(p_i)\bigg )\bigg (\prod_{j=1}^\ell a_{q_j}\overline{\Lambda_{\pi}(q_j)}\bigg ) \sum_{\lvert \gamma_{\pi} \rvert \leqslant T} P^{-\rho_{\pi}}Q^{-\overline{\rho_{\pi}}}, \notag \\
&=\frac{1}{N_{\pi}(T)}\sum_{\substack{p_1, p_2,\dots, p_k \leqslant V \\ q_1,q_2,\dots, q_\ell \leqslant V}} \dfrac{(\prod_{i=1}^k a_{p_i} \Lambda_{\pi}(p_i))(\prod_{j=1}^\ell a_{q_j} \overline{\Lambda_{\pi}(q_j)})}{(PQ)^{\frac{1}{2}}} \sum_{\substack{\lvert \gamma_{{\pi}} \rvert \leqslant T} } (PQ)^{\beta_{\pi}-\frac{1}{2}} \bigg (\frac{Q}{P} \bigg )^{i \gamma_{\pi}}, \label{Transition First Step}
\end{align}
where $P=p_1p_2 \cdots p_k$, $Q=q_1q_2 \cdots q_\ell$ and we use the fact that if $L(\rho_{\pi}, \pi)=0$ then $L(1-\overline{\rho_{\pi}}, \pi)=0$. By Proposition \ref{Crucial RH removing Lemma}, the inner sum over $\gamma_{{\pi}}$ in \eqref{Transition First Step} is
\begin{equation}\label{Replacement}
\sum_{\substack{ \lvert \gamma_{\pi} \rvert \leqslant T} } \bigg (\frac{Q}{P}\bigg )^{\rho_{\pi}-\frac{1}{2}}+O \bigg (\frac{T \log^2 V}{(\log T)^{2-B_{\pi}}} \bigg ).
\end{equation}
By Cauchy--Schwarz and Lemma \ref{Rudnick-Sarnak Lemma}, the contribution to \eqref{Transition First Step} from the error term in \eqref{Replacement} is
\begin{equation*}
\ll \frac{\log^2 V}{(\log T)^{3-B_{\pi}}}\bigg ( \sum_{p \leqslant V} \frac{a_p \lvert \Lambda_{\pi}(p) \rvert}{\sqrt{p}} \bigg)^{2d} \ll_{\varepsilon} \frac{X^{d+\varepsilon}}{(\log T)^{3-B_{\pi}}} \ll_{\varepsilon} (\log T)^{-\varepsilon},
\end{equation*}
for some $\varepsilon>0$ depending on $\pi,d,\alpha$ and $\xi$. To estimate the contribution from the first term in \eqref{Replacement}, we note that $Q$ can never be equal to $P$. We break into cases depending on whether $P>Q$ or $P<Q$. After applying Corollary \ref{Gonek assuming RH} and \ref{when b>a}, 
the overall contribution in Case 1 is $\ll_{\varepsilon} (\log T)^{-\varepsilon}$ for some $\varepsilon>0$ depending on $\pi,d,\alpha$ and $\xi$. A concomitant argument follows when $\ell >k$.\\

\noindent\textbf{Case 2 : When \( k = \ell \)}. We write
\begin{align}
\mathcal{B}_{d,d}&=\frac{1}{N_{\pi}(T)}\sum_{\substack{p_1, p_2,\dots, p_d \leqslant V \\ q_1,q_2,\dots, q_d \leqslant V}} \dfrac{(\prod_{i=1}^d a_{p_i} \Lambda_{\pi}(p_i))(\prod_{j=1}^d a_{q_j} \overline{\Lambda_{\pi}(q_j)})}{(PQ)^{\frac{1}{2}}}\sum_{\substack{\lvert \gamma_{{\pi}} \rvert \leqslant T} } (PQ)^{\beta_{\pi}-\frac{1}{2}} \bigg (\frac{Q}{P} \bigg )^{i \gamma_{\pi}} \notag,
\end{align}
where $P=p_1p_2 \cdots p_d$ and $Q=q_1q_2 \cdots q_d$. The main terms again occur when $P =Q$. An application of Proposition \ref{Crucial RH removing Lemma} shows that the contribution from the case $P=Q$ is
\begin{equation}
\sum_{\substack{p_1, p_2,\dots, p_d \leqslant V \\ q_1,q_2,\dots, q_d \leqslant V \\ P=Q}}\dfrac{(\prod_{i=1}^d a_{p_i} \Lambda_{\pi}(p_i))(\prod_{j=1}^d a_{q_j} \overline{\Lambda_{\pi}(q_j)})}{P} +O_{\varepsilon}((\log T)^{-\varepsilon}). \label{Transition Step Main Term}
\end{equation}
Similar to the proof of Lemma \ref{Moment Calculations Assuming RH Step 1}, the first sum in \eqref{Transition Step Main Term} is equal to
\begin{align*}
d!\bigg ( \sum_{p\leqslant V}\frac{a_p^2 \lvert \Lambda_{\pi}(p) \rvert^2}{p} \bigg )^d + O(\max\{1,(\log \log T)^{d-2}\}).
\end{align*}
The cases $P\neq Q$ can be treated as in Case 1. 

Putting together Cases 1 and 2, the proof follows.
\end{proof}
\begin{lemma}\label{General Moments Step 2}
Let $m \in \mathbb{N}$ and $\pi \in \mathscr{A}_m$. Assume Hypothesis \textnormal{\textbf{Z$_{\pi}$}} for $L(s,\pi)$ with $B_{\pi}<3$. Let $d \in \mathbb{N}$, $T \geqslant 3$ and $\alpha,\xi, X, V$ be as in \eqref{Defining Parameters II} satisfying $d\alpha m < 3-B_{\pi}$. Let $\{a_n\}$ be the sequence defined by \eqref{Defining coefficients}. Then
\begin{align*}
\frac{1}{N_{\pi}(T)}\sum_{\lvert \gamma_{\pi} \rvert \leqslant T} \bigg(\Re \sum_{p \leqslant V} \frac{a_p \Lambda_{\pi}(p)}{p^{ \rho_{\pi}}}\bigg)^{2d}=\boldsymbol{\mu}_{2d}\bigg(\frac{1}{2}\sum_{p \leqslant V} \frac{a_p^2  \lvert \Lambda_{\pi}(p) \rvert ^2}{p}\bigg)^d+O ( \max \{1, (\log \log T)^{d-2}\} ),
\end{align*}
where $\boldsymbol{\mu}_{2d}$ is given by \eqref{Defining Normal Distribution Coefficients}. Also, if $(2d-1)\alpha m < 2(3-B_{\pi})$ then for some $\varepsilon= \varepsilon(\pi,d,\alpha,\xi)>0$,
$$
\frac{1}{N_{\pi}(T)}\sum_{\lvert \gamma_{\pi} \rvert \leqslant T} \bigg(\Re \sum_{p \leqslant V} \frac{a_p \Lambda_{\pi}(p)}{p^{\rho_{\pi}}}\bigg)^{2d-1} \ll (\log T)^{-\varepsilon}.
$$
The implied constants in the error term and the $\ll$ notation depend at most on $\pi, d, \alpha$ and $\xi$.
\end{lemma}
\begin{proof}
The proof follows the arguments in Lemma \ref{Moment Calculations Assuming RH Step 2}.
\end{proof}
Recall in \eqref{Defining Mean} and \eqref{Defining SD}, we defined
\begin{align*}
\text{Av}_{\pi}(T, X, V, \{a_n\}) &= \frac{1}{N_{\pi}(T)}\sum_{\lvert \gamma_{\pi} \rvert \leqslant T}\Re \sum_{n \leqslant V} \frac{a_n \Lambda_{\pi}(n)}{n^{ \rho_{\pi}}}\\
\textrm{and} \quad \text{Var}_{\pi}(T, X, V, \{a_n\}) &= \frac{1}{2}\sum_{n \leqslant V} \frac{a_n^2 \lvert \Lambda_{\pi} (n) \rvert ^2}{n} - \text{Av}_{\pi}^2(T, X, V, \{a_n\}).
\end{align*}
The key difference in our work here compared to Section \ref{sec: Conditional 1} is the range of $X$. However this doesn't change the structure of the asymptotics for $\text{Av}_{\pi}$ or $\text{Var}_{\pi}$ as in Lemma \ref{Calculating mean}. In particular, following the proof of Lemma \ref{Calculating mean} and using the estimates from Lemma \ref{Rudnick-Sarnak Lemma}, we have for $T \geqslant 3$ and $\alpha,\xi, X, V$ be as in \eqref{Defining Parameters II} satisfying $\alpha m < 2(3-B_{\pi})$,
\begin{align}\label{Mean UC Asymp}
\text{Av}_{\pi}(T, X, V, \{a_n\}) &=  -\frac{\alpha \log \log T}{\log T}+O_{\pi,\alpha,\xi}\left(\frac{1}{\log T}\right), \\
\label{SD UC Asymp}
\textrm{and} \quad \text{Var}_{\pi}(T, X, V, \{a_n\}) &= \frac{ \alpha m \log \log T}{4}+O_{\pi,\alpha,\xi}(1).
\end{align}
We are now ready to estimate our moments involving sums over prime powers.
\begin{lemma}\label{General Moment Calculations Step 3}
Let $m \in \mathbb{N}$ and $\pi \in \mathscr{A}_m$. Assume Hypothesis \textnormal{\textbf{Z$_{\pi}$}} for $L(s,\pi)$ with $B_{\pi}<3$. Let $d \in \mathbb{N}$, $T \geqslant 3$ and $\alpha,\xi, X, V$ be as in \eqref{Defining Parameters II} satisfying $d\alpha m < 3-B_{\pi}$. Let $\{a_n\}$ be the sequence defined by \eqref{Defining coefficients}. Then
\begin{align*}
\frac{1}{N_{\pi}(T)}\sum_{\lvert \gamma_{\pi} \rvert \leqslant T} \bigg(\Re \sum_{n \leqslant V} \frac{a_n \Lambda_{\pi}(n) }{n^{ \rho_{\pi}}}\bigg)^{2d}=\boldsymbol{\mu}_{2d}\textnormal{Var}_{\pi}^{d} + O((\log \log T)^{d-\frac{1}{2}}),
\end{align*}
where $\boldsymbol{\mu}_{2d}$ is given by \eqref{Defining Normal Distribution Coefficients} and $\textnormal{Var}_{\pi}$ is defined by \eqref{Defining SD}. Also, if $(2d-1)\alpha m < 2(3-B_{\pi})$ then
$$
\frac{1}{N_{\pi}(T)}\sum_{\lvert \gamma_{\pi} \rvert \leqslant T} \bigg(\Re \sum_{n \leqslant V} \frac{a_n \Lambda_{\pi}(n)}{n^{\rho_{\pi}}}\bigg)^{2d-1} \ll 1.
$$
The implied constants in the error term and the $\ll$ notation depend at most on $\pi, d, \alpha$ and $\xi$.
\end{lemma}
\begin{proof}
All implied constants in the proof are allowed to depend on $\pi, d, \alpha$ and $\xi$. The proof is similar to that of Lemma \ref{Moment Calculations Assuming RH Step 3} with slight adjustments. For even moments, we write
\begin{align}\label{Binomial Expansion Step Unconditional}
\sum_{\lvert \gamma_{\pi} \rvert \leqslant T}\bigg(\Re \sum_{n \leqslant V} \frac{a_n \Lambda_{\pi}(n)}{n^{ \rho_{\pi}}}\bigg)^{2d}\notag = \sum_{j=0}^{2d}\binom{2d}{j} \sum_{\lvert \gamma_{\pi} \rvert \leqslant T} \bigg(\Re \sum_{p \leqslant V} \frac{a_p \Lambda_{\pi}(p)}{p^{\rho_{\pi}}}\bigg)^j\bigg(\Re \sum_{\substack{n \leqslant V \\ n =p^k, k \geqslant2}} \frac{a_n\Lambda_{\pi}(n)}{n^{\rho_{\pi} }}\bigg)^{2d-j}.
\end{align}
By Lemma \ref{General Moments Step 2} and \eqref{SD UC Asymp}, when $j=2d$, we have
\begin{equation}\label{Transition 1 Unconditional}
\frac{1}{N_{\pi}(T)}\sum_{\lvert \gamma_{\pi} \rvert \leqslant T} \bigg(\Re \sum_{p \leqslant V} \frac{a_p \Lambda_{\pi}(p)}{p^{\rho_{\pi}}}\bigg)^{2d}=\boldsymbol{\mu}_{2d}\textnormal{Var}_{\pi}^{d}+O ((\log \log T)^{d-1}).
\end{equation}
For $j=0$, we break the sum over $\rho_{\pi}$ into two cases, $\beta_{\pi}>\frac{1}{4}+\frac{\theta_m}{2}$ and otherwise. Note that
\[
\bigg \lvert \sum_{\substack{n \leqslant V \\ n =p^k , k \geqslant2}} \frac{a_n\Lambda_{\pi}(n)}{n^{\rho_{\pi} }} \bigg \rvert^{2d} \ll_{d} \bigg \lvert \sum_{\substack{n=p^k \leqslant V \\ 2 \leqslant k \leqslant 2(m^2+1)}} \frac{a_{p^k}\Lambda_{\pi}(p^k)}{p^{k\rho_{\pi} }} \bigg \rvert^{2d}+\bigg \lvert \sum_{\substack{n=p^k \leqslant V \\ k >2(m^2+1) }} \frac{a_n\Lambda_{\pi}(n)}{n^{\rho_{\pi} }} \bigg \rvert^{2d}.
\]
When $\beta_{\pi}>\frac{1}{4}+\frac{\theta_m}{2}$, we have
\begin{align*}
\sum_{\substack{n =p^k \leqslant V\\ k> 2(m^2+1)}} \frac{a_n\Lambda_{\pi}(n)}{n^{\rho_{\pi} }} \ll 1,
\end{align*}
which implies that
\begin{align}\label{Beta>2/5 Higher Prime Powers}
\frac{1}{N_{\pi}(T)}\sum_{\substack{\lvert \gamma_{\pi} \rvert \leqslant T \\ \beta_{\pi}>\frac{1}{4}+\frac{\theta_m}{2}}} \bigg \lvert \sum_{\substack{n=p^k \leqslant V \\ k >2(m^2+1)}} \frac{a_n\Lambda_{\pi}(n)}{n^{\rho_{\pi} }} \bigg \rvert ^{2d} \ll 1.
\end{align}
On the other hand, for $\beta_{\pi} \leqslant \frac{1}{4}+\frac{\theta_m}{2}$, using Hypothesis \textbf{Z$_{\pi}$}, \begin{align}\label{Beta<2/5 Higher Prime Powers}
\frac{1}{N_{\pi}(T)} \sum_{\substack{\lvert \gamma_{\pi} \rvert \leqslant T \\ \beta_{\pi} \leqslant \frac{1}{4}+\frac{\theta_m}{2}}} \bigg \lvert \sum_{\substack{n=p^k \leqslant V \\ k >2(m^2+1)}} \frac{a_n\Lambda_{\pi}(n)}{n^{\rho_{\pi} }}\bigg \rvert ^{2d} \ll 1.
\end{align}
Therefore putting together \eqref{Beta>2/5 Higher Prime Powers} and \eqref{Beta<2/5 Higher Prime Powers}, the contribution from $j=0$ is
\begin{align*}
\frac{1}{N_{\pi}(T)}\sum_{\lvert \gamma_{\pi} \rvert \leqslant T} \bigg(\Re \sum_{\substack{n \leqslant V \\ n =p^k , k \geqslant2}} \frac{a_n\Lambda_{\pi}(n)}{n^{\rho_{\pi} }}\bigg)^{2d}\ll \frac{1}{N_{\pi}(T)}\sum_{\lvert \gamma_{\pi} \rvert \leqslant T} \bigg \lvert \sum_{\substack{n=p^k \leqslant V \\ 2 \leqslant k \leqslant 2(m^2+1)}} \frac{a_{p^k}\Lambda_{\pi}(p^k)}{p^{k\rho_{\pi} }} \bigg \rvert^{2d}+1.
\end{align*}
From this point, we can proceed using the same approach as in Lemma \ref{Moment Calculations Assuming RH Step 3}. We apply Jensen's inequality, and then handle each $k$ individually, following the reasoning outlined in Lemma \ref{General Moments First Step}. In particular, we obtain
\[\frac{1}{N_{\pi}(T)}\sum_{\lvert \gamma_{\pi} \rvert \leqslant T} \bigg \lvert \sum_{\substack{n=p^k \leqslant V \\ 2 \leqslant k \leqslant 2(m^2+1)}} \frac{a_{p^k}\Lambda_{\pi}(p^k)}{p^{k\rho_{\pi} }} \bigg \rvert^{2d} \ll 1,
\]
which implies the contribution when $j = 0$ is $O(1)$. Finally when $j\neq 0, 2d$, by H\"{o}lder's inequality, the sum corresponding to each such $j$ is
$\ll (\log \log T)^{d-\frac{1}{2}}$. Combining the estimates for different values of $j$, we arrive at our desired result. In the case of odd moments, the argument is precisely as in Lemma \ref{Moment Calculations Assuming RH Step 3} with necessary applications of Lemma \ref{General Moments Step 2}.
\end{proof}
\section{Proofs of Theorems \ref{main thm 2} and \ref{thm:twomMomentUnconditional}}\label{sec : unconditional 2}
Let $m \in \mathbb{N}$ and $\pi \in \mathscr{A}_m$. Suppose $T \geqslant 3$ and $\alpha, \xi, X, V$ be as defined in \eqref{Defining Parameters II} satisfying $\alpha m < 2(3-B_{\pi})$. Recall from Section \ref{sec: Introduction}, the shifted sums 
\begin{align*}
\widetilde{\mathcal{S}}_{\pi}(X,T, \rho') &= \mathcal{S}_{\pi}(X,T, \rho')+ X^{-2} \frac{L'}{L}( \rho'_{\pi} - 2 , \pi), \\
\textrm{and} \quad \Re \widetilde{\mathcal{S}}_{\text{av}, \pi}(X,T) &=\frac{1}{N_{\pi}(T)}\sum_{\lvert \gamma_{\pi}' \rvert \leqslant T} \Re \widetilde{\mathcal{S}}_{\pi}(X,T, \rho_{\pi}').
\end{align*}
Following the proofs of Lemma \ref{Sum over Zeros to Sum over Primes} and \ref{Connecting S to sum over primes}, for all but $O_{\pi}(\sqrt{T}\log T)$ zeros $\rho' = \beta'_{\pi}+i\gamma_{\pi}'$ of $L(s, \pi)$ with $0 < \lvert \gamma'_{\pi} \rvert \leqslant T$, we have
\begin{equation}\label{Zeros to Primes for Tilde-S}
\widetilde{\mathcal{S}}_{\pi}(X,T, \rho'_{\pi}) = - \sum_{n =1}^{\infty} \frac{a_n\Lambda_{\pi}(n)}{n^{\rho'_{\pi}}} + O_{\pi}\bigg( \frac{X \log T}{\sqrt{T}}\bigg).
\end{equation}
Similar to Section \ref{sec: Conditional 2}, we slightly adjust \eqref{Defining Parameters II} by allowing $\xi>3$. We prove the following lemma.
\begin{lemma}\label{S_av to Mean Unconditional}
Let $m \in \mathbb{N}$ and $\pi \in \mathscr{A}_m$. Assume Hypothesis \textnormal{\textbf{Z$_{\pi}$}} for $L(s,\pi)$ with $B_{\pi}<3$. Let $T \geqslant 3$ and $\alpha,\xi, X, V$ be as in \eqref{Defining Parameters II} satisfying $\alpha m < 2(3-B_{\pi})$ and $\xi>3$. Let $\{a_n\}$ be the sequence in \eqref{Defining coefficients}. Then
\begin{equation*}
\Re \widetilde{\mathcal{S}}_{\textnormal{av}, \pi}(X,T)=\frac{\alpha \log \log T}{\log T} + O_{\pi, \alpha, \xi}\bigg(\frac{1}{\log T} \bigg).
\end{equation*}
\end{lemma}
\begin{proof}
We have
\begin{align}
\Re \widetilde{\mathcal{S}}_{\textnormal{av}, \pi}(X,T) &=-\frac{1}{N_{\pi}(T)}\sum_{\lvert \gamma_{\pi}' \rvert \leqslant T} \Re \sum_{n =1}^{\infty} \frac{a_n\Lambda_{\pi}(n)}{n^{\rho'_{\pi}}} \notag \\
&\quad+\frac{1}{N_{\pi}( T )} \sum_{\substack{\gamma_{\pi}' \in \mathcal{E}}} \bigg (  \Re \mathcal{\widetilde{S}}_{\pi}(X, T, \rho'_{\pi})+ \Re \sum_{n =1}^{\infty} \frac{a_n\Lambda_{\pi}(n)}{n^{\rho'_{\pi}}} \bigg ) + O_{\pi}\bigg( \frac{X \log T}{\sqrt{T}}\bigg)  \label{S-AV Step 1 Unconditional},
\end{align}
where $\mathcal{E}$ is an exceptional set of non-trivial zeros $\rho'_{\pi}=\beta_{\pi}+i\gamma'_{\pi}$ of $L(s,\pi)$ with $\lvert \gamma'_{\pi} \rvert \leqslant T$ and $\# \mathcal{E} \ll_{\pi} \sqrt{T}\log T$. The sum over the exceptional zeros is
\begin{equation}\label{S-AV Step 2 Unconditional}
    \ll_{\pi} \frac{1}{N_{\pi}( T )} \sum_{\substack{\gamma' \in \mathcal{E}}} (X\log T+X\log X) \ll_{\pi} \frac{X \log T}{\sqrt{T}}.
\end{equation}
Similar to Lemma \ref{S_av to Mean}, by using a dyadic decomposition argument, we have
\begin{align}\label{S-AV Step 3 Unconditional}
\frac{1}{N_{\pi}(T)} \sum_{\lvert \gamma'_{\pi} \rvert \leqslant T} \Re\sum_{n =1}^{\infty} \frac{a_n\Lambda_{\pi}(n)}{n^{\rho'_{\pi}}}- \textnormal{Av}_{\pi} = \frac{1}{N_{\pi}(T)} \sum_{\lvert \gamma'_{\pi} \rvert \leqslant T} \Re\sum_{n>V} \frac{a_n\Lambda_{\pi}(n)}{n^{\rho'_{\pi}}} \ll_{\pi, \alpha, \xi} (\log T)^{-1}.
\end{align}
Putting together \eqref{S-AV Step 2 Unconditional} and \eqref{S-AV Step 3 Unconditional} into \eqref{S-AV Step 1 Unconditional} and using \eqref{Mean UC Asymp}, the proof follows.
\end{proof}
\subsection*{Proof of Theorem \ref{thm:twomMomentUnconditional}}
First assume the exponent is even and write $r=2d$. We let $X=(\log T)^{\alpha m}$ with $d\alpha m < 3-B_{\pi}$. Choose $V=X^{\xi}$ for some absolute constant $\xi>3$. Lemma \ref{Connecting S to sum over primes}, together with the argument we used to establish \eqref{S-AV Step 3 Unconditional}, shows that for all non-trivial zeros \( \rho'_{\pi} = \beta'_{\pi} + i \gamma'_{\pi} \) of $L(s,\pi)$ with \( \lvert \gamma'_{\pi} \rvert \leqslant T \) outside an exceptional set $\mathcal{E}$ of size at most $O_{\pi}(\sqrt{T}\log T)$, we have
\begin{equation*}
\Re\mathcal{\widetilde{S}}_{\pi}(X,T, \rho'_{\pi}) = - \Re\sum_{n \leqslant V} \frac{a_n\Lambda_{\pi}(n)}{n^{\rho'_{\pi}}} + \widetilde{E}_{\pi}(X,T,\rho_{\pi}),
\end{equation*} 
where $\widetilde{E}_{\pi}(X,T,\rho_{\pi}) \ll_{\pi, \alpha} X^{-1}$. Then we can write
\begin{align*}
\frac{1}{N_{\pi}(T)}& \sum_{\lvert \gamma'_{\pi} \rvert \leqslant T} (\Re \mathcal{\widetilde{S}}_{\pi}(X, T, \rho'_{\pi}) )^{2d}\\
&= \frac{1}{N_{\pi}(T)}\sum_{\lvert \gamma'_{\pi} \rvert \leqslant T} \bigg (-\Re \sum_{n \leqslant V} \frac{a_n\Lambda_{\pi}(n)}{n^{\rho'_{\pi}}} +\widetilde{E}_{\pi}(X,T,\rho_{\pi})\bigg)^{2d} \\
&\quad + \frac{1}{N_{\pi}( T )} \sum_{\substack{\gamma' \in \mathcal{E}}} \bigg ((\Re \mathcal{\widetilde{S}}_{\pi}(X,T, \rho'_{\pi}))^{2d}- \bigg (-\Re \sum_{n \leqslant V} \frac{a_n\Lambda_{\pi}(n)}{n^{\rho'_{\pi}}} +\widetilde{E}_{\pi}(X,T,\rho_{\pi})\bigg)^{2d} \bigg ) = \widetilde{\mathcal{U}}_1+\widetilde{\mathcal{U}}_2,
\end{align*}
say. An application of Lemma \ref{General Moment Calculations Step 3} shows that
\begin{align*}
\widetilde{\mathcal{U}}_1 =\boldsymbol{\mu}_{2d}\textnormal{Var}_{\pi}^{d} + O_{\pi,d, \alpha}((\log \log T)^{d-\frac{1}{2}}).
\end{align*}
An argument similar to how we proved \eqref{S-AV Step 2 Unconditional} shows that $\widetilde{\mathcal{U}}_2 \ll_{\pi, d,\alpha} T^{-\frac{1}{2}+\varepsilon}$ for any $\varepsilon = \varepsilon(\pi, d,\alpha)>0$. Hence we obtain
\begin{equation}\label{Raw Moments}
\frac{1}{N_{\pi}(T)}\sum_{\lvert \gamma'_{\pi} \rvert \leqslant T} (\Re \mathcal{\widetilde{S}}_{\pi}(X, T, \rho'_{\pi}) )^{2d} =\boldsymbol{\mu}_{2d}\textnormal{Var}_{\pi}^{d} + O_{\pi,d, \alpha}((\log \log T)^{d-\frac{1}{2}}).
\end{equation}
Now applying \eqref{SD UC Asymp} and Lemma \ref{S_av to Mean Unconditional}, it follows that
\begin{equation*}
\mathcal{\widetilde{M}}_{\pi, 2d}((\log T)^{\alpha m},T) = \boldsymbol{\mu}_{2d}\left ( \frac{\alpha m \log \log T}{4} \right)^{d} + O_{\pi, d,\alpha}((\log \log T)^{d-\frac{1}{2}}).
\end{equation*}
When the exponent is odd, the proof is similar except that $\mathcal{U}_1 \ll_{\pi,d,\alpha} 1$. \qed

\subsection*{Proof of Theorem \ref{main thm 2}} The proof follows along the lines of the proof of Theorem \ref{main thm}. \qed
\begin{sectionremark}\label{Rem: Addressing B<3}
We briefly address the condition $B_{\pi}<3$ which shows up in the statements of Theorems \ref{main thm 2} and \ref{thm:twomMomentUnconditional}, as well as in the lemmas across Sections \ref{sec: unconditional 1} and \ref{sec : unconditional 2}. This arises essentially from Lemma \ref{Crucial RH removing Lemma}. Noting that \(V=o(\log T)\), we see that, relative to the number of zeros of \(L(s,\pi)\) up to height \(T\), which is \(\asymp_{\pi} T\log T\), the error term in \eqref{Eq: Bpi Analysis} is truly small only when \(B_{\pi}<3\). 
\end{sectionremark}
\section{General Weight Functions: Proof of Theorem \ref{Thm:Main_for_test_funcs}}\label{sec: general weight functions}
\subsection{Preliminaries}
Recall from Section \ref{sec: triple cor intro}, for $X \geqslant 2$ and $\Psi \in C_{c}^{\infty}(0,\infty)$ a fixed non-negative compactly supported smooth function, we defined
\begin{equation*}
    \Psi_X(x) := \Psi\left(\frac{x}{X}\right), \quad x \in \mathbb{R}.
\end{equation*}
Fix $\xi>1$. Suppose $T\geqslant 2$ and let $\alpha, X,V$ satisfy
\begin{align}
\label{Defining Parameters III}
0<\alpha< \frac{1}{m},\quad X=T^{\alpha m} \quad \textrm{and} \quad V=X^{\xi}.
\end{align}
We first record the following estimates.
\begin{lemma}\label{lem:test_general_asymp}
 Let $m \in \mathbb{N}$ and $\pi \in \mathscr{A}_{m}$. Let $\Psi \in C_{c}^{\infty}(0,\infty)$ be a fixed non-negative compactly supported smooth function. Let $T \geqslant 3$ and $\alpha,\xi, X, V$ be as in \eqref{Defining Parameters III}. Then 
 \begin{equation} \label{eqn:mean_test_asymp}
    \sum_{n \leqslant V}\frac{\Psi_X(n)|\Lambda_{\pi}(n)|^2}{n} = \alpha m \log T \int_{0}^{\infty}\frac{\Psi(t)}{t} \diff t + O(1)
 \end{equation}
  and 
  \begin{equation} \label{eqn:std_test_asymp}
    \sum_{n \leqslant V}\frac{\Psi^2_X(n)|\Lambda_{\pi}(n)|^2}{n} = \alpha m \log T \int_{0}^{\infty}\frac{\Psi^2(t)}{t} \diff t + O(1).
 \end{equation}
The implied constants in the error terms depend at most on $\pi, \Psi, \alpha$ and $\xi$.
\end{lemma}
\begin{proof}
Both \eqref{eqn:mean_test_asymp} and \eqref{eqn:std_test_asymp} follow by an application of partial summation along with Lemma~\ref{lem:prime_number_theorem}. 
\end{proof}
\subsection{Explicit Formula}
As before, we conduct the bulk of our analysis on a sum of primes related to our sum over zeros by an explicit formula. We shall require the following lemma. 
\begin{lemma}\label{lem:explicit_formula_test}
Let $m \in \mathbb{N}$ and $\pi \in \mathscr{A}_m$. Assume $\varphi \in C_{c}^{\infty}(2,\infty).$ Let $s\in \mathbb{C}$ such that $L(s,\pi) \neq 0.$ Then 
\begin{equation}\label{eqn:explicit_test_func}
    \sum_{n =1}^{\infty}\frac{\Lambda_{\pi}(n)\varphi(n)}{n^s} = \delta_{\pi}\widehat{\varphi}(1-s) -\sum_{\rho_{\pi}}\widehat{\varphi}(\rho_{\pi}-s) -\underset{\substack{k \in \mathbb{N}\cup \{0\} \\ 1 \leqslant j \leqslant m}}{\sum \sum} \widehat{\varphi}(-2k - \kappa_{\pi}(j) -s),
\end{equation}
where $\kappa_{\pi}(j)\in \mathbb{C}$ are the spectral parameters of $\pi.$ When $s=1,$ the first term in the right hand side of \eqref{eqn:explicit_test_func} should be interpreted as $\delta_{\pi}\widehat{\varphi}(0)$, where $\delta_{\pi} = 1 $ if $L(s,\pi) = \zeta(s)$ and zero otherwise.
\end{lemma}
\begin{proof}
See Appendix. 
\end{proof}
We now extend Lemma \ref{lem:explicit_formula_test} to include $s \in \mathbb{C}$ that are nontrivial zeros of $L(s,\pi).$
\begin{lemma} \label{eqn:explicit_test_func_non_trivial_zeros} Let $m \in \mathbb{N}$ and $ \pi \in \mathscr{A}_m$. Assume $\varphi \in C_{c}^{\infty}(2,\infty).$ Let $\rho'_{\pi}$ be a non-trivial zero of $L(s, \pi)$.  Then we have
\begin{equation}\label{eqn:explicit_formula weight functions}
\sum_{n =1}^{\infty}\frac{\Lambda_{\pi}(n)\varphi(n)}{n^{\rho'_{\pi}}} = \delta_{\pi}\widehat{\varphi}(1-\rho'_{\pi}) -\sum_{\rho_{\pi}}\widehat{\varphi}(\rho_{\pi}-\rho'_{\pi}) -\underset{\substack{k \in \mathbb{N}\cup \{0\} \\ 1 \leqslant j \leqslant m}}{\sum \sum} \widehat{\varphi}(-2k - \kappa_{\pi}(j) -\rho'_{\pi}).
\end{equation}
where $\rho_{\pi}$ runs over non-trivial zeros of $L(s,\pi)$ and $\kappa_{\pi}(j)\in \mathbb{C}$ are the spectral parameters of $\pi$. 
\end{lemma}
\begin{proof}
Fix $\rho'_{\pi}$ to be a zero of $L(s,\pi).$ To prove this lemma, we use the uniqueness principle of complex analysis. It requires us to understand where each term of \eqref{eqn:explicit_test_func} is holomorphic. It is clear that the Dirichlet series on the left hand side of \eqref{eqn:explicit_test_func} is entire. We also find that 
\begin{align*}
    \underset{\substack{k \in \mathbb{N}\cup \{0\} \\ 1 \leqslant j \leqslant m}}{\sum \sum} \widehat{\varphi}(-2k - \kappa_{\pi}(j) -s)
\end{align*}
is analytic in $\Re s \geqslant 0.$ This follows from $\Re{\kappa_{\pi}(j)} \geqslant -\frac{1}{2}.$ Indeed, for all $k \in \mathbb{N}\cup\{0\},$  
\begin{align*}
    \sum_{1 \leqslant j \leqslant m} \widehat{\varphi}(-2k-\kappa_{\pi}(j) -s) \ll_{\pi} 2^{-2k}\int_{2}^{\infty}t^{-1/2}\varphi(t)\, \diff  t \ll_{\pi,\varphi} 2^{-2k}.
\end{align*}
 Because of the support of $\varphi,$ we find that its Mellin transform is entire. Due to the rapid decay of $\widehat{\varphi}$, there exists $\varepsilon(\rho'_{\pi}) >0$ such that for $s\in \mathbb{C}$ with $|s - \rho'_{\pi}| \leqslant \varepsilon(\rho'_{\pi})$, the sum over the nontrivial zeros in \eqref{eqn:explicit_test_func} is uniformly bounded. We now apply the uniqueness principle on the small ball of radius $\varepsilon(\rho'_{\pi})$ centered at $\rho'_{\pi}$ and \eqref{eqn:explicit_formula weight functions} follows.  
\end{proof}
Recall that in Section \ref{sec: triple cor intro}, for $X,T \geqslant 2$ and $\rho'_{\pi}$ be a nontrivial zero of $L(s,\pi),$ we defined 
\begin{align*}
\mathcal{S}_{\pi,\Psi}(X,T,\rho'_{\pi}) = \sum_{|\Im \rho_{\pi}| \leqslant T }\widehat{\Psi}_X(\rho_{\pi} - \rho'_{\pi}).
\end{align*}
We now view $\mathcal{S}_{\pi,\Psi}(X,T,\rho'_{\pi})$ as a sum over prime powers. 
\begin{corollary}\label{cor:S_pi(T)_explicit} Let $m \in \mathbb{N}$ and $\pi \in \mathscr{A}_m.$ Let $\Psi \in C_{c}^{\infty}(0,\infty)$ be a fixed non-negative compactly supported smooth function.  Let $X,T \geqslant 2$. Then for all nontrivial zeros \( \rho'_{\pi} = \beta'_{\pi} + i \gamma'_{\pi} \) of $L(s,\pi)$ with \( \lvert \gamma'_{\pi} \rvert \leqslant T \) outside an exceptional set of size at most $O_{\pi}(\sqrt{T}\log T)$, we have
\begin{equation*}
\mathcal{S}_{\pi,\Psi}(X,T, \rho'_{\pi}) = - \sum_{n =1}^{\infty} \frac{\Psi_{X}(n)\Lambda_{\pi}(n)}{n^{\rho'_{\pi}}} + O_{\pi,\Psi}\bigg(\frac{X^{1-\beta'_{\pi}}}{T} +\frac{X\log T}{\sqrt{T}}\bigg).
\end{equation*}
Moreover, assuming the Riemann Hypothesis for $L(s,\pi)$, we have
\begin{equation*}
\mathcal{S}_{\pi,\Psi}(X,T, \rho'_{\pi}) = - \sum_{n =1}^{\infty} \frac{\Psi_X(n)\Lambda_{\pi}(n)}{n^{\rho'_{\pi}}}+ O_{\pi,\Psi}\bigg(\frac{X^{\frac{1}{2}}}{T} +\frac{\log T}{\sqrt{T}}\bigg).
\end{equation*}
\end{corollary}
\begin{proof} Fix a nontrivial zero \( \rho'_{\pi} = \beta'_{\pi} + i \gamma'_{\pi} \) of $L(s,\pi)$. As before, we may assume that \( \sqrt{T} < \lvert \gamma'_{\pi} \rvert < T-\sqrt{T} \).  Using integration by parts, we derive the following bound
\begin{equation}\label{eqn:mell_transform_bound}
  \widehat{\Psi}_X(z) \ll_{\Psi} \frac{\max\{X^{\Re z},X^{-1}\}}{|z||z+1|}
\end{equation}
for $z \neq 0,-1.$
We consider each term of the explicit formula of Lemma~\ref{eqn:explicit_test_func_non_trivial_zeros} separately.
 In the case $\pi$ is the trivial representation, we find by \eqref{eqn:mell_transform_bound}
\begin{equation*}
    \widehat{\Psi}_X(1-\rho'_{\pi}) \ll_{\Psi} \frac{X^{\Re (1-\rho'_{\pi})}}{|1-\rho'_{\pi}||2-\rho'_{\pi}|} \ll_{\Psi}
    \frac{X^{1-\beta'_{\pi}}}{T},
\end{equation*}
Next, we truncate the infinite sum over our nontrivial zeros. Again by \eqref{eqn:mell_transform_bound},
\begin{align*}
\widehat{\Psi}_X(\rho_{\pi} -\rho'_{\pi}) \ll_{\Psi} \frac{X^{\Re(\rho'_{\pi}-\rho_{\pi})}}{|\rho_{\pi} -\rho'_{\pi}||\rho_{\pi} -\rho'_{\pi}+1|}\ll_{\Psi} \frac{X^{\Re(\rho'_{\pi}-\rho_{\pi})}}{|\gamma_{\pi}-\gamma'_{\pi}|^2}, 
\end{align*}
when $\rho_{\pi} -\rho'_{\pi} \neq 0,-1.$ Following the proof of \eqref{Extending the Sum over all Zeros}, the contribution from these terms are
\begin{align*}
    \ll_{\pi,\Psi} \frac{X\log T}{\sqrt{T}}.
\end{align*}
Assuming RH for $L(s,\pi),$ we have $X^{\Re(\rho'_{\pi}-\rho_{\pi})} \ll 1$. Finally, the sum over the trivial zeros is bounded by our error for the tail of our sum over nontrivial zeros by a similar argument using \eqref{eqn:mell_transform_bound}.
\end{proof}
\subsection{Proof of Theorem \ref{Thm:Main_for_test_funcs}}
Most of the work follows from Section~\ref{sec: Conditional 1} and Section~\ref{sec: Conditional 2}. The key difference is the use of Lemma~\ref{lem:test_general_asymp} instead of Lemma~\ref{Rudnick-Sarnak Lemma}. We also mention that the condition $\xi>1$ in \eqref{Defining Parameters III} is sufficient here, unlike the constraint $\xi>\tfrac{4}{3}$ in Section \ref{sec: Conditional 2}. This contributes to the improvement in the admissible range of $\alpha$ in Theorem \ref{Thm:Main_for_test_funcs} compared to Theorem \ref{thm:twomMomentConditional}. For the benefit of the reader, we include below the lemma for the mixed moments, corresponding to Lemma \ref{Moment Calculations Assuming RH Step 1}, in the context of Theorem~\ref{Thm:Main_for_test_funcs}. The analogous versions of Lemma \ref{Moment Calculations Assuming RH Step 2} and \ref{Moment Calculations Assuming RH Step 3} reads similarly.
\begin{lemma}\label{lem:even_moments_test}
Let $m \in \mathbb{N}, \pi \in \mathscr{A}_m$ and $\theta_{m} \in [0, \frac{1}{2}-\frac{1}{m^2+1}]$ be an admissible exponent towards the Ramanujan conjecture for $L(s,\pi)$. Assume the Riemann Hypothesis for $L(s,\pi)$. Let $\Psi \in C_{c}^{\infty}(0,\infty)$ be a fixed non-negative compactly supported smooth function. Let $k, \ell$ be non-negative integers with $k+\ell=2d$ for some $d \in \mathbb{N}$. Let $T \geqslant 2$ and $\alpha, \xi, X, V$ be as in \eqref{Defining Parameters III} satisfying
\[\alpha < \frac{1}{2dm(1+\xi\theta_m)}.
\]
Then 
\begin{align*}
\frac{1}{N_{\pi}(T)} \sum_{\lvert \gamma_{\pi} \rvert \leqslant T}\bigg(\sum_{p \leqslant V} \frac{\Psi_X(p) \Lambda_{\pi}(p)}{p^{\frac{1}{2}+i\gamma_{\pi}}}\bigg)^k\bigg(\sum_{q \leqslant V} \frac{\Psi_X(p) \overline{\Lambda_{\pi}(q)}}{q^{\frac{1}{2}-i\gamma_{\pi}}}\bigg)^\ell & = \delta_{k\ell} d!\bigg(\sum_{p \leqslant V}\frac{\Psi^2_{X}(p)|\Lambda_{\pi}(p)|^2}{p}\bigg)^d\\
& \quad + O(\max\{1,(\log X)^{d-2}\})
\end{align*}
where $p$ and $q$ run over the primes up to $V$ and $\delta_{k\ell} =1$ if $k=\ell$ and zero otherwise. The implied constant in the error term depends at most on $\pi, \Psi, d, \alpha$ and $\xi$.
\end{lemma}
We also require the following lemma, which is an analogue of Lemma \ref{S_av to Mean} and establishes the average of $\mathcal{S}_{\pi,\Psi}(X,T, \rho_{\pi}')$ as $\rho_{\pi}'$ vary over the zeros of $L(s,\pi)$. We set 
\begin{equation}
    \Re \mathcal{S}_{\textnormal{av},\pi,\Psi}(X,T) := \frac{1}{N_{\pi}(T)} \sum_{\lvert \Im{\rho_{\pi}'} \rvert \leqslant T} 
 \Re \mathcal{S}_{\pi,\Psi}(X,T,\rho'_{\pi}).
\end{equation}
\begin{lemma}\label{lem:mean_computation}
Let $m \in \mathbb{N}, \pi \in \mathscr{A}_m$ and $\theta_{m} \in [0, \frac{1}{2}-\frac{1}{m^2+1}]$ be an admissible exponent towards the Ramanujan conjecture for $L(s,\pi)$. Assume the Riemann Hypothesis for $L(s,\pi)$. Let $\Psi \in C_{c}^{\infty}(0,\infty)$ be a fixed non-negative compactly supported smooth function. Let $T \geqslant 2$ and $\alpha, \xi, X, V$ be as in \eqref{Defining Parameters III} satisfying 
\[\alpha< \frac{1}{m(1+\xi \theta_m)}.\]
Then  
\begin{equation*}
    \Re \mathcal{S}_{\textnormal{av},\pi,\Psi}(X,T) = -\alpha  \int_{0}^{\infty}\frac{\Psi(t)}{t} \diff t + O_{\pi,\Psi, \alpha,\xi}\bigg(\frac{1}{\log T}\bigg).
\end{equation*}
\end{lemma}
From here on, the techniques from Section \ref{sec: Conditional 2} can be adapted to complete the proof of Theorem \ref{Thm:Main_for_test_funcs}. For clarity, we mention here that the analogue of \eqref{Tail Convergence 2} in the proof of Lemma \ref{lem:mean_computation} holds once $\xi>1$. 
\section{Triple Correlations: Proofs of Theorems \ref{Theorem : Triple Correlation} and \ref{Theorem : Triple Correlation CDF}}\label{sec: Triple Correlation}
\subsection{Sums over Zeros to Sums over Primes} Recall from Section \ref{sec: triple cor intro} that for \( X_1, X_2, T \geqslant 2 \) and \( \rho_{\pi} \) a nontrivial zero of \( L(s,\pi) \), we defined
\begin{align}\label{3-Cor Def 1}
\mathcal{S}_{\pi,\Psi}(\rho_{\pi}) := \mathcal{S}_{\pi,\Psi}(X_1, X_2, T, \rho_{\pi}) = \sum_{|\Im \rho_{\pi,1}| \leqslant T } \sum_{|\Im \rho_{\pi,2}| \leqslant T } \widehat{\Psi}_{X_1}(\rho_{\pi,1} - \rho_{\pi}) \widehat{\Psi}_{X_2}(\rho_{\pi} - \rho_{\pi,2})
\end{align}
where \( \Psi_X(x) \) is defined by \eqref{eqn:psi_W}, and the sums run over the nontrivial zeros \( \rho_{\pi,1}, \rho_{\pi,2} \) of \( L(s, \pi) \). Throughout this section, when the parameters \( X_1, X_2, \) and \( T \) are clear from the context, we will omit them from the notation. We begin our analysis with the following lemma.

\begin{lemma}\label{Triple Correlations Zeros to Primes}
Let $m \in \mathbb{N}$ and $\pi \in \mathscr{A}_m$. Assume the Riemann Hypothesis for $L(s,\pi)$. Let $\Psi \in C_{c}^{\infty}(0,\infty)$ be a fixed non-negative compactly supported smooth function. Suppose \( T \geqslant 2 \), $\alpha_1, \alpha_2>0$, $X_1 = T^{\alpha_1 m}$ and $X_2=T^{\alpha_2 m}$. Then for all non-trivial zeros \( \rho_{\pi} = \beta_{\pi} + i \gamma_{\pi} \) of $L(s,\pi)$ with \( \lvert \gamma_{\pi} \rvert \leqslant T \) outside an exceptional set of size at most $O_{\pi}(\sqrt{T}\log T)$, we have
\begin{equation*}
\mathcal{S}_{\pi,\Psi}(\rho_{\pi}) =\sum_{n_1,n_2 =1}^{\infty} \frac{\Psi_{X_1}(n_1) \Psi_{X_2}(n_2) \Lambda_{\pi}(n_1) \Lambda_{\widetilde{\pi}}(n_2)}{(n_1n_2)^{\frac{1}{2}}} \bigg( \frac{n_2}{n_1} \bigg)^{i \gamma_{\pi}}+O\bigg(\frac{X_1^{1/2} \log T}{T^{1/2}}+\frac{X_2^{1/2} \log T}{T^{1/2}}\bigg).
\end{equation*}
The implied constant in the error term depends at most on $\pi, \Psi, \alpha_1$ and $\alpha_2$.
\end{lemma}
\begin{proof}
All implied constants in the proof are allowed to depend on $\pi, \Psi, \alpha_1$ and $\alpha_2$. By the functional equation \eqref{Functional Equation} for $L(s,\pi)$, we rewrite \eqref{3-Cor Def 1} as
\begin{align}\label{3-Cor Step 1}
\mathcal{S}_{\pi,\Psi}(\rho_{\pi}) &= \sum_{|\Im \rho_{\pi,1}| \leqslant T } \widehat{\Psi}_{X_1}(\rho_{\pi,1} - \rho_{\pi}) \sum_{|\Im \rho_{\widetilde{\pi},2}| \leqslant T }  \widehat{\Psi}_{X_2}(\rho_{\widetilde{\pi},2} - \rho_{\widetilde{\pi}}).
\end{align}
By \eqref{cor:S_pi(T)_explicit}, for all non-trivial zeros \( \rho_{\pi} = \beta_{\pi} + i \gamma'_{\pi} \) of $L(s,\pi)$ with \( \lvert \gamma_{\pi} \rvert \leqslant T \) outside an exceptional set of size at most $O_{\pi}(\sqrt{T}\log T)$, we have
\begin{align}
\sum_{|\Im \rho_{\pi,1}| \leqslant T } \widehat{\Psi}_{X_1}(\rho_{\pi,1} - \rho_{\pi}) &= - \sum_{n_1 =1}^{\infty} \frac{\Psi_{X_1}(n_1)\Lambda_{\pi}(n_1)}{n_1^{\rho_{\pi}}}+ O\bigg(\frac{\log T}{\sqrt{T}}\bigg) \label{Each S separately 1} \\
\textrm{and} \quad \sum_{|\Im \rho_{\widetilde{\pi},2}| \leqslant T }  \widehat{\Psi}_{X_2}(\rho_{\widetilde{\pi},2} - \rho_{\widetilde{\pi}}) &= - \sum_{n_2 =1}^{\infty} \frac{\Psi_{X_2}(n_2)\Lambda_{\widetilde{\pi}}(n_2)}{n_2^{\rho_{\widetilde{\pi}}}}+ O\bigg(\frac{\log T}{\sqrt{T}}\bigg). \label{Each S separately 2}
\end{align}
Substituting \eqref{Each S separately 1} and \eqref{Each S separately 2} in \eqref{3-Cor Step 1} and using Hypothesis \textnormal{\textbf{H}$_{\pi}$}, we arrive at 
\begin{align*}
\mathcal{S}_{\pi,\Psi}(\rho_{\pi}) = \sum_{n_1,n_2 =1}^{\infty} \frac{\Psi_{X_1}(n_1) \Psi_{X_2}(n_2) \Lambda_{\pi}(n_1) \Lambda_{\widetilde{\pi}}(n_2)}{(n_1n_2)^{\frac{1}{2}}} \bigg( \frac{n_2}{n_1} \bigg)^{i \gamma_{\pi}}+O\bigg(\frac{X_1^{1/2} \log T}{T^{1/2}}+\frac{X_2^{1/2} \log T}{T^{1/2}}\bigg),
\end{align*}
which completes the proof.
\end{proof}
\subsection{Moment Computations} We now study moments of sums over primes. Fix $\xi>1$. Let $T\geqslant 2$ and $\alpha_1,\alpha_2, X_1,X_2, V_1$ and $V_2$ satisfy
\begin{align}
\label{Defining Parameters Triple Cor}
0<\alpha_1, \alpha_2< \frac{1}{m},\quad X_1=T^{\alpha_1 m}, \quad X_2=T^{\alpha_2 m}, \quad  V_1=X_1^{\xi} \quad \textrm{and} \quad V_2=X_2^{\xi}.
\end{align}
Define
\begin{align}
\mathcal{G}_{\pi,\Psi}(\rho_{\pi}) &= \mathcal{G}_{\pi,\Psi}(X_1,X_2, \rho_{\pi}) :=\sum_{\substack{n_1 \leqslant V_1 \\ n_2 \leqslant V_2}} \frac{\Psi_{X_1}(n_1) \Psi_{X_2}(n_2) \Lambda_{\pi}(n_1) \Lambda_{\widetilde{\pi}}(n_2)}{(n_1n_2)^{\frac{1}{2}}} \left( \frac{n_2}{n_1} \right)^{i \gamma_{\pi}},\label{G Def} \\
\textrm{and} \quad \mathcal{G}^{*}_{\pi,\Psi}(\rho_{\pi}) &= \mathcal{G}^{*}_{\pi,\Psi}(X_1,X_2, \rho_{\pi}) :=\sum_{\substack{p_1 \leqslant V_1 \\ p_2 \leqslant V_2}} \frac{\Psi_{X_1}(p_1) \Psi_{X_2}(p_2) \Lambda_{\pi}(p_1) \Lambda_{\widetilde{\pi}}(p_2)}{(p_1p_2)^{\frac{1}{2}}} \left( \frac{p_2}{p_1} \right)^{i \gamma_{\pi}}\label{G Def Primes} ,
\end{align}
where $p_1,p_2$ runs over primes up to $V_1$ and $V_2$ respectively.
\begin{lemma}\label{Moment Calculations Assuming RH Step 1 : 3-Cor}
Let $m \in \mathbb{N}, \pi \in \mathscr{A}_m$ and $\theta_{m} \in [0, \frac{1}{2}-\frac{1}{m^2+1}]$ be an admissible exponent towards the Ramanujan conjecture for $L(s,\pi)$. Assume the Riemann Hypothesis for $L(s,\pi)$. Let $\Psi \in C_{c}^{\infty}(0,\infty)$ be a fixed non-negative compactly supported smooth function. Suppose $k, \ell$ be non-negative integers with $k+\ell=2d$ for some $d \in \mathbb{N}$. Let $T \geqslant 2$ and $\alpha_1, \alpha_2,\xi, X_1,X_2, V_1,V_2$ be as defined in \eqref{Defining Parameters Triple Cor} satisfying $\alpha_1 \neq \alpha_2$ and
\[
\alpha_1+\alpha_2 < \frac{1}{2dm(1+\xi \theta_m)}.\]
Then
\begin{align}\label{k-l moment}
\frac{1}{N_{\pi}(T)}\sum_{\lvert \gamma_{\pi} \rvert \leqslant T}\mathcal{G}^{*}_{\pi,\Psi}(\rho_{\pi})^k \overline{\mathcal{G}^{*}_{\pi,\Psi}( \rho_{\pi})}^{\ell}&=\boldsymbol{\delta}_{k\ell}d!^2\bigg ( \sum_{p\leqslant
 V_1}\frac{\Psi_{X_1}^2(p) \lvert \Lambda_{\pi}(p) \rvert ^2}{p} \bigg )^d\bigg ( \sum_{p\leqslant
 V_2}\frac{\Psi_{X_2}^2(p) \lvert \Lambda_{\pi}(p) \rvert ^2}{p} \bigg )^d \notag \\
 &\quad+ O((\log T)^{2d-2} ).
\end{align}
Here $\mathcal{G}^{*}_{\pi,\Psi}(\rho_{\pi})$ is defined by \eqref{G Def Primes} and $\boldsymbol{\delta}_{k\ell}=1$ if $k=\ell$, and zero otherwise. The implied constant in the error term depends at most on $\pi, \Psi, d, \alpha_1, \alpha_2$ and $\xi$.
\end{lemma}
\begin{proof}
All implied constants in the proof may depend at most on $\pi,\Psi, d, \alpha_1,\alpha_2$ and $\xi$. We assume $T$ is sufficiently large. For $k, \ell \in \mathbb{N}$, let $\vec{p_i} = (p_{i1}, p_{i2})$ for $1 \leqslant i \leqslant k$ and $\vec{q_j} = (q_{j1}, q_{j2})$ for $1 \leqslant j \leqslant \ell$. We let $P_1= \prod_{i=1}^{k}p_{i1}, P_2= \prod_{i=1}^{k}p_{i2}, Q_1= \prod_{j=1}^{\ell}q_{j1}$ and $Q_1= \prod_{j=1}^{\ell}q_{j2}$. The left hand side of \eqref{k-l moment} is
\begin{align}\label{C Expansion}
\sum_{\substack{\vec{p_1},\dots, \vec{p_k} \\ \vec{q_1},\dots, \vec{q_\ell}}} &\frac{\prod_{i=1}^{k} \Psi_{X_1}(p_{i1})   \Psi_{X_2}(p_{i2}) \Lambda_{\pi}(p_{i1}) \Lambda_{\widetilde{\pi}}(p_{i2}) \prod_{j=1}^{\ell} \Psi_{X_1}(q_{j1}) \Psi_{X_2}(q_{j2}) \Lambda_{\widetilde{\pi}}(q_{j1}) \Lambda_{\pi}(q_{j2})}{(P_1P_2Q_1Q_2)^{\frac{1}{2}}}\notag \\
&\quad \times \frac{1}{N_{\pi}(T)} \sum_{\lvert \gamma_{\pi} \rvert \leqslant T} \bigg(\frac{P_2Q_1}{P_1Q_2} \bigg )^{i \gamma_{\pi}}, 
\end{align}
where the outer sum in the right hand side of \eqref{C Expansion} is restricted to vectors $\vec{p_i}, \vec{q_j}$ with $2 \leqslant p_{i1}, q_{j1} \leqslant V_1$ and $2 \leqslant p_{i2}, q_{j2} \leqslant V_2$. We break our proof into two cases.\\

\noindent \textbf{Case 1 : When $k \neq \ell$.} In this situation, since $\alpha_1 \neq \alpha_2$, $(P_2 Q_1)/(P_1Q_2)$ or $(P_1 Q_2)/(P_2Q_1)$ is never an integer. Similar to the proof of Case 1 in Lemma \ref{Moment Calculations Assuming RH Step 1}, we apply Corollaries \ref{Gonek assuming RH} and \ref{when b>a} to the inner sum over $\gamma_{\pi}$ in \eqref{C Expansion}. Hence by Cauchy--Schwarz and Lemma \ref{Convergent Sums}, the contribution to \eqref{C Expansion} from this case is
\begin{equation}\label{C Non-diagonal}
\ll  \frac{(X_1 X_2)^{2d(1+\xi \theta_m)+\varepsilon} \log T}{T}.
\end{equation}
\noindent \textbf{Case 2 : When $k=\ell$.} The main terms occur when $P_2Q_1 =P_1Q_2$. Since $\alpha_1 \neq \alpha_2$, for $T$ large, this is only possible when $P_1=Q_1$ and $P_2=Q_2$. Applying a combinatorial argument similar to Lemma \ref{Moment Calculations Assuming RH Step 1}, the contribution to \eqref{C Expansion} when $P_2Q_1 =P_1Q_2$ is
\begin{align}\label{C Diagonal}
&d!^2\bigg ( \sum_{p\leqslant
 V_1}\frac{\Psi_{X_1}^2(p) \lvert \Lambda_{\pi}(p) \rvert ^2}{p} \bigg )^d\bigg ( \sum_{p\leqslant
 V_2}\frac{\Psi_{X_2}^2(p) \lvert \Lambda_{\pi}(p) \rvert ^2}{p} \bigg )^d + O( (\log T)^{2d-2} ).
\end{align}
The situation when $P_2Q_1 \neq P_1Q_2$ can be handled exactly as in Case 1.\\

Combining \eqref{C Non-diagonal} and \eqref{C Diagonal} and choosing $\varepsilon>0$ sufficiently small, we arrive at our desired result.
\end{proof}
\begin{lemma}\label{Moment Calculations Assuming RH Step 2 3-Cor}
Let $m \in \mathbb{N}, \pi \in \mathscr{A}_m$ and $\theta_{m} \in [0, \frac{1}{2}-\frac{1}{m^2+1}]$ be an admissible exponent towards the Ramanujan conjecture for $L(s,\pi)$. Assume the Riemann Hypothesis for $L(s,\pi)$. Let $\Psi \in C_{c}^{\infty}(0,\infty)$ be a fixed non-negative compactly supported smooth function. Let $d\in \mathbb{N}$, $T \geqslant 2$ and $\alpha_1, \alpha_2,\xi, X_1,X_2, V_1,V_2$ be as in \eqref{Defining Parameters Triple Cor} satisfying $\alpha_1 \neq \alpha_2$ and
\[
\alpha_1+\alpha_2 < \frac{1}{2dm(1+\xi \theta_m)}.\]
Then
\begin{align*}
\frac{1}{N_{\pi}(T)}\sum_{\lvert \gamma_{\pi} \rvert \leqslant T} (\Re \mathcal{G}^{*}_{\pi,\Psi}( \rho_{\pi}))^{2d}&=\boldsymbol{\mathcal{L}}_{2d}\bigg ( \frac{1}{2}\sum_{p\leqslant
 V_1}\frac{\Psi_{X_1}^2(p) \lvert \Lambda_{\pi}(p) \rvert ^2}{p} \bigg )^d\bigg ( \frac{1}{2} \sum_{p\leqslant
 V_2}\frac{\Psi_{X_2}^2(p) \lvert \Lambda_{\pi}(p) \rvert ^2}{p} \bigg )^d\\
 &\quad+O((\log T)^{2d-2} )
\end{align*}
where $\boldsymbol{\mathcal{L}}_{2d}$ is defined by \eqref{Laplace Coefficients}. Also, if 
\[
\alpha_1+\alpha_2 < \frac{1}{m(2d-1)(1+\xi \theta_m)},\] then
\begin{align*}
\frac{1}{N_{\pi}(T)}\sum_{\lvert \gamma_{\pi} \rvert \leqslant T} (\Re \mathcal{G}^{*}_{\pi,\Psi}( \rho_{\pi}))^{2d-1} \ll T^{-\varepsilon} \quad \textrm{for some} \quad \varepsilon = \varepsilon(\pi,d, \alpha_1, \alpha_2, \xi)>0.
\end{align*}
The implied constants in the error term and the $\ll$ notation depend at most on $\pi, \Psi, d, \alpha_1, \alpha_2$ and $\xi$.
\end{lemma}
\begin{proof}
The proof follows the arguments in Lemma \ref{Moment Calculations Assuming RH Step 2}.
\end{proof}
We now define
\begin{align}\label{Defining Mean 3-Cor}
\widehat{\text{Av}}_{\pi, \Psi}(T, X_1, X_2, \xi) : &= \frac{1}{N_{\pi}(T)}\sum_{\lvert \gamma_{\pi} \rvert \leqslant T}\Re \mathcal{G}_{\pi,\Psi}(\rho_{\pi}),\\
\widehat{\text{Var}}_{\pi, \Psi}(T, X_1,X_2, \xi) : &= \bigg ( \frac{1}{2}\sum_{n\leqslant
 V_1}\frac{\Psi_{X_1}^2(n) \lvert \Lambda_{\pi}(n) \rvert ^2}{n} \bigg )\bigg ( \frac{1}{2} \sum_{n\leqslant
 V_2}\frac{\Psi_{X_2}^2(n) \lvert \Lambda_{\pi}(n) \rvert ^2}{n} \bigg )- \widehat{\text{Av}}_{\pi, \Psi}^2, \label{Defining SD 3-Cor}
\end{align}
where $\mathcal{G}_{\pi,\Psi}(\rho_{\pi})$ is as in \eqref{G Def}. We establish below asymptotic formulas for $\widehat{\text{Av}}_{\pi,\Psi}$ and $\widehat{\text{Var}}_{\pi, \Psi}$ without assuming RH for $L(s,\pi)$.
\begin{lemma}\label{Calculating mean 3-Cor}
Let $m \in \mathbb{N}, \pi \in \mathscr{A}_m$ and $\theta_{m} \in [0, \frac{1}{2}-\frac{1}{m^2+1}]$ be an admissible exponent towards the Ramanujan conjecture for $L(s,\pi)$. Let $\Psi \in C_{c}^{\infty}(0,\infty)$ be a fixed non-negative compactly supported smooth function. Let $d\in \mathbb{N}$, $T \geqslant 2$ and $\alpha_1, \alpha_2,\xi, X_1,X_2, V_1,V_2$ be as in \eqref{Defining Parameters Triple Cor} satisfying $\alpha_1 \neq \alpha_2$ and \[
\alpha_1+\alpha_2 < \frac{1}{m(1+\xi \theta_m)}.\]
Then
\begin{align}
\widehat{\textnormal{Av}}_{\pi,\Psi}(T, X_1, X_2, \xi) & \ll (\log T)^{-1} \label{Mean Estimate 3-Cor},\\
\textrm{and} \quad \widehat{\textnormal{Var}}_{\pi,\Psi}(T, X_1, X_2, \xi) &=  \frac{\alpha_1 \alpha_2 (m \log T)^2}{4} \left( \int_{0}^{\infty}\frac{\Psi^2(t)}{t} \diff t \right)^2+O(\log T) \label{SD Estimate 3-Cor}.
\end{align}
The implied constants in the $\ll$ notation and the error term depend at most on $\pi, \Psi, \alpha_1, \alpha_2$ and $\xi$.
\end{lemma} 
\begin{proof}
All implied constants in this proof depend only on \( \pi, \Psi \), \( \alpha_1 \), \( \alpha_2 \), and \( \xi \). We assume \( T \) is sufficiently large. Without loss of generality, suppose that \( \alpha_2 > \alpha_1 \). By Corollary \ref{when b>a}, we obtain
\begin{align}\label{Expression for AV}
\widehat{\text{Av}}_{\pi,\Psi} = -\frac{T}{\pi N_{\pi}(T)} \Re \sum_{\substack{n_1 \leqslant V_1 \\ n_2 \leqslant V_2}}  \frac{\Psi_{X_1}(n_1) \Psi_{X_2}(n_2) \Lambda_{\pi}(n_1) \Lambda_{\widetilde{\pi}}(n_2)}{n_2}\Lambda_{\pi} \left( \frac{n_2}{n_1} \right) + O\left( \frac{(X_1X_2)^{1+\xi \theta_m + \varepsilon} \log T}{T} \right)
\end{align}
for any \( \varepsilon > 0 \). The error term is acceptable by choosing $\varepsilon$ sufficiently small. We note that $\Lambda_{\pi}(n_2/n_1) \neq 0$ if and only if $n_2/n_1$ is a prime power, that is, if for some prime $p$, $n_1 = p^{\ell_1}$ and $n_2=p^{\ell_2}$ where $\ell_1,\ell_2 \in \mathbb{N}$ with $\ell_2>\ell_1$. So we rewrite the double sum over $n_1,n_2$ in \eqref{Expression for AV} as
\begin{align}\label{Rewriting the Sum}
\sum_{\substack{\ell_1, \ell_2 \geqslant 1 \\ \ell_2 > \ell_1}}^{\infty} \sum_{\substack{p \leqslant \min \{ V_1^{1/\ell_1}, V_2^{1/\ell_2}\}}}  \frac{\Psi_{X_1}(p^{\ell_1}) \Psi_{X_2}(p^{\ell_2}) \Lambda_{\pi}(p^{\ell_1}) \Lambda_{\widetilde{\pi}}(p^{\ell_2})}{p^{\ell_2}}\Lambda_{\pi} (p^{\ell_2-\ell_1})
\end{align}
Note that the length of the sums over $\ell_1$ and $\ell_2$ in \eqref{Rewriting the Sum} is $\ll 1$. We consider two cases:\\

\noindent \textbf{Case 1 : $\ell_2 > m^2+1$.} Using \eqref{Luo-Rudnick-Sarnak Bound}, the sum over $p$ and $\ell_2$ converges absolutely. Therefore, \eqref{Rewriting the Sum} is $\ll 1$.\\

\noindent \textbf{Case 2 : $\ell_2 \leqslant m^2+1$.} Here we treat each tuple $(\ell_1,\ell_2)$ separately. Fixing $(\ell_1,\ell_2)$ and applying Cauchy--Schwarz, \eqref{Rewriting the Sum} is
\begin{align}\label{Cauchy--Schwarz Applied}
\ll \bigg( \sum_p \frac{\lvert \Lambda_{\widetilde{\pi}}(p^{\ell_2}) \rvert^2}{p^{\ell_2}} \bigg)^{\frac{1}{2}}\bigg( \sum_p \frac{\lvert \Lambda_{\pi}(p^{\ell_1}) \Lambda_{\pi}(p^{\ell_2-\ell_1}) \rvert^2}{p^{\ell_2}} \bigg)^{\frac{1}{2}}.
\end{align}
Since $\ell_2>1$, using Hypothesis \textnormal{\textbf{H}$_{\pi}$}, the first factor in the right hand side of \eqref{Cauchy--Schwarz Applied} is $\ll 1$. For the second factor, using Cauchy--Schwarz, Hypothesis \textnormal{\textbf{H}$_{\pi}$} and \eqref{Towards Ramanujan Bounds}, the critical case reduces to $\ell_1=1, \ell_2=2$. To address this final case, we use Lemma \ref{Convergent Sums}. Overall, we conclude that \eqref{Cauchy--Schwarz Applied} is $\ll 1$.
\par
Combining Cases 1 and 2 in \eqref{Expression for AV} and applying \eqref{Zero Counting General}, we deduce \eqref{Mean Estimate 3-Cor}. Applying Lemma \ref{lem:test_general_asymp} and using \eqref{Mean Estimate 3-Cor}, we then conclude the proof of \eqref{SD Estimate 3-Cor}.
\end{proof}
We are now ready to estimate our moments involving sums over prime powers.
\begin{lemma}\label{Moment Calculations Assuming RH Step 3 3-Cor}
Let $m \in \mathbb{N}, \pi \in \mathscr{A}_m$ and $\theta_{m} \in [0, \frac{1}{2}-\frac{1}{m^2+1}]$ be an admissible exponent towards the Ramanujan conjecture for $L(s,\pi)$. Assume the Riemann Hypothesis for $L(s,\pi)$. Let $\Psi \in C_{c}^{\infty}(0,\infty)$ be a fixed non-negative compactly supported smooth function. Let $d\in \mathbb{N}$, $T \geqslant 2$ and $\alpha_1, \alpha_2,\xi, X_1,X_2, V_1,V_2$ be as in \eqref{Defining Parameters Triple Cor} satisfying $\alpha_1 \neq \alpha_2$ and 
\[
\alpha_1+\alpha_2 < \frac{1}{2dm(1+\xi \theta_m)}.\]
Then
\begin{align*}
\frac{1}{N_{\pi}(T)}\sum_{\lvert \gamma_{\pi} \rvert \leqslant T} (\Re \mathcal{G}_{\pi,\Psi}( \rho_{\pi}))^{2d}=\boldsymbol{\mathcal{L}}_{2d} \widehat{\textnormal{Var}}_{\pi, \Psi}^d+O( (\log T)^{2d-\frac{1}{2}} ).
\end{align*}
where $\boldsymbol{\mathcal{L}}_{2d}$ is given by \eqref{Laplace Coefficients} and $\widehat{\textnormal{Var}}_{\pi, \Psi}$ is defined by \eqref{Defining SD 3-Cor}. Also, if 
\[
\alpha_1+\alpha_2 < \frac{1}{m(2d-1)(1+\xi \theta_m)},\] then
\begin{align*}
\frac{1}{N_{\pi}(T)}\sum_{\lvert \gamma_{\pi} \rvert \leqslant T} (\Re \mathcal{G}_{\pi,\Psi}( \rho_{\pi}))^{2d-1} \ll (\log T)^{d-1}.
\end{align*}
The implied constants in the error term and the $\ll$ notation depend at most on $\pi, \Psi, d, \alpha_1, \alpha_2$ and $\xi$.
\end{lemma}
\begin{proof}
All implied constants in the proof may depend at most on $\pi,\Psi, d, \alpha_1,\alpha_2$ and $\xi$. First consider the case when the exponent is even. We break the double sum in $\mathcal{G}_{\pi,\Psi}$ as follows:
\begin{align}\label{Breaking G}
    \mathcal{G}_{\pi,\Psi}(\rho_{\pi}) &= \sum_{\substack{p_1 \leqslant V_1}}\sum_{\substack{p_2 \leqslant V_2}}+\sum_{\substack{p_1 \leqslant V_1}}\sum_{\substack{n_2 = p_2^k \leqslant V_2\\ k \geqslant 2}}+\sum_{\substack{ n_1 = p_1^k \leqslant V_1 \\ k \geqslant 2}}\sum_{\substack{p_2 \leqslant V_1}}+\sum_{\substack{ n_1 = p_1^k \leqslant V_1\\ k \geqslant 2}}\sum_{\substack{n_2 = p_2^{\ell} \leqslant V_2\\ \ell \geqslant 2}} \notag\\
    &=\mathcal{G}^{*}(\rho_{\pi})+\mathcal{G}_1(\rho_{\pi})+\mathcal{G}_2(\rho_{\pi})+\mathcal{G}_3(\rho_{\pi}), 
\end{align}
where $\mathcal{G}^{*}$ is defined by \eqref{G Def Primes}, $p_1,p_2$ run over primes and we drop the suffix $\pi,\Psi$ for brevity. By Binomial expansion, we have
\begin{align}\label{Multinomial Expansion Step}
\frac{1}{N_{\pi}(T)}\sum_{\lvert \gamma_{\pi} \rvert \leqslant T}(\Re \mathcal{G}_{\pi,\Psi}( \gamma_{\pi}))^{2d}&= \frac{1}{N_{\pi}(T)}\sum_{\substack{j_1,j_2,j_3,j_4=0\\ j_1+j_2+j_3+j_4=2d}}^{2d} \frac{2d!}{j_1! j_2!j_3!j_4!} \notag \\
&\quad \times\sum_{\lvert \gamma_{\pi} \rvert \leqslant T} (\Re \mathcal{G}^{*}(\rho_{\pi}))^{j_1}(\Re \mathcal{G}_1(\rho_{\pi}))^{j_2}(\Re \mathcal{G}_2(\rho_{\pi}))^{j_3}(\Re \mathcal{G}_3(\rho_{\pi}))^{j_4}.
\end{align}
By Lemma \ref{Moment Calculations Assuming RH Step 2 3-Cor} and Hypothesis \textnormal{\textbf{H}$_{\pi}$} (see \cite[Eq. 2.25]{rudnickZerosPrincipalLfunctions1996}), the contribution to \eqref{Multinomial Expansion Step} from $j_1=2d$ is
\begin{align}\label{j1=2d}
\boldsymbol{\mathcal{L}}_{2d} \widehat{\textnormal{Var}}_{\pi,\Psi}^{d}+O((\log T)^{2d-1}).
\end{align}
The cases $j_2=2d$ and $j_3=2d$ are symmetric. So we only consider $j_2=2d$. We break into cases when $k>m^2+1$ and otherwise. We write
\begin{align}\label{2dth Power}
\bigg \lvert \sum_{\substack{p_1 \leqslant V_1}}\sum_{\substack{p_2 \leqslant V_2\\ n_2 = p_2^k, k \geqslant 2}} \bigg \rvert^{2d} \ll_d \bigg \lvert \sum_{\substack{p_1 \leqslant V_1}} \sum_{\substack{n_2=p_2^k \leqslant V_2 \\ 2 \leqslant k \leqslant m^2+1 }} \bigg \rvert^{2d}+\bigg \lvert \sum_{\substack{p_1 \leqslant V_1}} \sum_{\substack{n_2=p_2^k \leqslant V_2 \\ k > m^2+1 }} \bigg \rvert^{2d}.
\end{align}
Since 
\[
\sum_{k >m^2+1} \sum_{\substack{n_2 \leqslant V_2 \\ n_2 =p_2^k}} \frac{\Psi_{X_2}(n_2)\Lambda_{\pi}(n_2)}{n^{\frac{1}{2}-i\gamma_{\pi} }} \ll 1,
\]
the second term in the right hand side of \eqref{2dth Power} can be addressed as follows: 
\begin{align}\label{Separating the Prime Powers RH Case Step 1 3-Cor}
\frac{1}{N_{\pi}(T)}\sum_{\lvert \gamma_{\pi} \rvert \leqslant T} \bigg \lvert \sum_{\substack{p_1 \leqslant V_1}} \sum_{\substack{n_2=p_2^k \leqslant V_2\\ k > m^2+1 }} \bigg \rvert^{2d} \ll \frac{1}{N_{\pi}(T)}\sum_{\lvert \gamma_{\pi} \rvert \leqslant T} \bigg \lvert \sum_{\substack{p_1 \leqslant V_1}} \bigg \rvert^{2d} \ll (\log T)^d,
\end{align}
where we apply Lemma \ref{lem:even_moments_test} to obtain the final inequality. Now we treat the first term in the right hand side of \eqref{2dth Power}. To this end, we first apply Jensen's inequality to obtain 
\begin{align}\label{Separating the Prime Powers RH Case Step 2 3-Cor}
\frac{1}{N_{\pi}(T)}\sum_{\lvert \gamma_{\pi} \rvert \leqslant T} \bigg \lvert \sum_{\substack{p_1 \leqslant V_1}} \sum_{\substack{n_2=p_2^k \leqslant V_2 \\ 2 \leqslant k \leqslant m^2+1 }} \bigg \rvert^{2d} &\ll \sum_{k=2}^{m^2+1} \frac{1}{N_{\pi}(T)}\sum_{\lvert \gamma_{\pi} \rvert \leqslant T} \bigg \lvert \sum_{\substack{p_1 \leqslant V_1}} \sum_{\substack{n_2=p_2^k \leqslant V_2 }} \bigg \rvert^{2d},
\end{align}
For each fixed $k$, we first expand the $2d$-th power in the right hand side of \eqref{Separating the Prime Powers RH Case Step 2 3-Cor}. Assume without loss of generality that $\alpha_2>\alpha_1$. Suppose there exists some $2 \leqslant k_0 \leqslant m^2+1$ such that $\alpha_1 =\alpha_2/k_0$. When $k \neq k_0$, we follow the argument from Lemma \ref{Moment Calculations Assuming RH Step 1 : 3-Cor}. The combinatorial argument in this scenario is precisely the same and thus, the contribution to the right side of \eqref{Separating the Prime Powers RH Case Step 2 3-Cor} when $k \neq k_0$ is
\begin{align}\label{Separating the Prime Powers RH Case Step 3 3-Cor}
\ll \sum_{\substack{k=2 \\ k \neq k_0}}^{m^2+1} \bigg ( \sum_{p\leqslant
 V_1}\frac{\Psi_{X_1}^2(p) \lvert \Lambda_{\pi}(p) \rvert ^2}{p} \bigg )^d\bigg ( \sum_{p^k \leqslant
 V_2}\frac{\Psi_{X_2}^2(p^k) \lvert \Lambda_{\pi}(p^k) \rvert ^2}{p^k} \bigg )^d \ll (\log T)^d,
\end{align}
by Hypothesis \textnormal{\textbf{H}$_{\pi}$}. The case $k = k_0$ is only considered if such $k_0$ exists. The argument in this case follows as in Case 2 of Lemma \ref{Calculating mean 3-Cor}. In particular, the contribution to \eqref{Separating the Prime Powers RH Case Step 2 3-Cor} from this case is 
\begin{align}\label{Separating the Prime Powers RH Case Step 4 3-Cor}
\ll \bigg ( \sum_{p\leqslant
 V_1}\frac{\Psi_{X_1}^2(p) \lvert \Lambda_{\pi}(p) \rvert ^2}{p} \bigg )^d \prod_{j=1}^{d} \mathcal{T}_j, \\
 \textrm{where} \quad \mathcal{T}_j = \sum_{p^{k_0} \leqslant
 V_2}\frac{\Psi_{X_2}^2(p^{k_0}) \lvert \Lambda_{\pi}(p^{k_0}) \rvert ^2}{p^{k_0}} \quad \textrm{or} &\quad \sum_{p^{k_0} \leqslant
 V_2}\frac{ \Psi_{X_1}^{k_0}(p) \Psi_{X_2}(p^{k_0}) \lvert \Lambda_{\widetilde{\pi}}(p) \rvert ^{k_0} \lvert \Lambda_{\pi}(p^{k_0}) \rvert}{p^{k_0}}, 1 \leqslant j \leqslant d. \notag
\end{align}
By Cauchy--Schwarz, Hypothesis  \textnormal{\textbf{H}$_{\pi}$} and Lemma \ref{Convergent Sums}, it follows that each $\mathcal{T}_j \ll 1$. This implies that \eqref{Separating the Prime Powers RH Case Step 4 3-Cor} is $\ll (\log T)^d$. Putting this together with \eqref{Separating the Prime Powers RH Case Step 1 3-Cor} and \eqref{Separating the Prime Powers RH Case Step 3 3-Cor}, the overall contribution when $j_2 = 2d$ is $\ll (\log T)^d$. The same bound holds when $j_3=2d$. Finally, techniques similar to the case $j=0$ in Lemma \ref{Moment Calculations Assuming RH Step 3} and the cases $j_2=2d, j_3=2d$ show that the contribution from $j_4=2d$ is $O(1)$. 
\par
We now apply a generalized version of H\"{o}lder's inequality to deal with the other tuples $(j_1,j_2,j_3,j_4)$ in \eqref{Multinomial Expansion Step} to arrive at our desired result. In the case of odd moments, the argument is precisely as in Lemma \ref{Moment Calculations Assuming RH Step 3} with necessary applications of Lemma \ref{Moment Calculations Assuming RH Step 2 3-Cor}. 
\end{proof}
\begin{remark}
Assume that $\alpha_1,\alpha_2>0$ such that for $T \geqslant 2$ sufficiently large, we have
\begin{align}\label{Laplace : Extra Condition}
  \lvert \alpha_1 - \alpha_2 \rvert \geqslant \frac{\log \log T}{\log T}.  
\end{align}
Then the combinatorial analysis corresponding to Case 2 in the proof of Lemma \ref{Moment Calculations Assuming RH Step 1 : 3-Cor} remains the same if we replace the condition $\alpha_1 \neq \alpha_2$ with the condition \eqref{Laplace : Extra Condition}. We may now adapt the arguments in the subsequent lemmas accordingly to replace the condition $\alpha_1 \neq \alpha_2$ with \eqref{Laplace : Extra Condition}.
\end{remark}
\subsection{The case when $\alpha_1=\alpha_2$} \label{Special Case : alpha equal} In the special case where \( \alpha_1 = \alpha_2 = \alpha \), our techniques remain the same as in Lemma \ref{Moment Calculations Assuming RH Step 1 : 3-Cor}. But the combinatorial analysis is different. Considering Lemma \ref{Moment Calculations Assuming RH Step 1 : 3-Cor}, for any non-negative integers $k, \ell$ with $k+\ell=r$ for some $r \in \mathbb{N}$, if all the primes in \( P_1, P_2, Q_1, \) and \( Q_2 \) are distinct, there are \( r! \) possible ways to permute the primes among themselves. In particular, we remark here that the odd moments also contribute main terms. Hence in this scenario, writing $X=T^{\alpha}$ and $V=X^{\xi}$, we obtain
\begin{align*}
\frac{1}{N_{\pi}(T)}\sum_{\lvert \gamma_{\pi} \rvert \leqslant T}\mathcal{G}^{*}_{\pi,\Psi}(\rho_{\pi})^k \overline{\mathcal{G}^{*}_{\pi,\Psi}( \rho_{\pi})}^{\ell} 
&= r! \bigg( \sum_{p \leqslant V} \frac{\Psi_{X}^2(p) \lvert \Lambda_{\pi}(p) \rvert^2}{p} \bigg)^{r} + O( \max \{ 1, (\log T)^{r-2} \}).
\end{align*}
An application of Lemma \ref{Complex Number Identity}, then implies that
\begin{align*}
\frac{1}{N_{\pi}(T)}\sum_{\lvert \gamma_{\pi} \rvert \leqslant T} (\Re \mathcal{G}^{*}_{\pi,\Psi}( \rho_{\pi}))^{r}&=\boldsymbol{\mathcal{\chi}}_r \bigg (\frac{1}{2}\sum_{p\leqslant
 V}\frac{\Psi_{X}^2(p) \lvert \Lambda_{\pi}(p) \rvert ^2}{p} \bigg )^{r}+O( \max \{1, (\log T)^{r-2}\} ).
\end{align*}
The implied constants in the error terms in the above two relations depend at most on $\pi,\Psi,r,\alpha$ and $\xi$. From here on, the arguments in Lemma \ref{Moment Calculations Assuming RH Step 3 3-Cor} can be adapted with necessary changes. Note that the analysis remains the same if we replace the condition $\alpha_1=\alpha_2$ by
\begin{align}\label{Laplace : Extra Condition 2}
  \lvert \alpha_1 - \alpha_2 \rvert \leqslant \frac{1}{(\log \log T) \log T}
\end{align}
for $T \geqslant 2$ sufficiently large.
\subsection{Proof of Theorem \ref{Theorem : Triple Correlation}} All implied constants in the proof depend at most on $\pi,\Psi,r,\alpha_1,\alpha_2$ and $\xi$. Assume $\alpha_1 \neq \alpha_2$. First consider the case when the exponent is even and write $r=2d$. We let $X_1=T^{\alpha_1 m}, X_2 = T^{\alpha_2 m} $ and choose  $V_1=X_1^{\xi}, V_2=X_2^{\xi}$ for some $\xi = \xi(\pi,d,\alpha_1,\alpha_2)>1$ such that 
\[\alpha_1+\alpha_2< \frac{1}{2dm(1+\xi \theta_m)}.\]
Then Lemma \ref{Triple Correlations Zeros to Primes} shows that for all non-trivial zeros \( \rho_{\pi} = \frac{1}{2} + i \gamma_{\pi} \) of $L(s,\pi)$ with \( \lvert \gamma_{\pi} \rvert \leqslant T \) outside an exceptional set of size at most $O_{\pi}(\sqrt{T}\log T)$, we have
\begin{align}\label{Tail Convergence 3 3-Cor}
\Re \mathcal{S}_{\pi,\Psi}(\rho_{\pi}) = \Re\mathcal{G}_{\pi,\Psi}( \rho_{\pi}) + \widehat{E}_{\pi, \Psi}(\rho_{\pi}),
\end{align} 
where $\widehat{E}_{\pi, \Psi}(\rho_{\pi}) \ll T^{-\delta_1}$ for some $\delta_1 = \delta_1(\pi,\Psi,d,\alpha_1,\alpha_2)>0$. Therefore we have
\begin{align*}
\frac{1}{N_{\pi}(T)} \sum_{\lvert \gamma_{\pi} \rvert \leqslant T} (\Re \mathcal{S}_{\pi,\Psi}(\rho_{\pi}) )^{2d}&= \frac{1}{N_{\pi}(T)}\sum_{\lvert \gamma_{\pi} \rvert \leqslant T} (\Re \mathcal{G}_{\pi,\Psi}(\rho_{\pi}) +\widehat{E}_{\pi, \Psi}(\rho_{\pi}))^{2d} \\
&\quad + \frac{1}{N_{\pi}( T )} \sum_{\substack{ \rho_{\pi} \in \mathcal{E}}}  ((\Re \mathcal{S}_{\pi,\Psi}(\rho_{\pi}))^{2d}- (\Re \mathcal{G}_{\pi,\Psi}( \rho_{\pi}) +\widehat{E}_{\pi, \Psi}(\rho_{\pi}))^{2d} )\\
&= \widehat{\mathcal{U}}_1+\widehat{\mathcal{U}}_2,
\end{align*}
say, where $\mathcal{E}$ is an exceptional set of non-trivial zeros $\rho_{\pi}=\beta_{\pi}+i\gamma_{\pi}$ of $L(s,\pi)$ with $\lvert \gamma_{\pi} \rvert \leqslant T$ and $\# \mathcal{E} \ll_{\pi} \sqrt{T}\log T$. An application of Lemma \ref{Moment Calculations Assuming RH Step 3 3-Cor} shows that
\begin{align*}
\widehat{\mathcal{U}}_1 =\boldsymbol{\mathcal{L}}_{2d} \widehat{\textnormal{Var}}_{\pi,\Psi}^{d} + O((\log T)^{2d-\frac{1}{2}})
\end{align*}
where $\boldsymbol{\mathcal{L}}_{2d}$ is given by \eqref{Defining Normal Distribution Coefficients} and $\widehat{\textnormal{Var}}_{\pi,\Psi}$ is as defined in \eqref{Defining SD 3-Cor}. To bound $\widehat{\mathcal{U}}_2$, we use arguments similar to how we proved \eqref{S-AV Step 2}. For clarity, we mention that the term $\Re \mathcal{S}_{\pi,\Psi}(\rho_{\pi})$ is bounded using rapid decay of the Mellin-transform and standard estimates on sums over zeros (see \cite{IK2004}). The term $\Re \mathcal{G}_{\pi,\Psi}( \rho_{\pi})$ is bounded using Cauchy--Schwarz and Lemma \ref{lem:test_general_asymp}. In conclusion, we have $\widehat{\mathcal{U}}_2 \ll T^{-\delta_2}$ for some $\delta_2 = \delta_2(\pi,\Psi,d,\alpha_1,\alpha_2)>0$. Combining the estimates for $\widehat{\mathcal{U}}_1$ and $\widehat{\mathcal{U}}_2$, we obtain
\begin{equation}\label{Raw Moments}
\frac{1}{N_{\pi}(T)}\sum_{\lvert \gamma_{\pi} \rvert \leqslant T} (\Re \mathcal{S}_{\pi,\Psi}(X_1,X_2,T,\rho_{\pi}) )^{2d} =\boldsymbol{\mathcal{L}}_{2d} \widehat{\textnormal{Var}}_{\pi,\Psi}^{d} + O((\log T)^{2d-\frac{1}{2}}).
\end{equation}
Applying \eqref{SD Estimate 3-Cor}, it follows that
\begin{equation*}
\widehat{\mathcal{M}}_{\pi,\Psi, 2d}(T^{\alpha_1 m},T^{\alpha_2 m},T) = \boldsymbol{\mathcal{L}}_{r}\left(\frac{m \sqrt{\alpha_1 \alpha_2} \log T}{2} \int_{0}^{\infty}\frac{\Psi^2(t)}{t} \diff t \right)^{2d} + O((\log T)^{2d-\frac{1}{2}}).
\end{equation*}
For odd moments, the proof is similar except that $\widehat{\mathcal{U}}_1 \ll (\log T)^{d-1}$. 
\par
When $\alpha_1 = \alpha_2$, the arguments are similar except that we use the necessary asymptotics estimates from Section \ref{Special Case : alpha equal} \qed
\subsection{Proof of Theorem \ref{Theorem : Triple Correlation CDF}} The proof follows along the lines of the proof of Theorem \ref{main thm}. The new input is that both the distributions $\textnormal{Laplace }(0,1)$ and $\chi^2_2$ are uniquely determined by their moments (see Billingsley \cite[Theorem 30.1]{Billingsley}). \qed
\section{Statements and Declarations}
\subsection{Competing Interests.} The authors have no competing interests to declare that are relevant to the
content of this article.
\subsection{Data Availability} The data that support the computations of this work are obtained from \cite{lmfdb} and \cite{Rubcalc}. Further details can be found in \cite{github}.
\subsection{Artificial Intelligence Usage} The authors declare that no artificial intelligence or AI-assisted technologies were used in the conduct of the research, or in the generation of any content presented in the manuscript.
\section{Acknowledgements}
We thank Peter Sarnak for helpful comments and for raising several interesting questions during the third author’s talk at the Joint IAS--Princeton University Number Theory Seminar which motivated our work in Sections \ref{sec: general weight functions} and \ref{sec: Triple Correlation}. We thank Jesse Thorner for valuable discussions on the subject of the paper and for encouraging our work in Sections \ref{sec: unconditional 1} and \ref{sec : unconditional 2}. The authors are grateful to Dimitris Koukoulopoulos for pointing out the applicability of the method of moments to the Laplace and Chi-squared distributions. We also thank Vorappan Chandee, Yujiao Jiang, Xiannan Li and Di Liu for valuable comments on the topic of the paper.
\par
Debmalya Basak is partially supported by the Juliette Alexandra Trjitzinsky Fellowship, Department of Mathematics, University of Illinois Urbana-Champaign and by the Max Planck Institute for Mathematics, Bonn. Cruz Castillo is partially supported by the Alfred P. Sloan Foundation’s MPHD Program, awarded in 2021, and by the National Science Foundation Graduate Research Fellowship Program under Grant No. DGE 21-46756. 
\section*{Appendix} 
For the benefit of the reader, we present below the details for the proofs of Lemma \ref{lem:explicit_formula_for_L(s,pi)} and Lemma \ref{lem:explicit_formula_test}.\\
\par

\noindent \textbf{Proof of Lemma \ref{lem:explicit_formula_for_L(s,pi)}.} All implied constants are allowed to depend on $\pi$. Let $\sigma_0 = \max \{ 0, 1-\operatorname{Re}(s)\}+(\log X)^{-1}$. A standard argument using Perron's formula \cite[Corollary 5.3]{montgomery2007multiplicative} allows us to write
\begin{align}\label{Truncated Integral}
\sum_{n \leqslant X} \frac{\Lambda_{\pi}(n)}{n^{s}} =\frac{1}{2 \pi i} \int_{\sigma_0-i T}^{\sigma_0+i T} -\frac{L'}{L}(s+w, \pi) \frac{X^w}{w} \diff w +R_{\pi}(X,T,s),
\end{align}
where $T \geqslant 2+\lvert \operatorname{Im}(s) \rvert+\lvert \operatorname{Im}(\kappa_{\pi}(j)) \rvert $ for each $1 \leqslant j \leqslant m$ and
\begin{align}\label{The Remainder Part}
R_{\pi}(X,T,s) \ll_{\pi} \sum_{\substack{n=1 \\ n \neq X}}^{\infty} \frac{\Lambda_{\pi}(n)}{n^{\operatorname{Re}(s)}} \left(\frac{X}{n} \right)^{\sigma_0} \min \bigg \{1, \frac{1}{T \lvert \log (X/n) \rvert } \bigg \}+\frac{4^{\sigma_0}+X^{\sigma_0}}{T}\left(-\frac{L^{\prime}}{L}\left(\sigma_0+\operatorname{Re}(s), \pi \right)\right).
\end{align}
When $n \geqslant 2X$ or $n \leqslant X/2$, $\lvert \log (X/n) \rvert$ has a positive lower bound and therefore, the contribution from these terms in \eqref{The Remainder Part} is 
\[
\ll_{\pi} \frac{X^{\sigma_0}}{T} \left ( -\frac{L'}{L}(\sigma_0+\operatorname{Re}(s), \pi)\right) \ll_{\pi} \frac{X^{\max \{ 0, 1-\operatorname{Re}(s)\}} \log X}{T}.
\]
Consider next the terms for which $X/2<n<X$. Let $X_1$ be the largest prime power less than $X$. We can assume $X/2<X_1<X$, since otherwise the terms under consideration vanish. For the term $n=X_1$, we have $\log(X/n) \geqslant (X-X_1)/X$ 

and therefore the contribution of this term to \eqref{The Remainder Part} is
$$
\ll_{\pi} \frac{4^{\lvert \operatorname{Re}(s) \rvert }\Lambda_{\pi} \left(X_1\right)}{X^{\operatorname{Re}(s)}} \min \left(1, \frac{X}{T\left(X-X_1\right)}\right) \ll_{\pi} 4^{\lvert \operatorname{Re}(s) \rvert }X^{\theta_m-\operatorname{Re}(s)} (\log X) \min \left(1, \frac{X}{T\left(X-X_1\right)}\right) .
$$
\par
For the other terms, we can put $n=X_1-\nu$ where $0<\nu<X/2$ and then $\log(X/n) \geqslant \nu/X_1$. Hence the contribution of these terms to \eqref{The Remainder Part} is
\begin{align*}
\ll_{\pi} \frac{4^{\lvert \operatorname{Re}(s) \rvert } X^{1+\theta_m} \log X}{X^{\operatorname{Re}(s)} T} \sum_{0<\nu<X/2} \frac{1}{\nu} \ll_{\pi} \frac{4^{\lvert \operatorname{Re}(s) \rvert } X^{1+\theta_m-\operatorname{Re}(s)} \log ^2 X}{ T} .
\end{align*}
A similar argument holds for the terms $X < n < 2X$. Therefore combining all the cases, we obtain
\begin{align}\label{Bound for the Remainder Term}
R_{\pi}(X,T,s) &\ll_{\pi} \frac{4^{\lvert \operatorname{Re}(s) \rvert } X^{1+\theta_m-\operatorname{Re}(s)} \log ^2 X}{ T} +4^{\lvert \operatorname{Re}(s) \rvert }X^{\theta_m-\operatorname{Re}(s)} (\log X) \min \left(1, \frac{X}{T\langle X \rangle}\right) \notag \\
&\quad +\frac{4^{\lvert \operatorname{Re}(s) \rvert } X^{\max \{ 0, 1-\operatorname{Re}(s)\}} \log X}{T},
\end{align}
where $ \langle X \rangle$ denotes the distance of $X$ from the nearest prime power other than $X$ itself. 
\par
Our next step is to shift the path of integration in \eqref{Truncated Integral} to the left. Let $K_{\pi}$ be a large positive number depending only on $\pi$ satisfying $K_{\pi} > -\operatorname{Re}(s)$ that will be chosen later suitably. Since $T \geqslant 2+\lvert \operatorname{Im}(s) \rvert+\lvert \operatorname{Im}(\kappa_{\pi}(j)) \rvert $ for each $1 \leqslant j \leqslant m$, by \cite[Proposition 5.7]{IK2004}, there exists $T_1, T_2 \in[T, T+1]$ such that
\begin{align}\label{Iwaniec-Kowalski Prop 5.7}
\frac{L^{\prime}}{L}\left(\sigma+i \operatorname{Im}(s)-i T_2, \pi \right), \frac{L^{\prime}}{L}\left(\sigma+i \operatorname{Im}(s)+i T_1 , \pi\right) \ll_{\pi} \log T
\end{align}
uniformly for $-1 \leqslant \sigma \leqslant 2$. Note that the range of $\sigma$ in \cite[Proposition 5.7]{IK2004} can be extended to $-1 \leqslant \sigma \leqslant 2$ at the cost of amplifying the implied constant. Moreover, by extending the range of integration in \eqref{Truncated Integral} from $|\operatorname{Im}(w)| \leqslant T$ to $-T_2 \leqslant \operatorname{Im}(w) \leqslant T_1$, the additional error is at most
$$
\ll_{\pi} \frac{X^{\sigma_0}}{T} \left ( -\frac{L'}{L}(\sigma_0+\operatorname{Re}(s), \pi)\right) \ll_{\pi} \frac{X^{\max \{ 0, 1-\operatorname{Re}(s)\}} \log X}{T},
$$
which is acceptable. We replace the contour in \eqref{Truncated Integral} with a new contour $\mathcal{C}$ connecting the points $\sigma_0-i T_2,-K_{\pi}-\operatorname{Re}(s)-i T_2,-K_{\pi}-\operatorname{Re}(s)+i T_1, \sigma_0+i T_1$, in this order. By Cauchy's residue theorem, we get
\begin{align}\label{Contour Theorem}
\frac{1}{2 \pi i} \int_{\sigma_0-i T}^{\sigma_0+i T} -\frac{L'}{L}(s+w, \pi) \frac{X^w}{w} \diff w &=\frac{1}{2 \pi i} \int_{\mathcal{C}}-\frac{L'}{L}(s+w, \pi) \frac{X^w}{w} \diff w -\sum_{-T_2<\Im(\rho_{\pi}-s)<T_1} \frac{X^{\rho_{\pi}-s}}{\rho_{\pi}-s} \notag\\
&\quad+\underset{\substack{k \in \mathbb{N}\cup \{0\}, 1 \leqslant j \leqslant m \\ -K_{\pi}<\operatorname{Re}(-2k-\kappa_{\pi}(j))}}{\sum \sum} \frac{X^{-2 k-\kappa_{\pi}(j)-s}}{2 k+\kappa_{\pi}(j)+s} -\frac{L^{\prime}}{L}(s, \pi )+\delta_{\pi}\frac{X^{1-s}}{1-s}.
\end{align}
Here $\delta_{\pi}=1$ if $L(s,\pi)=\zeta(s)$ and zero otherwise. Also, if $s=1$ and $\delta_{\pi}=1$, the integrand has a double pole at $w=0$ and $X^{1-s}(1-s)^{-1}-\zeta^{\prime} / \zeta(s)$ should be replaced with the residue $\log X-\gamma$ where $\gamma$ is the Euler--Mascheroni constant. We may shorten the sum over $-T_2<\Im(\rho-s)<T_1$ to a sum over $-T \leqslant \Im(\rho-s) \leqslant T$ with an additional error of at most
$$
\ll_{\pi} \sum_{\Im(\rho_{\pi}-s) \in\left(T, T_1\right) \cup\left(-T_2,-T\right)} \frac{X^{1-\operatorname{Re}(s)}}{|\rho_{\pi}-s|} \ll_{\pi} \frac{X^{1-\operatorname{Re}(s)} \log T}{T}
$$
which is again acceptable. Therefore, it remains to bound the integral over $\mathcal{C}$. 
\par
To bound the horizontal parts of our integral over $\mathcal{C}$, we break the range of integration $\operatorname{Re}(w) \in\left[-K_{\pi}-\operatorname{Re}(s), \sigma_0\right]$ into three separate parts. First consider $\operatorname{Re}(w) \in \left[-1-\operatorname{Re}(s), \min \left\{2-\operatorname{Re}(s), \sigma_0\right\}\right]$. We apply \eqref{Iwaniec-Kowalski Prop 5.7} to obtain
\begin{align}\label{Horizontal Integral Part I}
\frac{1}{2 \pi i} \int_{-1-\operatorname{Re}(s)+i T_1}^{\min \left\{2-\operatorname{Re}(s), \sigma_0\right\}+i T_1} -\frac{L'}{L}(s+w, \pi) \frac{X^w}{w} \diff w \ll_{\pi} \frac{\log T}{T} \frac{X^{\min \left\{2-\operatorname{Re}(s), \sigma_0\right\}}}{\log X},
\end{align}
and the same bound holds if $T_1$ is replaced with $-T_2$. Next, the contribution of $\operatorname{Re}(w) \in \left [2-\operatorname{Re}(s), \sigma_0\right]$ should only be considered if this is a nonempty interval. In this case, we obtain
\begin{align}\label{Horizontal Integral Part II}
\frac{1}{2 \pi i} \int_{2-\operatorname{Re}(s)+i T_1}^{\sigma_0+i T_1} -\frac{L'}{L}(s+w, \pi) \frac{X^w}{w} \diff w \ll_{\pi} \frac{X^{\max \{ 0, 1-\operatorname{Re}(s)\}}}{T \log X}.
\end{align}
The same bound holds if $T_1$ is replaced with $-T_2$. Now we consider the contribution from $\operatorname{Re}(w) \in[-K_{\pi}-\operatorname{Re}(s),-1-\operatorname{Re}(s)]$. Let $D \subset \mathbb{C}$ be the half-plane $\{\sigma < -1-\operatorname{Re}(s) \}$ minus a disc of radius $1/(4m)$ around each of the points $-2m-\kappa_{\pi}(j)-s$, where $ m \in \mathbb{N} \cup \{0\}$ and $1 \leqslant j \leqslant m$. For any $w \in D$, by taking the log derivative of the functional equation \eqref{Functional Equation} for $L(s,\pi)$ we have 
\begin{align}
-\frac{L'}{L}(s+w, \pi) &=\frac{L'}{L}(1-s-w, \widetilde{\pi})\notag \\
&\quad+\frac{1}{2}\sum_{j=1}^{m} \bigg ( \frac{\Gamma'}{\Gamma}\left(\frac{s+w+\kappa_{\pi}(j)}{2} \right)+ \frac{\Gamma'}{\Gamma}\left(\frac{1-s-w+\kappa_{\widetilde{\pi}}(j)}{2} \right) \bigg )+O_{\pi}(1).
\end{align} 
Since $w \in D$, we have $L'/L(1-s-w, \widetilde{\pi}) \ll_{\pi} 1$. By Stirling's formula,
\begin{align}\label{Stirling's Formula Step 1}
\sum_{j=1}^{m} \frac{\Gamma'}{\Gamma}\left(\frac{1-s-w+\kappa_{\widetilde{\pi}}(j)}{2} \right) \ll_{\pi} \log (\lvert s+w \rvert+2).
\end{align}
If $\operatorname{Re}(s+w+\kappa_{\pi}(j))>\frac{1}{2}$, again by Stirling's formula,
\begin{align}\label{Stirling's Formula Step 2}
\sum_{j=1}^{m} \frac{\Gamma'}{\Gamma}\left(\frac{s+w+\kappa_{\pi}(j)}{2} \right)  \ll_{\pi} \log (\lvert s+w \rvert+2).
\end{align}
Otherwise, by using the reflection formula for $\Gamma(s)$, we have
\begin{align}\label{Stirling's Formula Step 3}
\sum_{j=1}^{m} \frac{\Gamma'}{\Gamma}\left(\frac{s+w+\kappa_{\pi}(j)}{2} \right) = \sum_{j=1}^{m} \frac{\Gamma'}{\Gamma}\left(1-\frac{s+w+\kappa_{\pi}(j)}{2} \right) - \sum_{j=1}^{m} \pi \cot \left(\pi \frac{s+w+\kappa_{\pi}(j)}{2} \right).
\end{align}
The sum involving $\Gamma'/\Gamma$ on the right-hand side of \eqref{Stirling's Formula Step 3} is again $\ll_{\pi} \log (\lvert s+w \rvert+2)$. Since $w \in D$, the cotangent sum in \eqref{Stirling's Formula Step 3} is $O_{\pi}(1)$. Combining all the estimates, we see that for any $w \in D$, 
\begin{align}\label{Stirling's Formula Step 4}
\bigg \lvert \frac{L'}{L}(s+w, \pi)\bigg \rvert \ll_{\pi} \log (\lvert s+w \rvert+2).
\end{align}
Therefore, it follows that
\begin{align}\label{Horizontal Integral Part III}
\frac{1}{2 \pi i} \int_{-K_{\pi}-\operatorname{Re}(s)+i T_1}^{-1-\operatorname{Re}(s)+i T_1} -\frac{L'}{L}(s+w, \pi) \frac{X^w}{w} \diff w \ll_{\pi} \frac{\log T}{T} \int_{-K_{\pi}}^{-1} X^{a-\operatorname{Re}(s)} \diff  a \ll_{\pi} \frac{\log T}{T} \frac{X^{-1-\operatorname{Re}(s)}}{\log X}.
\end{align}
The same bound holds if $T_1$ is replaced with $-T_2$. The combined error from \eqref{Horizontal Integral Part I}, \eqref{Horizontal Integral Part II} and \eqref{Horizontal Integral Part III} is acceptable. We may now choose $K_{\pi}$ such that our contour lies inside $D$ and therefore the integral over the vertical part of $\mathcal{C}$ is bounded using \eqref{Stirling's Formula Step 4} again. To this end, we obtain
\begin{align}\label{Stirling's Formula Step 5}
\frac{1}{2 \pi i} \int_{-K_{\pi}-\operatorname{Re}(s)-i T_2}^{-K_{\pi}-\operatorname{Re}(s)+i T_1} -\frac{L'}{L}(s+w, \pi) \frac{X^w}{w} \diff w &\ll_{\pi} \frac{\log (K_{\pi} T)}{K_{\pi}+\operatorname{Re}(s)} X^{-K_{\pi}-\operatorname{Re}(s)} \int_{-T_2}^{T_1} \diff  t \notag \\
&\ll_{\pi} \frac{T \log (K_{\pi} T) X^{-K_{\pi}-\operatorname{Re}(s)}}{K_{\pi}+\operatorname{Re}(s)} .
\end{align}
Letting $K_{\pi}$ tend to $\infty$ while still ensuring the constraints on $K_{\pi}$, the bound in \eqref{Stirling's Formula Step 5} tends to zero. Therefore combining all the error terms and putting together \eqref{Truncated Integral} and \eqref{Contour Theorem}, we obtain \eqref{Explicit Formula Equation 1}.
\par
If $X$ is not a prime power, then by shifting $X$ appropriately we may assume $\langle X \rangle \geqslant \frac{1}{2}$. Therefore letting $T$ tending to $\infty$, $E_{\pi}(X,T,s)$ tends to zero and thus \eqref{Explicit Formula Equation 2} follows. \qed \\

\noindent \textbf{Proof of Lemma \ref{lem:explicit_formula_test}.} Assume $s \neq 1.$ For $z \neq 0,$ by integration by parts, 
\begin{equation}\label{eqn:phi_mell}
\widehat{\varphi}(z) = \int_{2}^{\infty}\varphi(y)y^{z-1} \diff y = -\frac{1}{z}\int_{2}^{\infty}\varphi'(y)y^z \diff y .
\end{equation}
We write
\begin{align*}
\sum_{n =1}^{\infty} \frac{\Lambda_{\pi}(n)\varphi(n)}{n^s} = \sum_{n \geqslant 2} \frac{\Lambda_{\pi}(n)}{n^s}\int_{n}^{\infty} -\varphi'(y) \diff y = -\int_{2}^{\infty}\varphi'(y)\sum_{2 \leqslant n \leqslant y} \frac{\Lambda_{\pi}(n)}{n^s} \diff y.
\end{align*}
From \eqref{Explicit Formula Equation 2} of Lemma~\ref{lem:explicit_formula_for_L(s,pi)}, we see 
\begin{equation}\label{eqn:test_explicit_form_step_1}
\begin{split}
\sum_{n =1}^{\infty} \frac{\Lambda_{\pi}(n) \varphi(n)}{n^s} & =  \int_{2}^{\infty} \varphi'(y)\frac{L'}{L}(s,\pi) \diff y + \frac{\delta_{\pi}}{1-s}\int_{2}^{\infty}\varphi'(y)y^{1-s} \diff y + \int_2^{\infty} \sum_{\rho_{\pi}} \frac{\varphi'(y)y^{\rho_{\pi}-s}}{\rho_{\pi}-s} \diff y\\
    & \quad - \int_{2}^{\infty} \underset{\substack{k \in \mathbb{N}\cup \{0\} \\ 1 \leqslant j \leqslant m}}{\sum \sum} \frac{\varphi'(y)y^{-2 k-\kappa_{\pi}(j)-s}}{2 k+\kappa_{\pi}(j)+s}.
\end{split}
\end{equation}
Note we may apply \eqref{Explicit Formula Equation 2} of Lemma~\ref{lem:explicit_formula_for_L(s,pi)} because $\mathbb{N}$ has Lebesgue measure zero. The first integral on the right hand side of \eqref{eqn:test_explicit_form_step_1} vanishes by the Fundamental Theorem of Calculus. Since $s\neq 1,$ we apply \eqref{eqn:phi_mell} to find that the second term of the right-hand side of \eqref{eqn:test_explicit_form_step_1} equals $\delta_{\pi} \widehat{\varphi}(1-s).$ We now proceed to integral of the sum over the zeros of $L(s,\pi).$ Note, we have for $\Re(z) \leqslant 1,$
\begin{equation}\label{eqn:mell_transform_trivial}
    \widehat{\varphi}(z) \ll_{\varphi}2^{\Re(z) -1}.
\end{equation}
Furthermore, when $z \neq 0,-1$ with $\Re z \leqslant 1$, we find 
\begin{equation}\label{eqn_varphi_mell_2nd_derivative}
    \widehat{\varphi}(z) \ll_{\varphi} \frac{1}{|z||z+1|.}
\end{equation}
Using \eqref{eqn:phi_mell} with \eqref{eqn:mell_transform_trivial} and \eqref{eqn_varphi_mell_2nd_derivative}, we may apply Fubini's theorem and interchange the sums with our integrals in  \eqref{eqn:test_explicit_form_step_1}. With this, we arrive at \eqref{eqn:explicit_test_func} for $s\neq 1.$ When $s =1$ and $\pi$ is the trivial representation, we exchange the first two terms of the right-hand side of \eqref{eqn:test_explicit_form_step_1} for $\widehat{\varphi}(0);$ the rest of the work is left unchanged. \qed
\printbibliography

@article{davenportDistributionQuadraticHigher1952,
  title = {The Distribution of Quadratic and Higher Residues},
  author = {Davenport, H. and Erd\H{o}s, P.},
  date = {1952},
  journaltitle = {Publicationes Mathematicae Debrecen},
  shortjournal = {Publ. Math. Debrecen},
  volume = {2},
  pages = {252--265},
  issn = {0033-3883,2064-2849},
  doi = {10.5486/pmd.1952.2.3-4.18},
  mrnumber = {55368}
}

@article{fordDistributionImaginaryParts2009,
  title = {On the Distribution of Imaginary Parts of Zeros of the {{Riemann}} Zeta Function. {{II}}},
  author = {Ford, K. and Soundararajan, K. and Zaharescu, A.},
  date = {2009},
  journaltitle = {Mathematische Annalen},
  shortjournal = {Math. Ann.},
  volume = {343},
  number = {3},
  pages = {487--505},
  issn = {0025-5831},
  doi = {10.1007/s00208-008-0280-x},
  mrnumber = {2480701}
}

@article{gallagherDistributionPrimesShort1976,
  title = {On the Distribution of Primes in Short Intervals},
  author = {Gallagher, P. X.},
  date = {1976},
  journaltitle = {Mathematika. A Journal of Pure and Applied Mathematics},
  shortjournal = {Mathematika},
  volume = {23},
  number = {1},
  pages = {4--9},
  issn = {0025-5793},
  doi = {10.1112/S0025579300016442},
  mrnumber = {409385}
}

@incollection{goldstonPairCorrelationZeros1987,
  title = {Pair Correlation of Zeros and Primes in Short Intervals},
  booktitle = {Analytic Number Theory and {{Diophantine}} Problems ({{Stillwater}}, {{OK}}, 1984)},
  author = {Goldston, D. A. and Montgomery, H. L.},
  date = {1987},
  series = {Progr. {{Math}}.},
  volume = {70},
  pages = {183--203},
  publisher = {{Birkhäuser Boston, Boston, MA}},
  isbn = {978-0-8176-3361-5},
  mrnumber = {1018376}
}

@incollection{gonekExplicitFormulaLandau1993,
  title = {An Explicit Formula of {{Landau}} and Its Applications to the Theory of the Zeta-Function},
  booktitle = {A Tribute to {{Emil Grosswald}}: Number Theory and Related Analysis},
  author = {Gonek, S. M.},
  date = {1993},
  series = {Contemp. {{Math}}.},
  volume = {143},
  pages = {395--413},
  publisher = {{Amer. Math. Soc., Providence, RI}},
  doi = {10.1090/conm/143/1210528},
  mrnumber = {1210528}
}

@article{kim2003functoriality,
  title={Functoriality for the exterior square of GL\_4 and symmetric fourth of GL\_2, Appendix 1 by Dinakar Ramakrishnan, Appendix 2 by Henry H. Kim and Peter Sarnak},
  author={Kim, H.},
  journal={J. Amer. Math. Soc.},
  volume={16},
  pages={139--183},
  year={2003}
}

@article {gallagher1985pair,
    AUTHOR = {Gallagher, P. X.},
     TITLE = {Pair correlation of zeros of the zeta function},
   JOURNAL = {J. Reine Angew. Math.},
  FJOURNAL = {Journal f\"ur die Reine und Angewandte Mathematik. [Crelle's
              Journal]},
    VOLUME = {362},
      YEAR = {1985},
     PAGES = {72--86},
      ISSN = {0075-4102,1435-5345},
   MRCLASS = {11M26},
  MRNUMBER = {809967},
MRREVIEWER = {Masako\ Izumi},
       DOI = {10.1515/crll.1985.362.72},
       URL = {https://doi.org/10.1515/crll.1985.362.72},
}

@book{goldston1985prime,
  title={Prime numbers and the pair correlation of zeros of the zeta function},
  author={Goldston, D. A.},
  year={1985},
  publisher={Univ. Texas Press}
}

@inproceedings{montgomeryPairCorrelationZeros1973,
  title = {The Pair Correlation of Zeros of the Zeta Function},
  booktitle = {Analytic Number Theory ({{Proc}}. {{Sympos}}. {{Pure Math}}., {{Vol}}. {{XXIV}}, {{St}}. {{Louis Univ}}., {{St}}. {{Louis}}, {{Mo}}., 1972)},
  author = {Montgomery, H. L.},
  date = {1973},
  pages = {181--193},
  publisher = {{Amer. Math. Soc., Providence, R.I.}},
  mrnumber = {0337821}
}

@misc{alpöge2026thirdszetazerossimple,
      title={More than two thirds of the zeta zeros are simple and on the critical line}, 
      author={L. Alpöge and R. Furman},
      year={2026},
      eprint={2608.13637},
      archivePrefix={arXiv},
      primaryClass={math.NT},
      url={https://arxiv.org/abs/2608.13637}, 
}

@article{Heath1982,
author = {Heath-Brown, D. R.},
journal = {Acta Arithmetica},
language = {eng},
number = {1},
pages = {85-99},
title = {Gaps between primes, and the pair correlation of zeros of the zeta-function},
url = {http://eudml.org/doc/205833},
volume = {41},
year = {1982},
}

@article {HB-Goldston,
    AUTHOR = {Heath-Brown, D. R. and Goldston, D. A.},
     TITLE = {A note on the differences between consecutive primes},
   JOURNAL = {Math. Ann.},
  FJOURNAL = {Mathematische Annalen},
    VOLUME = {266},
      YEAR = {1984},
    NUMBER = {3},
     PAGES = {317--320},
      ISSN = {0025-5831,1432-1807},
   MRCLASS = {11N05},
  MRNUMBER = {730173},
MRREVIEWER = {Julia\ Mueller},
       DOI = {10.1007/BF01475582},
       URL = {https://doi.org/10.1007/BF01475582},
}

@incollection {Mont-Odlyzko,
    AUTHOR = {Montgomery, H. L. and Odlyzko, A. M.},
     TITLE = {Gaps between zeros of the zeta function},
 BOOKTITLE = {Topics in classical number theory, {V}ol. {I}, {II}
              ({B}udapest, 1981)},
    SERIES = {Colloq. Math. Soc. J\'anos Bolyai},
    VOLUME = {34},
     PAGES = {1079--1106},
 PUBLISHER = {North-Holland, Amsterdam},
      YEAR = {1984},
      ISBN = {0-444-86509-8},
   MRCLASS = {11M26},
  MRNUMBER = {781177},
MRREVIEWER = {A.\ Perelli},
}

@incollection {Odlyzko2,
    AUTHOR = {Odlyzko, A. M.},
     TITLE = {The {$10^{22}$}-nd zero of the {R}iemann zeta function},
 BOOKTITLE = {Dynamical, spectral, and arithmetic zeta functions ({S}an
              {A}ntonio, {TX}, 1999)},
    SERIES = {Contemp. Math.},
    VOLUME = {290},
     PAGES = {139--144},
 PUBLISHER = {Amer. Math. Soc., Providence, RI},
      YEAR = {2001},
      ISBN = {0-8218-2079-6},
   MRCLASS = {11M26 (11Y35)},
  MRNUMBER = {1868473},
       DOI = {10.1090/conm/290/04578},
       URL = {https://doi.org/10.1090/conm/290/04578},
}

@article{montgomeryPrimesShortIntervals2004,
  title = {Primes in Short Intervals},
  author = {Montgomery, H. L. and Soundararajan, K.},
  date = {2004},
  journaltitle = {Communications in Mathematical Physics},
  shortjournal = {Comm. Math. Phys.},
  volume = {252},
  number = {1-3},
  pages = {589--617},
  issn = {0010-3616,1432-0916},
  doi = {10.1007/s00220-004-1222-4},
  mrnumber = {2104891}
}

@article{chandee2014simple,
  title={Simple zeros of primitive Dirichlet $L$-functions and the asymptotic large sieve},
  author={Chandee, V. and Lee, Y. and Liu, S. and Radziwi{\l}{\l}, M.},
  journal={Quarterly Journal of Mathematics},
  volume={65},
  number={1},
  pages={63--87},
  year={2014},
  publisher={OUP}
}

@article {MR0022872,
    AUTHOR = {Selberg, A.},
     TITLE = {Contributions to the theory of {D}irichlet's {$L$}-functions},
   JOURNAL = {Skr. Norske Vid.-Akad. Oslo I},
  FJOURNAL = {Skrifter Utgitt av det Norske Videnskaps-Akademi i Oslo. I.
              Mat.-Naturv. Klasse},
    VOLUME = {1946},
      YEAR = {1946},
    NUMBER = {3},
     PAGES = {62},
      ISSN = {0029-2338},
   MRCLASS = {10.0X},
  MRNUMBER = {22872},
MRREVIEWER = {H.\ Heilbronn},
}

@article {Luo,
    AUTHOR = {Luo, W.},
     TITLE = {Zeros of {H}ecke {$L$}-functions associated with cusp forms},
   JOURNAL = {Acta Arith.},
  FJOURNAL = {Acta Arithmetica},
    VOLUME = {71},
      YEAR = {1995},
    NUMBER = {2},
     PAGES = {139--158},
      ISSN = {0065-1036,1730-6264},
   MRCLASS = {11F66 (11F11 11M41)},
  MRNUMBER = {1339122},
MRREVIEWER = {Kazuyuki\ Hatada},
       DOI = {10.4064/aa-71-2-139-158},
       URL = {https://doi.org/10.4064/aa-71-2-139-158},
}

@article {Deligne,
    AUTHOR = {Deligne, P.},
     TITLE = {La conjecture de {W}eil. {I}},
   JOURNAL = {Inst. Hautes \'Etudes Sci. Publ. Math.},
  FJOURNAL = {Institut des Hautes \'Etudes Scientifiques. Publications
              Math\'ematiques},
    NUMBER = {43},
      YEAR = {1974},
     PAGES = {273--307},
      ISSN = {0073-8301,1618-1913},
   MRCLASS = {14G13},
  MRNUMBER = {340258},
MRREVIEWER = {Nicholas\ M.\ Katz},
       URL = {http://www.numdam.org/item?id=PMIHES_1974__43__273_0},
}

@article {MR4574646,
    AUTHOR = {Beckwith, O. and Liu, D. and Thorner, J. and Zaharescu,
              A.},
     TITLE = {A zero density estimate and fractional imaginary parts of
              zeros for {$\rm{GL}_2$} {$L$}-functions},
   JOURNAL = {Math. Proc. Cambridge Philos. Soc.},
  FJOURNAL = {Mathematical Proceedings of the Cambridge Philosophical
              Society},
    VOLUME = {174},
      YEAR = {2023},
    NUMBER = {3},
     PAGES = {605--630},
      ISSN = {0305-0041,1469-8064},
   MRCLASS = {11F66 (11M41)},
  MRNUMBER = {4574646},
MRREVIEWER = {Guohua\ Chen},
       DOI = {10.1017/S0305004122000445},
       URL = {https://doi.org/10.1017/S0305004122000445},
}

@article {Odlyzko,
    AUTHOR = {Odlyzko, A. M.},
     TITLE = {On the distribution of spacings between zeros of the zeta
              function},
   JOURNAL = {Math. Comp.},
  FJOURNAL = {Mathematics of Computation},
    VOLUME = {48},
      YEAR = {1987},
    NUMBER = {177},
     PAGES = {273--308},
      ISSN = {0025-5718,1088-6842},
   MRCLASS = {11M26 (11-04 11Y35)},
  MRNUMBER = {866115},
MRREVIEWER = {S.\ L.\ Segal},
       DOI = {10.2307/2007890},
       URL = {https://doi.org/10.2307/2007890},
}

@article {MR0020594,
    AUTHOR = {Selberg, A.},
     TITLE = {Contributions to the theory of the {R}iemann zeta-function},
   JOURNAL = {Arch. Math. Naturvid.},
  FJOURNAL = {Archiv for Mathematik og Naturvidenskab},
    VOLUME = {48},
      YEAR = {1946},
    NUMBER = {5},
     PAGES = {89--155},
      ISSN = {0365-4524},
   MRCLASS = {10.0X},
  MRNUMBER = {20594},
MRREVIEWER = {E.\ C.\ Titchmarsh},
}

@article{chandee2022pair,
  title={Pair correlation of zeros of $\Gamma$ 1 (q) $L$-functions},
  author={Chandee, V. and Klinger-Logan, K. and Li, X.},
  journal={Mathematische Zeitschrift},
  volume={302},
  number={1},
  pages={219--258},
  year={2022},
  publisher={Springer}
}

@article {GillespieYe,
    AUTHOR = {Gillespie, T. L. and Ye, Y.},
     TITLE = {The prime number theorem and {H}ypothesis {H} with lower-order
              terms},
   JOURNAL = {J. Number Theory},
  FJOURNAL = {Journal of Number Theory},
    VOLUME = {141},
      YEAR = {2014},
     PAGES = {59--82},
      ISSN = {0022-314X,1096-1658},
   MRCLASS = {11M26 (11F70 11M41)},
  MRNUMBER = {3195388},
MRREVIEWER = {Lejla\ Smajlovi\'c},
       DOI = {10.1016/j.jnt.2014.01.022},
       URL = {https://doi.org/10.1016/j.jnt.2014.01.022},
}

@article {ThornerHumphriesHoheisel,
    AUTHOR = {Humphries, P. and Thorner, J.},
     TITLE = {Towards a {${\rm GL}_n$} variant of the {H}oheisel phenomenon},
   JOURNAL = {Trans. Amer. Math. Soc.},
  FJOURNAL = {Transactions of the American Mathematical Society},
    VOLUME = {375},
      YEAR = {2022},
    NUMBER = {3},
     PAGES = {1801--1824},
      ISSN = {0002-9947,1088-6850},
   MRCLASS = {11F66},
  MRNUMBER = {4378080},
MRREVIEWER = {Liyang\ Yang},
       DOI = {10.1090/tran/8544},
       URL = {https://doi.org/10.1090/tran/8544},
}

@article {kimgl4,
    AUTHOR = {Kim, H. H.},
     TITLE = {A note on {F}ourier coefficients of cusp forms on {${\rm
              GL}_n$}},
   JOURNAL = {Forum Math.},
  FJOURNAL = {Forum Mathematicum},
    VOLUME = {18},
      YEAR = {2006},
    NUMBER = {1},
     PAGES = {115--119},
      ISSN = {0933-7741,1435-5337},
   MRCLASS = {11F30 (11F66 11F70)},
  MRNUMBER = {2206246},
MRREVIEWER = {Ramin\ Takloo-Bighash},
       DOI = {10.1515/FORUM.2006.007},
       URL = {https://doi.org/10.1515/FORUM.2006.007},
}

@article {KimShahidi,
    AUTHOR = {Kim, H. H. and Shahidi, F.},
     TITLE = {Functorial products for {${\rm GL}_2\times{\rm GL}_3$} and the
              symmetric cube for {${\rm GL}_2$}},
      NOTE = {With an appendix by Colin J. Bushnell and Guy Henniart},
   JOURNAL = {Ann. of Math. (2)},
  FJOURNAL = {Annals of Mathematics. Second Series},
    VOLUME = {155},
      YEAR = {2002},
    NUMBER = {3},
     PAGES = {837--893},
      ISSN = {0003-486X,1939-8980},
   MRCLASS = {11F70 (22E55)},
  MRNUMBER = {1923967},
MRREVIEWER = {A.\ Raghuram},
       DOI = {10.2307/3062134},
       URL = {https://doi.org/10.2307/3062134},
}

@article{ozluk1996theq,
  title={On the $q$-Analogue of the Pair Correlation Conjecture},
  author={{\"O}zl{\"u}k, A. E},
  journal={Journal of number theory},
  volume={59},
  number={2},
  pages={319--351},
  year={1996},
  publisher={Elsevier}
}

@book{montgomery2007multiplicative,
  title={Multiplicative number theory I: Classical theory},
  author={Montgomery, H. L. and Vaughan, R. C},
  number={97},
  year={2007},
  publisher={Cambridge university press}
}

@book{landau1909handbuch,
  title={Handbook of the theory of the distribution of prime numbers},
  author={Landau, E.},
  volume={1},
  year={1909},
  publisher={BG Teubner}
}

@inproceedings{luo1999generalized,
  title={On the generalized Ramanujan conjecture for GL(n)},
  author={Luo, W. and Rudnick, Z. and Sarnak, P.},
  booktitle={Proceedings of Symposia in Pure Mathematics},
  volume={66},
  number={2},
  pages={301--310},
  year={1999},
  organization={Providence, RI; American Mathematical Society; 1998}
}

@article{luo1995selberg,
  title={On Selberg’s eigenvalue conjecture},
  author={Luo, W. and Rudnick, Z. and Sarnak, P.},
  journal={Geometries in Interaction: GAFA special issue in honor of Mikhail Gromov},
  pages={387--401},
  year={1995},
  publisher={Springer}
}

@article{murtyExplicitFormulasPair2002,
  title = {Explicit Formulas for the Pair Correlation of Zeros of Functions in the {{Selberg}} Class},
  author = {Murty, M. R. and Zaharescu, A.},
  date = {2002},
  journaltitle = {Forum Mathematicum},
  shortjournal = {Forum Math.},
  volume = {14},
  number = {1},
  pages = {65--83},
  issn = {0933-7741},
  doi = {10.1515/form.2002.006},
  mrnumber = {1880195}
}

@online{Rubcalc,
    author = {M. Rubinstein},
    title = {\emph{{$L$}-function c++ class library and, the command line program lcalc}},
    year = 2014,
    url = {https://github.com/agrawroh/l-calc},
    urldate = {2024-26-02}
}

@incollection{rudnickZerosPrincipalLfunctions1996,
  title = {Zeros of Principal {{L-functions}} and Random Matrix Theory},
  booktitle = {Duke {{Mathematical Journal}}},
  author = {Rudnick, Z. and Sarnak, P.},
  date = {1996},
  volume = {81},
  number = {2},
  pages = {269--322},
  issn = {0012-7094},
  doi = {10.1215/S0012-7094-96-08115-6},
  mrnumber = {1395406}
}

@book {Serre,
    AUTHOR = {Serre, J-P.},
     TITLE = {Abelian {$l$}-adic representations and elliptic curves},
      NOTE = {McGill University lecture notes written with the collaboration
              of W. Kuyk and J. Labute},
 PUBLISHER = {W. A. Benjamin, Inc., New York-Amsterdam},
      YEAR = {1968},
     PAGES = {xvi+177 pp.},
   MRCLASS = {14.49},
  MRNUMBER = {263823},
}

@article {Dyson,
    AUTHOR = {Dyson, F. J.},
     TITLE = {Statistical theory of the energy levels of complex systems.
              {III}},
   JOURNAL = {J. Mathematical Phys.},
  FJOURNAL = {Journal of Mathematical Physics},
    VOLUME = {3},
      YEAR = {1962},
     PAGES = {166--175},
      ISSN = {0022-2488,1089-7658},
   MRCLASS = {81.60},
  MRNUMBER = {143558},
MRREVIEWER = {G.\ K\"all\'en},
       DOI = {10.1063/1.1703775},
       URL = {https://doi.org/10.1063/1.1703775},
}

@article {Heath-Brown2,
    AUTHOR = {Heath-Brown, D. R.},
     TITLE = {The average analytic rank of elliptic curves},
   JOURNAL = {Duke Math. J.},
  FJOURNAL = {Duke Mathematical Journal},
    VOLUME = {122},
      YEAR = {2004},
    NUMBER = {3},
     PAGES = {591--623},
      ISSN = {0012-7094,1547-7398},
   MRCLASS = {11G05 (11G40)},
  MRNUMBER = {2057019},
MRREVIEWER = {Joseph\ H.\ Silverman},
       DOI = {10.1215/S0012-7094-04-12235-3},
       URL = {https://doi.org/10.1215/S0012-7094-04-12235-3},
}

@book {Mehta,
    AUTHOR = {Mehta, M.},
     TITLE = {Random matrices},
    SERIES = {Pure and Applied Mathematics (Amsterdam)},
    VOLUME = {142},
   EDITION = {Third},
 PUBLISHER = {Elsevier/Academic Press, Amsterdam},
      YEAR = {2004},
     PAGES = {xviii+688},
      ISBN = {0-12-088409-7},
   MRCLASS = {82-02 (15-02 15A52 60B99 60K35 82B41)},
  MRNUMBER = {2129906},
}

@article {Young,
    AUTHOR = {Young, M. P.},
     TITLE = {Low-lying zeros of families of elliptic curves},
   JOURNAL = {J. Amer. Math. Soc.},
  FJOURNAL = {Journal of the American Mathematical Society},
    VOLUME = {19},
      YEAR = {2006},
    NUMBER = {1},
     PAGES = {205--250},
      ISSN = {0894-0347,1088-6834},
   MRCLASS = {11G40 (11G05 11L20 11M41)},
  MRNUMBER = {2169047},
MRREVIEWER = {Harald\ A.\ Helfgott},
       DOI = {10.1090/S0894-0347-05-00503-5},
       URL = {https://doi.org/10.1090/S0894-0347-05-00503-5},
}

@article {Keating-Snaith,
    AUTHOR = {Keating, J. P. and Snaith, N. C.},
     TITLE = {Random matrix theory and {$L$}-functions at {$s=1/2$}},
   JOURNAL = {Comm. Math. Phys.},
  FJOURNAL = {Communications in Mathematical Physics},
    VOLUME = {214},
      YEAR = {2000},
    NUMBER = {1},
     PAGES = {91--110},
      ISSN = {0010-3616,1432-0916},
   MRCLASS = {11M26 (15A52 82B41)},
  MRNUMBER = {1794267},
MRREVIEWER = {Ze\'ev\ Rudnick},
       DOI = {10.1007/s002200000262},
       URL = {https://doi.org/10.1007/s002200000262},
}

@incollection {Iwaniec--Sarnak,
    AUTHOR = {Iwaniec, H. and Sarnak, P.},
     TITLE = {Perspectives on the analytic theory of {$L$}-functions},
      NOTE = {GAFA 2000 (Tel Aviv, 1999)},
   JOURNAL = {Geom. Funct. Anal.},
  FJOURNAL = {Geometric and Functional Analysis},
      YEAR = {2000},
     PAGES = {705--741},
      ISSN = {1016-443X,1420-8970},
   MRCLASS = {11M26 (11F67 11F70 11F72 11G40)},
  MRNUMBER = {1826269},
MRREVIEWER = {Henry\ H.\ Kim},
       DOI = {10.1007/978-3-0346-0425-3\_6},
       URL = {https://doi.org/10.1007/978-3-0346-0425-3_6},
}

@article {Iwaniec-Luo-Sarnak,
    AUTHOR = {Iwaniec, H. and Luo, W. and Sarnak, P.},
     TITLE = {Low lying zeros of families of {$L$}-functions},
   JOURNAL = {Inst. Hautes \'Etudes Sci. Publ. Math.},
  FJOURNAL = {Institut des Hautes \'Etudes Scientifiques. Publications
              Math\'ematiques},
    NUMBER = {91},
      YEAR = {2000},
     PAGES = {55--131},
      ISSN = {0073-8301,1618-1913},
   MRCLASS = {11M41 (11F66 11M26)},
  MRNUMBER = {1828743},
MRREVIEWER = {Philippe\ G.\ Michel},
       URL = {http://www.numdam.org/item?id=PMIHES_2000__91__55_0},
}

@book {IK2004,
    AUTHOR = {Iwaniec, H. and Kowalski, E.},
     TITLE = {Analytic number theory},
    SERIES = {American Mathematical Society Colloquium Publications},
    VOLUME = {53},
 PUBLISHER = {American Mathematical Society, Providence, RI},
      YEAR = {2004},
     PAGES = {xii+615},
      ISBN = {0-8218-3633-1},
   MRCLASS = {11-02 (11Fxx 11Lxx 11Mxx 11Nxx)},
  MRNUMBER = {2061214},
MRREVIEWER = {K.\ Soundararajan},
       DOI = {10.1090/coll/053},
       URL = {https://doi.org/10.1090/coll/053},
}

@misc{lmfdb,
  author       = {The {LMFDB Collaboration}},
  title        = {The {L}-functions and modular forms database},
  howpublished = {\url{https://www.lmfdb.org}},
  year         = {2023}
}

@article{katzZeroesZetaFunctions1999,
  title = {Zeroes of Zeta Functions and Symmetry},
  author = {Katz, N. and Sarnak, P.},
  year = {1999},
  journal = {American Mathematical Society. Bulletin. New Series},
  volume = {36},
  number = {1},
  pages = {1--26},
  issn = {0273-0979,1088-9485},
  doi = {10.1090/S0273-0979-99-00766-1},
  mrnumber = {1640151}
}

@book{katzRandomMatricesFrobenius1999,
  title = {Random Matrices, {{Frobenius}} Eigenvalues, and Monodromy},
  author = {Katz, N. and Sarnak, P.},
  year = {1999},
  series = {American {{Mathematical Society Colloquium Publications}}},
  volume = {45},
  publisher = {{American Mathematical Society, Providence, RI}},
  doi = {10.1090/coll/045},
  isbn = {978-0-8218-1017-0},
  mrnumber = {1659828}
}

@online{github,
  author = {D. Basak and C. Castillo and A. Zaharescu},
  title = {Surfaces Associated to Zeros of Automorphic $L$-functions Data Companion.},
  year = {2024},
  publisher = {GitHub},
  note = {Available at: \href{https://cruzcstll99.github.io/Correlation_Surfaces/index.html}{Surfaces Associated to Zeros of Automorphic $L$-functions}}
}

@book {Billingsley,
    AUTHOR = {Billingsley, P.},
     TITLE = {Probability and measure},
    SERIES = {Wiley Series in Probability and Statistics},
   EDITION = {Anniversary},
      NOTE = {With a foreword by Steve Lalley and a brief biography of
              Billingsley by Steve Koppes},
 PUBLISHER = {John Wiley \& Sons, Inc., Hoboken, NJ},
      YEAR = {2012},
     PAGES = {xviii+624},
      ISBN = {978-1-118-12237-2},
   MRCLASS = {60-01 (28-01)},
  MRNUMBER = {2893652},
}

@article {Baluyot2016,
    AUTHOR = {Baluyot, S. A. C.},
     TITLE = {On the pair correlation conjecture and the alternative
              hypothesis},
   JOURNAL = {J. Number Theory},
  FJOURNAL = {Journal of Number Theory},
    VOLUME = {169},
      YEAR = {2016},
     PAGES = {183--226},
      ISSN = {0022-314X,1096-1658},
   MRCLASS = {11M26 (11M06)},
  MRNUMBER = {3531237},
MRREVIEWER = {Micah\ B.\ Milinovich},
       DOI = {10.1016/j.jnt.2016.05.007},
       URL = {https://doi-org.proxy2.library.illinois.edu/10.1016/j.jnt.2016.05.007},
}

@article {BM2018,
    AUTHOR = {Bui, H. M. and Milinovich, M. B.},
     TITLE = {Gaps between zeros of the {R}iemann zeta-function},
   JOURNAL = {Q. J. Math.},
  FJOURNAL = {The Quarterly Journal of Mathematics},
    VOLUME = {69},
      YEAR = {2018},
    NUMBER = {2},
     PAGES = {403--423},
      ISSN = {0033-5606,1464-3847},
   MRCLASS = {11M06},
  MRNUMBER = {3811585},
MRREVIEWER = {Lilu\ Zhao},
       DOI = {10.1093/qmath/hax047},
       URL = {https://doi-org.proxy2.library.illinois.edu/10.1093/qmath/hax047},
}

@article {GS2022,
    AUTHOR = {Goldston, D. A. and Suriajaya, Ade Irma},
     TITLE = {The prime number theorem and pair correlation of zeros of the
              {R}iemann zeta-function},
   JOURNAL = {Res. Number Theory},
  FJOURNAL = {Research in Number Theory},
    VOLUME = {8},
      YEAR = {2022},
    NUMBER = {4},
     PAGES = {Paper No. 71, 10},
      ISSN = {2522-0160,2363-9555},
   MRCLASS = {11M06 (11M26 11N05)},
  MRNUMBER = {4483573},
MRREVIEWER = {Tsz\ Ho\ Chan},
       DOI = {10.1007/s40993-022-00371-4},
       URL = {https://doi-org.proxy2.library.illinois.edu/10.1007/s40993-022-00371-4},
}

@article {BGST2024,
    AUTHOR = {Baluyot, S. A. C. and Goldston, D. A. and
              Suriajaya, A. I. and Turnage-Butterbaugh, C. L.},
     TITLE = {An unconditional {M}ontgomery theorem for pair correlation of
              zeros of the {R}iemann zeta-function},
   JOURNAL = {Acta Arith.},
  FJOURNAL = {Acta Arithmetica},
    VOLUME = {214},
      YEAR = {2024},
     PAGES = {357--376},
      ISSN = {0065-1036,1730-6264},
   MRCLASS = {11M06 (11M26)},
  MRNUMBER = {4772293},
MRREVIEWER = {Alessandro\ Fazzari},
       DOI = {10.4064/aa230612-20-3},
       URL = {https://doi-org.proxy2.library.illinois.edu/10.4064/aa230612-20-3},
}

@misc{jiang2025,
      title={On Hypothesis H of Rudnick and Sarnak}, 
      author={Y. Jiang},
      year={2025},
      eprint={2507.20653},
      archivePrefix={arXiv},
      primaryClass={math.NT},
      url={https://arxiv.org/abs/2507.20653}, 
}

@article {WuYe,
    AUTHOR = {Wu, J. and Ye, Y.},
     TITLE = {Hypothesis {H} and the prime number theorem for automorphic
              representations},
   JOURNAL = {Funct. Approx. Comment. Math.},
  FJOURNAL = {Uniwersytet im. Adama Mickiewicza w Poznaniu. Wydzia\l\
              Matematyki i Informatyki. Functiones et Approximatio
              Commentarii Mathematici},
    VOLUME = {37},
      YEAR = {2007},
     PAGES = {461--471},
      ISSN = {0208-6573,2080-9433},
   MRCLASS = {11F70 (22E55)},
  MRNUMBER = {2364718},
MRREVIEWER = {A.\ Peth\H o},
       DOI = {10.7169/facm/1229619665},
       URL = {https://doi.org/10.7169/facm/1229619665},
}
\end{document}